\documentclass[10pt,a4paper,reqno]{amsart}
\usepackage{geometry}
\usepackage{fullpage} 

\usepackage[english]{babel}
\usepackage[
]{hyperref}
\hypersetup{
	colorlinks,
	linkcolor={red!80!black},
	citecolor={blue!50!black},
	urlcolor={blue!80!black}
}
\usepackage{amssymb,bbm}
\usepackage{mathdots,mathtools}
\usepackage{xcolor,graphicx}
\usepackage[all]{xy}
\usepackage{stmaryrd}
\usepackage{wasysym}
\usepackage{mathbbol} 
\numberwithin{equation}{section}
\allowdisplaybreaks

\makeatletter
\@namedef{subjclassname@2020}{\textup{2020} Mathematics Subject Classification}
\makeatother

\usepackage{subcaption}
\usepackage{easyReview}
\newtheorem{theorem}{Theorem}[section]

\newtheorem{proposition}[theorem]{Proposition}
\newtheorem{lemma}[theorem]{Lemma}
\newtheorem{corollary}[theorem]{Corollary}

\theoremstyle{remark}
\newtheorem{remark}{Remark}[section]
\theoremstyle{definition}
\newtheorem{definition}{Definition}
\newtheorem{example}{Example}
\newcommand\underrel[2]{\mathrel{\mathop{#2}\limits_{#1}}}
\title[Spectrum and ergodicity of a neutral Moran model]{On the spectrum and ergodicity of a neutral multi-allelic Moran model}
\author{\text{Josu\'e Corujo}}
\address[1]{Institut de Recherche Mathématique Avancée, UMR 7501 Université de Strasbourg et CNRS, 7 rue René-Descartes, 67000 Strasbourg, France}

\email{josue.corujorodriguez@math.unistra.fr}

\date{Mars 2023}

\begin{document}

\begin{abstract}
	\noindent The purpose of this paper is to provide a complete description of the eigenvalues of the generator of a neutral multi-type Moran model, and the applications to the study of the speed of convergence to stationarity.
	The Moran model we consider is a non-reversible in general, continuous-time Markov chain with an unknown stationary distribution.
	Specifically, we consider $N$ individuals such that each one of them is of one type among $K$ possible allelic types. 
	The individuals interact in two ways: by an independent irreducible mutation process and by a reproduction process, where a pair of individuals is randomly chosen, one of them dies and the other reproduces. 
	Our main result provides explicit expressions for the eigenvalues of the infinitesimal generator matrix of the Moran process, in terms of the eigenvalues of the jump rate matrix. 
	As consequences of this result, we study the convergence in total variation of the process to stationarity and show a lower bound for the mixing time of the Moran process.
	Furthermore, we study in detail the spectral decomposition of the neutral multi-allelic Moran model with parent independent mutation scheme, which is the unique mutation scheme that makes the neutral Moran process reversible. 
	Under the parent independent mutation, we also prove the existence of a cutoff phenomenon in the chi-square and the total variation distances when initially all the individuals are of the same type and the number of individuals tends to infinity. 
	Additionally, in the absence of reproduction, we prove that the total variation distance to stationarity of the parent independent mutation process when initially all the individuals are of the same type has a Gaussian profile. 
\end{abstract}



\maketitle


\section{Introduction and main results}

This paper is devoted to the study of a continuous-time Markov model of $N$ particles on $K$ sites with interaction, which is known as the neutral multi-allelic Moran model in the population genetics literature \cite{etheridge_coalescent_2009}: 
the $K$ sites correspond to $K$ allelic types in a population of $N$ individuals. 
The state space of the process is the $K$-dimensional $N$-discrete simplex: 
\begin{equation}\label{def:state_spaceEKN}
	\mathcal{E}_{K,N} := \left\{\eta \in \left[ N \right]_0^K: \; |\eta| = N\right\},
\end{equation}
where $\left[ N \right]_0 := \{0,1,\dots, N\}$ and $|\cdot|$ stands for the sum of elements in a vector.
The set $\mathcal{E}_{K,N}$ is a finite set with cardinality $\operatorname{Card}(\mathcal{E}_{K,N}) = \binom{K - 1 + N}{N}$.
The process is in state $\eta \in \mathcal{E}_{K,N}$ if there are $\eta(k) \in \left[ N \right]_0$ individuals with allelic type $k \in \left[ K \right]  := \{1,2, \dots, K\}$.
Consider $Q = (\mu_{i,j})_{i,j= 1}^K$ the infinitesimal rate matrix of an irreducible Markov chain on $\left[ K \right]$, which is called the \emph{mutation matrix} of the Moran process. The infinitesimal generator of the neutral multi-allelic Moran process, denoted $\mathcal{Q}_{N,p}$, acts on a real function $f$ on $\mathcal{E}_{K,N}$ as follows:
\begin{equation}\label{def:generatorQKNp}
	\big(\mathcal{Q}_{N,p} f \big)  (\eta) := \sum_{i,j \in \left[ K \right]} \eta(i) \left( \mu_{i,j} + \frac{p}{N} \, \eta(j)  \right) \left[ f(\eta - \mathbf{e}_i + \mathbf{e}_j) - f(\eta)\right]
,
\end{equation}
for all $\eta \in \mathcal{E}_{K,N}$, where $\mathbf{e}_k$ is the $k$-th canonical vector of $\mathbb{R}^K$ (cf.\ \cite{etheridge_coalescent_2009}). 
In words, $\mathcal{Q}_{N,p}$ drives a process of $N$ individuals, where each individual has one of $K$ possible types of alleles and where the type of the individual changes following two processes: a mutation process where individuals mutate independently of each other and a Moran type reproduction process, where the individuals interact. The $N$ individuals mutate independently from type $i \in \left[ K \right]$ to type $j \in \left[ K \right] \setminus \{i\}$ with rate $\mu_{i,j}$. 
In addition, with uniform rate $p \ge 0$, one of the $N$ individuals is uniformly chosen to be removed from the population and another one, also randomly chosen, is duplicated. 
Note that the transitions of an individual due to a reproduction is not independent of the position of the other individuals. 
As in the original model, introduced by Moran \cite{moran_random_1958}, the same individual removed from the population can be duplicated, in this case the state of the system does not change. 
In the variation of the model where the removed individual cannot be duplicated, the factor $\frac{p}{N}$ in (\ref{def:generatorQKNp}) must be replaced by $\frac{p}{N-1}$. 

Note that $\mathcal{Q}_{N,p}$ can be decomposed as
$\mathcal{Q}_{N,p} = \mathcal{Q}_{N} + \frac{p}{N} \, \mathcal{A}_{N}$,
where $\mathcal{Q}_{N}$ and $\mathcal{A}_{N}$ are also infinitesimal generators of Markov chains acting on every $f \in \mathbb{R}^{\mathcal{E}_{K,N}}$ as follows 
\begin{align}
	\big( \mathcal{Q}_{N} f \big)(\eta) &:= \sum_{i,j \in \left[ K \right]} \eta(i) \mu_{i,j} \left[ f(\eta - \mathbf{e}_i + \mathbf{e}_j ) - f(\eta)\right], \label{def:generatorQKN} \\
	\big( \mathcal{A}_{N} f \big)(\eta) &:= \sum_{i,j \in \left[ K \right]} \eta(i) \eta(j) \left[ f(\eta - \mathbf{e}_i + \mathbf{e}_j ) - f(\eta)\right],\label{def:reproduction}
\end{align}
for every $\eta \in \mathcal{E}_{K,N}$.
The processes driven by $\mathcal{Q}_{N}$ and $\mathcal{A}_{N}$ are called the \emph{mutation process}
and the \emph{reproduction process}, respectively. In words, $\mathcal{Q}_{N}$ models the dynamic of $N$ indistinguishable particles, where each one moves among $K$ sites according to the process generated by the mutation rate matrix $Q$. 
This process is usually called the \emph{compound chain}, cf.\ \cite{zhou_composition_2009}. 
On the other hand, $\mathcal{A}_{N}$ models the dynamic where at uniform rate two individuals are randomly chosen and one of them changes its type to the type of the other one.
This paper is devoted to the study of the spectrum of $\mathcal{Q}_{N}$, $\mathcal{A}_{N}$ and $\mathcal{Q}_{N,p}$, and of the convergence to stationarity of the generated Markov processes.
Before stating our main results in this direction, let us establish some notation.

We recall that if $V_n \in \mathbb{R}^K, \; 1 \leq n \leq N$ are $N$ vectors in $\mathbb{R}^K$, their tensor product is the vector $V_1 \otimes V_2 \otimes \dots \otimes V_N$ defined by
\(
(V_1 \otimes V_2 \otimes \dots \otimes V_N )(k_1,k_2,\dots,k_N) := V_1(k_1) V_2(k_2) \dots V_N(k_N),
\)
for all $1 \leq k_n \leq K$ and $1 \leq n \leq N$. 
The tensor $V_1 \otimes V_2 \otimes \dots \otimes V_N$ can be considered as a function on $\left[ K \right]^N$. 
Actually, throughout this paper we completely identify a real function $f$ on $\left[ K \right]^N$ and the tensor vector $V_f$ such that $V_f(k_1,k_2,\dots,k_N) = f(k_1,k_2,\dots,k_N)$, for all $(k_1,k_2,\dots,k_N) \in \left[ K \right]^N$.

Let us denote by $\sigma$ a permutation on $\left[ N \right]$, i.e.\ an element of the symmetric group $\mathcal{S}_N$. 
Then, the permutation of $f \in \mathbb{R}^{\left[ K \right]^N}$ by $\sigma$, denoted by $\sigma f$, is defined by
\[ 
\sigma f : (k_1,k_2,\dots, k_N) \mapsto f(k_{\sigma (1)}, k_{\sigma (2)}, \dots k_{\sigma (N)}),
\]
for all $(k_1,k_2,\dots, k_N) \in \left[ K \right]^N$.
In particular, for $V_1, V_2, \dots, V_N \in \mathbb{R}^{N}$ we have
\[
\sigma (V_1 \otimes V_2 \otimes \dots \otimes V_N ) = 
V_{\sigma^{-1} (1)} \otimes V_{\sigma^{-1} (2)} \otimes \dots \otimes V_{\sigma^{-1} (N)}.
\]
A real function $f$ on $\left[ K \right]^N$ is symmetric if $f = \sigma f$, for all $\sigma$ in $\mathcal{S}_N$. 
Moreover, every function $f$ on $\left[ K \right]^N$ can be symmetrised by the projector $\operatorname{Sym}$, defined as follows:
\begin{equation} \label{def:sim_operator}
\operatorname{Sym}: \, f \mapsto \overline{f} = \frac{1}{N!} \sum_{\sigma \in \mathcal{S}_N} \sigma f. 
\end{equation}
Symmetric functions on $\left[ K \right]^N$ are highly important in the sequel because of their relation to the functions on $\mathcal{E}_{K,N}$. 
Consider the application $\psi_{K,N} : \mathcal{E}_{K,N}  \rightarrow \left[ K \right]^N$ defined by
\begin{equation} \label{def:psiKN}
\psi_{K,N}: \eta \mapsto (\underbrace{1,1,\dots,1}_{\eta(1)},
\underbrace{2,2,\dots,2}_{\eta(2)}, \dots, \underbrace{K,K,\dots,K}_{\eta(K)}),
\end{equation}
where the number of $k$ in $\overbrace{k,k,\dots,k}^{\eta(k)}$ is set to $0$ if $\eta(k)=0$.
Note that for every symmetric function $f$ on $\left[ K \right]^N$, the function $\tilde{f} := f \circ \psi_{K,N}$ on $\mathcal{E}_{K,N}$ is well defined.
Let $U_0$ be the all-one vector in $\mathbb{R}^K$ and $U_1, U_2, \dots, U_{K-1} \in \mathbb{R}^K$ such that 
\(
\mathcal{U} := \{U_0, U_1, \dots, U_{K-1}\}
\)
is a basis of $\mathbb{R}^K$. Note that this is the type of basis given by the eigenvectors of a diagonalisable rate matrix of dimension $K$. 
For every $\eta \in \mathcal{E}_{K-1,L}$, for $1 \le L \le N$, let us also denote by $U_{\eta} \in \mathbb{R}^{\left[ K \right]^N}$, $V_{\eta} \in \operatorname{Sym} \big( \mathbb{R}^{\left[ K \right]^N} \big)$ and $\tilde{V}_{\eta} \in \mathbb{R}^{\mathcal{E}_{K,N}}$ the vectors:
\begin{align}
U_{\eta} &:= U_{k_1} \otimes U_{k_2} \otimes \dots \otimes U_{k_L}\otimes \underbrace{U_0 \otimes \dots \otimes U_0}_{N-L \text{ times}}, \label{def:Ueta}\\
V_{\eta} &:= \operatorname{Sym} (U_{\eta}), \label{def:Veta}\\
\tilde{V}_{\eta} &:= V_{\eta} \circ \psi_{K,N}, \label{def:Veta_tilde}
\end{align}
where $(k_1, k_2, \dots, k_L) = \psi_{K-1,L}(\eta)$, for every $\eta \in \mathcal{E}_{K-1,L}$ and $L \in \left[ N \right]$. 
In Section \ref{sec:two_model} we analyse the link between the spaces $\operatorname{Sym} \left( \mathbb{R}^{\left[ K \right]^N} \right)$ and $\mathbb{R}^{\mathcal{E}_{K,N} }$, and we clarify the nature of the definitions previously introduced.  
The next theorem explains the connection between the eigenstructures of $Q$ and $\mathcal{Q}_{N}$
\begin{theorem}[Eigenstructure of $\mathcal{Q}_{N}$] \label{thm:eigenvalues_generalisation}
	Assume $K \ge 2$, $N \ge 1$. Let $\mathcal{U} = \{U_0, U_1, \dots, U_{r-1}\}$ be a set of $r$ linearly independent right eigenvectors of $Q$ such that $U_0$ is the all-one vector. Let $\lambda_0 = 0,\lambda_1, \dots, \lambda_{K-1}$ be the $K$ complex roots of the characteristic polynomial of $Q$, counting algebraic multiplicities, such that $Q U_k = \lambda_k U_k$, for $k \in \{0,1, \dots, r-1\}$.
	Consider $\lambda_{\eta}$ defined as follows
	\begin{equation} \label{eq:lameta}
		\lambda_\eta :=  \sum_{k=1}^{K-1} \eta(k) \lambda_k, \text{ for every } \eta \in \bigcup\limits_{L=1}^{N} \mathcal{E}_{K-1,L}.
	\end{equation}
	Then,
	\begin{itemize}
		\item[(a)] The non-zero eigenvalues of $\mathcal{Q}_{N}$ are given by
		$\lambda_\eta$, for all $\eta \in \bigcup\limits_{L=1}^{N} \mathcal{E}_{K-1,L}$.
		\item[(b)] Every function $\tilde{V}_\eta$, as defined in (\ref{def:Veta_tilde}), for any $\eta \in \bigcup\limits_{L=1}^{N} \mathcal{E}_{K-1,L}$ satisfying $\eta({r}) =  \dots = \eta({K-1}) = 0$ is a right eigenfunction of $\mathcal{Q}_{N}$ such that $\mathcal{Q}_{N} \tilde{V}_{\eta} = \lambda_{\eta} \tilde{V}_{\eta}$. 
		\item[(c)] In particular, if $Q$ is diagonalisable, then $\mathcal{Q}_{N}$ is diagonalisable.
	\end{itemize}
\end{theorem}

\begin{remark}[Monotonicity of the spectrum of $\mathcal{Q}_{N}$ in $N$]
	Theorem \ref{thm:mainspec_intro} implies that the spectrum
	of $\mathcal{Q}_{N}$, for fixed values of $K$,
	is an increasing function of $N$ in the sense of the inclusion of sets.
\end{remark}

The proof of Theorem \ref{thm:eigenvalues_generalisation} can be found in Section \ref{sec:spectrum_mutation}. 
It can be seen as a continuous-time generalisation of the results provided by Zhou and Lange \cite{zhou_composition_2009} for the discrete-time analogue of the mutation process driven by $\mathcal{Q}_{N}$. 
We emphasize that our hypotheses do not require the mutation rate matrix $Q$ to be reversible, or even diagonalisable.

The next result deals with the spectrum of $\mathcal{A}_{N}$. 

\begin{theorem}[Spectrum of $\mathcal{A}_{N}$] \label{thm:specA}
	Assume $K \geq 2$ and $N \geq 2$. 
	The eigenvalues of $\mathcal{A}_{N}$ are
	\[
	\begin{array}{rl}
	0 & \text{ with multiplicity } K \text{ and}\\
	- L(L-1) & \text{ with multiplicity } \binom{ K+L-2 }{ L },  \text{ for } 2 \leq L \leq N.
	\end{array}
	\]
	Additionally, the infinitesimal rate matrix $\mathcal{A}_{N}$ is diagonalisable.
	
\end{theorem}

The proof of Theorem \ref{thm:specA} is deferred to Section \ref{subsec:reproduction_process}. 
Theorem \ref{thm:specA} can be seen as a generalisation, for $K \ge 3$, of the results in \cite[\S 4.2.2]{zhou_examples_2008} for the discrete analogous of the reproduction process driven by $\mathcal{A}_{N}$, for $K=2$.

Unlike the case of independent mutation process, the dynamics of the neutral multi-allelic Moran process driven by $\mathcal{Q}_{N,p}$, for $p > 0$, is that of an interacting particle system, which makes the study of its spectrum harder.
Our main result is precisely a characterisation of the eigenvalues of $\mathcal{Q}_{N,p}$ in terms of those of $Q$.

\begin{theorem}[Spectrum of $\mathcal{Q}_{N,p}$]\label{thm:mainspec_intro}
	Assume $K \ge 2$, $N \ge 1$ and $p \in [0,\infty)$. 
	Let us denote by $\lambda_k$, $k \in \left[ K-1 \right]$, the nonzero $K-1$ roots, counting algebraic multiplicities, of the characteristic polynomial of $Q$. For any $\eta \in \bigcup\limits_{L=1}^{N} \mathcal{E}_{K-1,L}$, let us define
	\begin{equation*} 
		\lambda_{\eta,p} :=  \sum_{k=1}^{K-1} \eta(k) \lambda_k - \frac{p}{N} |\eta|(|\eta|-1).
	\end{equation*}
	Then, the eigenvalues of $\mathcal{Q}_{N,p}$, counting algebraic multiplicities, are $0$ and $\lambda_{\eta,p}$, for
	$\eta \in \bigcup\limits_{L=1}^{N} \mathcal{E}_{K-1,L}$.
\end{theorem}

The proof of Theorem \ref{thm:mainspec_intro} is given in Section \ref{subsec:walkfv}. 

\begin{remark}[Relation to the spectrum of the Wright\,--\,Fisher diffusion]
	The eigenstructure of the Wright\,--\,Fisher diffusion is a special case of the eigenstructure in a Lambda\,--\,Fleming\,--\,Viot process studied in \cite{zbMATH06387922}.
	Theorem 5 of \cite{zbMATH06387922}, taking $W = 0$, gives the spectrum of the neutral Wright\,--\,Fisher diffusion, which coincides with the spectrum provided by Theorem \ref{thm:mainspec_intro}. 
	This is not surprising since the Wright\,--\,Fisher diffusion is the limit process for the Moran model (cf.\ \cite[Lemma 2.39]{etheridge_mathematical_2011}).
\end{remark}

\subsection*{Applications to the ergodicity of neutral multi-allelic Moran process}

The relation between the spectral properties of $\mathcal{Q}_{N,p}$ and $Q$ can be used to estimate the speed of convergence to stationarity of the Moran process. 

Let us first recall the \emph{total variation distance}. 
For two probability measures $\nu_1$ and $\nu_2$ defined on the same discrete space $\Omega$, the total variation distance is defined as follows:
\begin{equation*}
	\operatorname{d}^{\mathrm{TV}}(\nu_1, \nu_2) := \sup_{A \subset \Omega} |\nu_1(A) - \nu_2(A)| = \frac{1}{2} \sup_{f: \Omega \rightarrow [-1,1]} \left| \int f \mathrm{d} \nu_1 - \int f \mathrm{d} \nu_2 \right| = \frac{1}{2} \|\nu_1 - \nu_2\|_1,
\end{equation*}
where $\| \cdot \|_1$ denotes the $1$-norm in $\mathbb{R}^{\Omega}$.

The total variation distance to stationarity at time $t$ of an ergodic process driven by a generator $L$ on $\Omega$, with initial distribution $\nu$, is given by
\(
	\operatorname{d}^{\mathrm{TV}}( \nu \, \mathrm{e}^{t L}, \pi ),
\)
where $\nu$ is the initial distribution on $\Omega$ and $\pi$ is the stationary distribution of the process driven by $L$, i.e.\ the unique probability vector such that $\pi L = 0$. 
We are interested in the relationship between the spectrum of an infinitesimal rate matrix and the convergence to stationarity of the Markov process it drives. Let us define the \emph{maximum total variation distance} to stationarity of the process driven by $L$, denoted $\operatorname{D}_L^{\mathrm{TV}}$, as follows:
\begin{equation*} 
	\operatorname{D}^{\mathrm{TV}}_{L}(t) :=  \max_{\nu} \operatorname{d}^{\mathrm{TV}} (\nu \, \mathrm{e}^{t L}, \pi),
\end{equation*}
where the maximum runs over all possible initial distributions on $\Omega$. Using the convexity of $\operatorname{d}^{\mathrm{TV}}$, we can prove that $\operatorname{D}^{\mathrm{TV}}_L(t) = \frac{1}{2} \|\mathrm{e}^{t L} - \Pi\|_{\infty}$, where $\Pi$ stands for the matrix with every row equal to $\pi$, and $\|\cdot\|_{\infty}$ denotes the infinite norm of matrices (cf.\ \cite[Ch.\ 4]{levin_markov_2017}).
Recall that the infinite norm of a $n$-dimensional matrix $(A_{i,j})_{i, j = 1}^n$ is simply the maximum absolute row sum, i.e.\ $\|A\| = \max\limits_{1 \le i \le n} \sum_{j = 1}^n |a_{i,j}|$.

As a consequence of Theorem \ref{thm:mainspec_intro}, the \emph{second largest eigenvalue in modulus} ($\mathrm{SLEM}$) of $\mathrm{e}^{t \mathcal{Q}_{N,p}}$ is equal to that of $\mathrm{e}^{t Q}$, for every $t \ge 0$. 
The $\mathrm{SLEM}$ of the generator of the process is useful to study the asymptotic convergence of the process in total variation. 
Hence, in Section \ref{section:SLEM&ergodicity} we study the ergodicity of the process driven by $\mathcal{Q}_{N,p}$ in total variation using the spectral properties of $Q$. We also analyse several examples of neutral multi-allelic Moran processes with diagonalisable and non-diagonalisable mutation rate matrices.

For a real positive function $f$, we denote by $\Theta(f)$ another real positive function such that 
\(
	C_1 f(t) \le \Theta(f)(t) \le C_2 f(t),
\) 
for two constants $0 < C_1 \le C_2 < \infty$ and for all $t \ge T$, with $T > 0$ large enough.

\begin{corollary}[Asymptotic exponential ergodicity in total variation]\label{corol:ergodicity_gral_TV}
	Let us denote by $\mathrm{e}^{-\rho t}$ the $\mathrm{SLEM}$ of $\mathrm{e}^{t Q}$ and by $s \in \mathbb{N}$ the largest multiplicity in the minimal polynomial of $\mathrm{e}^{t Q}$ of all the eigenvalues with modulus $\mathrm{e}^{-\rho t}$, for some $t \ge 0$. 
	Then,
	\[
		\operatorname{D}^{\mathrm{TV}}_{\mathcal{Q}_{N,p}}( {t} )  = \Theta \big( \operatorname{D}^{\mathrm{TV}}_{Q}( {t} ) \big) =  \Theta \big( t^{s-1} \mathrm{e}^{- \rho t }\big).
	\]
\end{corollary}
Corollary \ref{corol:ergodicity_gral_TV} is proved in Section \ref{section:SLEM&ergodicity}.

Note that the constant in the asymptotic expression in Corollary \ref{corol:ergodicity_gral_TV}  may depend on $N$.
In addition, if $Q$ has a real negative eigenvalue, we can show a lower bound for the mixing time of the process driven by $\mathcal{Q}_{N,p}$.

\begin{theorem}[Lower bound for convergence in total variation]\label{thm:lower_bound}
	Assume $K \ge 2$, $N \ge 2$ and $p \in [0,\infty)$ and let $-\lambda < 0$ be an eigenvalue of $Q$ with associated right-eigenvector $V = [v_1, \dots, v_K]$.
	Let  $\nu_{N,p}$ be the stationary distribution of the process driven by $\mathcal{Q}_{N,p}$ and let us denote 
	\[
		t_{N,c} := \frac{\ln N - c}{2 \lambda} \text{ and }  \kappa := 8(2 + \|Q\|_{\infty}/\lambda).
	\]
	Then,
	\begin{equation*} 
		\operatorname{d}^{\mathrm{TV}}(\delta_{N \mathbf{e}_k} \mathrm{e}^{ t_{N,c} \mathcal{Q}_{N,p} }, \nu_{N,p}) \ge 1 - \kappa \frac{\|V\|_\infty}{|v_k|} \, \mathrm{e}^{-c},
	\end{equation*}
	for all $c \ge 0$ and for any $k \in [K]$ such that $v_k \neq 0$.
	In particular,
	\begin{equation*} 
		\operatorname{D}_{\mathcal{Q}_{N,p}}^{\mathrm{TV}}\left( \frac{\ln N - c}{ 2 \lambda} \right) \ge 1 - \kappa \mathrm{e}^{-c}.
	\end{equation*}

\end{theorem}

The proof of Theorem \ref{thm:lower_bound} is deferred to Section \ref{sec:proof:Thmlowerbound}.

The lower bound provided by Theorem \ref{thm:lower_bound} ensures that the mixing time of the neutral multi-allelic Moran model is at least of order of $ \ln N/ 2 \lambda$.
The characterisation of the mixing times of Markov chains is a hard task in general.
Lower bounds such that in Theorem \ref{thm:lower_bound} can be very difficult to prove, even more taking in account that the Markov chain generated by $\mathcal{ Q}_{N,p}$ is not necessarily reversible.
Our proof relies on spectral arguments and make use of a detailed understanding of the eigenvalues and eigenvectors of $\mathcal{ Q}_{N,p}$.
This result can be seen as an important application of Theorem \ref{thm:mainspec_intro}.

Our results do not allow us to prove an upper bound ensuring the existence of a \emph{cutoff phenomenon}.
A further study needs to be done in this direction.
However, for the parent independent mutation scheme, a deeper analysis can be done to prove the existence of a cutoff phenomenon in both the chi-square and the total variation distances, as we explain next.

\subsection*{Study of the neutral multi-allelic Moran model with parent independent mutation}
Consider the following mutation rate matrix:
 \begin{equation}\label{eq:def:Qrev_matrix}
 	Q_{\pmb{\mu}} :=
 	\left(
 		\begin{array}{ccccc}
 			-|\pmb{\mu}|+\mu_1 & \mu_2 & \mu_3 & \dots & \mu_{K}\\
 			\mu_1 & -|\pmb{\mu}|+\mu_2 & \mu_3 & \dots & \mu_{K}\\
 			\mu_1 &   \mu_2 &  -|\pmb{\mu}|+ \mu_3 & \dots & \mu_{K}\\
 			\vdots & \vdots & \vdots & \ddots & \vdots \\
 			\mu_1 &   \mu_2 &  \mu_3 & \dots & -|\pmb{\mu}|+ \mu_{K}\\			
 		\end{array}
 	\right),
 \end{equation}
 where $\pmb{\mu} = (\mu_1, \mu_2, \dots, \mu_{K}) \in (0, \infty)^K$ and $|\pmb{\mu}|$ stands for the sum of the entries of $\pmb{\mu}$. 
Let us define
\begin{equation*}
	(\mathcal{L}_{N,p} \, f) (\eta) := \sum_{i,j=1}^{K} \eta(i) \left( \mu_j + \frac{p}{N} \eta(j) \right) \left[ f(\eta - \mathbf{e}_i + \mathbf{e}_j) - f(\eta)\right]
,
\end{equation*}
for every $f$ on $\mathcal{E}_{K,N}$ and all $\eta \in \mathcal{E}_{K,N}$, the infinitesimal generator of the neutral multi-allelic Moran process with mutation rate matrix $Q_{\pmb{\mu}}$. 
The process driven by $\mathcal{L}_{N,p}$ is a special case of the neutral multi-allelic Moran process considered before, but with the difference that the mutation rate only depends on the type of the new individual, i.e.\ mutation changes each type $i$ individual to type $j$ at rate $\mu_j$, for all $i,j \in \left[ K \right]$. 
This is the neutral multi-allelic Moran process with \emph{parent independent} mutation (cf.\ \cite{etheridge_mathematical_2011}). 
Note that $  \mathcal{L}_{N,p} = \mathcal{L}_{N} + \frac{p}{N} \mathcal{A}_{N}$, where $\mathcal{L}_{N} := \mathcal{L}_{N, 0}$, satisfies
\begin{equation*}
	\big( \mathcal{L}_{N} f \big) (\eta) := \sum_{i,j=1}^{K} \eta(i) \mu_j \left[ f(\eta - \mathbf{e}_i + \mathbf{e}_j) - f(\eta)\right],
\end{equation*}
for every $f$ on $\mathcal{E}_{K,N}$ and all $\eta \in \mathcal{E}_{K,N}$.

\begin{remark}[Spectrum of $\mathcal{L}_{N,p}$]\label{corol:spectrum_reversible}
Using Theorem \ref{thm:mainspec_intro} we can easily recover the spectrum of $\mathcal{L}_{N,p}$, which is a well-known result in the population genetics literature.
For $K \geq 2$, $N \geq 2$ and $p \geq 0$, the infinitesimal generator $\mathcal{L}_{N,p}$ is diagonalisable with eigenvalues $\lambda_L$ with multiplicity $\binom{ K + L - 2 }{ L }$, where
\begin{equation}\label{eq:eigenvalues_reversible}
	\lambda_{L,p} := - |\pmb{\mu}| L - \frac{p}{N}L(L-1), 
\end{equation}
for $L \in \left[ N \right]_0$. In particular, the spectral gap of $\mathcal{L}_{N,p}$ is $\rho = |\pmb{\mu}|$.
\end{remark}

The \emph{complete graph model} studied by Cloez and Thai \cite{cloez_fleming-viot_2016} in the context of the Fleming\,--\,Viot particle processes is a particular case of the reversible process driven by $Q_{\pmb{\mu}}$ above when
$\mu_j = 1/K$, for all $j \in \left[ K \right]$. 
In this case, the eigenvalues of the mutation rate are $\beta_0 = 0$ and $\beta_1 = -1$, this last one with multiplicity $K-1$. 
In particular, the explicit expression \eqref{eq:eigenvalues_reversible} improves Lemma 2.14 in \cite{cloez_fleming-viot_2016}.

For a real $x$ and $n \in \mathbb{N}_{0}$ and $\eta \in \mathcal{E}_{K,N}$, we denote by $x_{(n)}$, $x_{[n]}$ and $\binom{ N }{ \eta }$ the \emph{increasing factorial coefficient}, the \emph{decreasing factorial coefficient} and the \emph{multinomial coefficient}, defined by
\[
	x_{(n)} := \prod_{k = 0}^{n-1} (x + k), \;\; \;\;x_{[n]} := \prod_{k = 0}^{n-1} (x - k) \;\; \;\; \text{ and } \;\; \;\; \binom{ N }{ \eta } := \frac{N!}{\prod\limits_{j=1}^K \eta(j)!}.
\] 
We set by convention $x_{(0)} := 1$ and $x_{[0]} := 1$, even for $x = 0$. 

The \emph{multinomial distribution} on $\mathcal{E}_{K,N}$ with parameters $N$ and $\mathbf{q} = (q_1, \dots, q_K) \in (0,1)^K$ such that $|\mathbf{q}| = 1$, denoted $\mathcal{M}(\cdot \mid N, \mathbf{q})$, satisfies
\[
	\mathcal{M}(\eta \mid N, \mathbf{q}) = \binom{ N }{ \eta } \prod_{i = 1}^K q_i^{\eta(i)},
\]
for all $\eta \in \mathcal{E}_{K,N}$.
Furthermore, the \emph{Dirichlet multinomial distribution} on $\mathcal{E}_{K,N}$ with parameters $N$ and $\pmb{\alpha} = (\alpha_1, \alpha_2, \dots, \alpha_K) \in (0, \infty)^{K}$, denoted $\mathcal{DM}(\cdot \mid N, \pmb{\alpha})$, satisfies
\[
	\mathcal{DM}(\eta \mid N, \pmb{\alpha}) = \frac{1}{|{\pmb{\alpha}}|_{(N)}} \binom{ N }{ \eta } \prod_{k=1}^{K}  (\alpha_k)_{(\eta(k))},
\] 
for all $\eta \in \mathcal{E}_{K,N}$.
The distribution $\mathcal{DM}(\cdot \mid N, \pmb{\alpha})$ is a mixture of $\mathcal{M}(\cdot \mid N, \mathbf{q})$ by a \emph{Dirichlet distribution}. 
See Mosimann \cite{mosimann_compound_1962} for the original reference to the Dirichlet multinomial distribution and Johnson et al.\ \cite[\S 13.1]{johnson_univariate_2005}, a classical reference on multivariate discrete distributions, for more details.

It is known in the population genetics literature that the process driven by $\mathcal{L}_{N,p}$, for $p > 0$, is reversible with stationary distribution $\mathcal{DM}(\cdot \mid N, N \pmb{\mu}/p)$, see e.g.\ \cite{etheridge_coalescent_2009}. 
Besides,  the stationary distribution of the process driven by $\mathcal{L}_{N}$ is $\mathcal{M}(\cdot \mid N, \pmb{\mu}/|\pmb{\mu}|)$, see e.g.\ \cite{zhou_composition_2009}. 
Let us define the distribution $\nu_{N,p}$ on $\mathcal{E}_{K,N}$, for all $p \ge 0$, as follows
\begin{equation}\label{def:nuKNp}
	\nu_{N,p}(\eta) := \left\{
		\begin{array}{ccc}
		\mathcal{DM}(\eta \mid N, N \pmb{\mu}/p) & \text{ if } & p > 0\\
		\mathcal{M }(\eta \mid \pmb{\mu}/|\pmb{\mu}|) & \text{ if } & p=0,
		\end{array}
	\right.
\end{equation}
for all $\eta \in \mathcal{E}_{K,N}$. Then, $\nu_{N,p}$ is the stationary distribution of $\mathcal{L}_{N,p}$, for all $p \ge 0$. Besides, the stationary distribution is continuous when $p \rightarrow 0$, in the sense that
\[
	\lim\limits_{p \rightarrow 0} \nu_{N,p}(\eta) = \nu_{N,0}(\eta) =: \nu_{N}(\eta) ,
\]
for every $\eta \in \mathcal{E}_{K,N}$.

In their study of the spectral properties of the discrete-time analogous of $\mathcal{Q}_{N}$, Zhou and Lange \cite{zhou_composition_2009} mainly focus on the case where the process driven by $Q$ is reversible, which is proved to be a necessary and sufficient condition for the reversibility of $\mathcal{Q}_{N}$. 
However, the reversibility of $Q$ is not sufficient to ensure the reversibility of the neutral multi-allelic Moran model driven by $\mathcal{Q}_{N,p}$, for $p > 0$, as we discuss in Section \ref{subsec:reversibleMoranModel}. 
In fact, the reversibility of neutral multi-allelic Moran processes is completely characterised via the next result.
\begin{lemma}[Reversible neutral Moran process and parent independent mutation]\label{thm_reversible_distrib}
Assume $K \ge 2$, $N \ge 2$ and $p > 0$. 
The process driven by $\mathcal{Q}_{N,p}$ is reversible if and only if the mutation rate matrix has the form $Q_{\pmb{\mu}}$ as in (\ref{eq:def:Qrev_matrix}), for some vector $\pmb{\mu}$, i.e.\ if and only if $\mathcal{Q}_{N,p}$ can be written as $\mathcal{L}_{N,p}$. 
Furthermore, the stationary distribution of the process driven by $\mathcal{L}_{N,p}$ is $\nu_{N,p}$ as defined by (\ref{def:nuKNp}).
\end{lemma}

The previous result is not surprising because of its analogy with the theory on the measure-valued Fleming\,--\,Viot process studied in \cite{ethier_fleming-viot_1993}. 
Indeed, the measure-valued Fleming\,--\,Viot process is  reversible if and only if its mutation factor is parent independent (see e.g.\ \cite[Thm.\ 8.2]{ethier_fleming-viot_1993} and \cite[Thm.\ 1.1]{li_reversibility_1999}.
Although the \emph{``if part''} in Lemma \ref{thm_reversible_distrib} is well known in the literature on Moran processes, we have not found an explicit statement, or a proof, of this equivalence on the level of prelimit Markov chains considered here.
Thus, for the sake of completeness, we provide a proof of Lemma \ref{thm_reversible_distrib} in Appendix \ref{appendix:poof_reversibility}.

Section \ref{sec:reversible} is devoted to the study of the spectral properties of $\mathcal{L}_{N,p}$, for $p \ge 0$, and its applications to the study of the convergence to stationarity. 
Our results in this section include a complete description of the set of eigenvalues and eigenfunctions of $\mathcal{L}_{N,p}$ and an explicit expression for its transition function.
The eigenfunctions of $\mathcal{L}_{N,p}$, $p > 0$, are explicitly given in terms of \emph{multivariate Hahn polynomials}, which are orthogonal with respect to the compound Dirichlet multinomial distribution (cf.\ \cite{MR0406574,khare_rates_2009}). 
The eigenfunctions of $\mathcal{L}_{N}$, i.e.\ for $p = 0$, are explicitly given in terms of \emph{multivariate Krawtchouk polynomials}, which are orthogonal with respect to the multinomial distribution (cf.\ \cite{MR3258404,MR184284,zhou_composition_2009}).

\begin{remark}[Relation to Ewens' sampling distribution]
	The Moran model, where mutation rates are constant across all types, i.e., $\mu_{i,j} = \theta$ for all $i,j$, has a connection to the Ewens' sampling distribution. 
	This link is explored, for instance, in \cite[\S\ 3.2]{Watterson1984} and \cite[\S\ 4.1]{zhou_examples_2008}. 
	Notably, in this scenario, we can treat mutation types as indistinguishable and ignore individual allelic types, resulting in a Markov process on the space of integer partitions. 
	For example, in a $3$-type Moran process, states $(3,2,2)$, $(2,3,2)$, and $(2,2,3)$ correspond to the same integer partition of $N=7$. 
	It would be worthwhile to explore if the techniques introduced in this paper could be used to compute the spectral elements of the lumped Markov chain in this scenario.
	This research direction falls beyond the scope of this paper, it was suggested by one of the anonymous reviewers and merits further investigation.
\end{remark}

\subsection*{Cutoff phenomenon}
The \emph{cutoff phenomenon} has been a rich topic of research on Markov chains since its introduction by the works of Aldous, Diaconis and Shahshahani in the 1980s (cf.\ \cite{aldous_random_1983,aldous_shuffling_1986, diaconis_generating_1981}). 
A Markov chain presents a cutoff if it exhibits an abrupt transition in its convergence to stationarity.
Some of the most used notions of convergence are, as we consider here, the {total variation} and the {chi-square distances}. 
A good introduction to this subject can be found in the classic book of Levin and Peres \cite[Ch.\ 18]{levin_markov_2017} and in the exhaustive work of Chen, Saloff-Coste et al.\ \cite{chen_cutoff_2006,chen_l2-cutoffs_2017,chen_cutoff_2008,chen_l2-cutoff_2010,MR1490046}.

A typical scenario for the existence of a cutoff is a Markov chain with a high degree of symmetry. 
Hence, the cutoff phenomenon has been deeply studied for the movement on $N$ independent particles on $K$ sites, a model which is usually known as the \emph{product chain}. 
Ycart \cite{ycart_cutoff_1999} studied the cutoff in total variation for $N$ independent particles driven by a diagonalisable rate matrix. 
Later, Barrera et al.\ \cite{barrera_cut-off_2006} and Connor \cite{connor_separation_2010} studied the cutoff on this model according to other notions of distance. 
See also
\cite{chen_l2-cutoffs_2017}, \cite{chen_cutoffs_2018}, \cite{lacoin_product_2015} and \cite[Ch.\ 20]{levin_markov_2017} for more recent studies about the cutoff on product chains. 
The Moran model we consider here preserves the high level of symmetry of the product chain, but the movements of the particles are not independent. Indeed, the particles interact according to a reproduction process that favours the jumps to the sites with greater proportions of individuals.

Before formally defining the cutoff phenomenon, let us recall the \emph{chi-square divergence} (sometimes called ``distance''), which naturally arises in the context of reversible Markov chains.
The chi-square divergence of $\nu_2$ with respect to the target distribution $\nu_1$ is defined by
\[
	\chi^2(\nu_2 \mid \nu_1) := \sum_{\omega \in \Omega} \frac{[\nu_2(\omega) - \nu_1(\omega)]^2}{\nu_1(\omega)} = \| \nu_2 - \nu_1 \|^2_{\frac{1}{\nu_1}},
\]
where $\| \cdot \|_{\frac{1}{\nu_1}}$ stands for the norm in $l^2(\mathbb{R}^{\Omega}, \frac{1}{\nu_1})$, and $\frac{1}{\nu_1} $ is the measure $\omega \mapsto 1/\mu_1(\omega)$.

The chi-square divergence is not a metric, but a measure of the difference between two probability distributions.
Note that the chi-square divergence, as well as the total variation distance, are special cases of the so called $f-$divergence functions, which measure the ``difference'' between two probability distributions \cite{IEEE_chi-square}. In this context, $\chi^2(\mu_2 \mid \mu_1)$ is also known as \emph{Pearson} \emph{chi-square divergence}.

Abusing notation, let us define the functions $\chi^2_{\eta}$ and $\operatorname{d}^{\mathrm{TV}}_{\eta}$, as follows 
\begin{align*}
	\operatorname{d}^{\mathrm{TV}}_\eta(t) &:= \operatorname{d}^{\mathrm{TV}}(\delta_{\eta} \mathrm{e}^{ t \mathcal{L}_{N,p} }, \nu_{N,p}) = \frac{1}{2} \sum_{\xi \in \mathcal{E}_{K,N}} \left| \left(\mathrm{e}^{t \mathcal{L}_{N,p}} \delta_{\xi} \right) (\eta) - \nu_{N,p}(\xi)\right|, \nonumber \\ 
	\chi^2_{\eta}(t) &:= \chi^2( \delta_{\eta} \mathrm{e}^{t \mathcal{L}_{N,p}} \mid \nu_{N,p} ) =  \sum_{\xi \in \mathcal{E}_{K,N}} \frac{ \left[ \left(\mathrm{e}^{t \mathcal{L}_{N,p}} \delta_{\xi} \right) (\eta) - \nu_{N,p}(\xi)\right]^2 }{\nu_{N,p}(\xi)}. 
\end{align*}
The functions $\operatorname{d}^{\mathrm{TV}}_\eta$ and $\chi^2_{\eta}$
are thus measures of the convergence to stationary of the process driven by $\mathcal{L}_{N,p}$ at time $t$ and with initial configuration $\eta \in \mathcal{E}_{K,N}$.
In agreement with \cite{khare_rates_2009,zhou_examples_2008} we call $\chi^2_{\eta}$ and $\operatorname{d}^{\mathrm{TV}}_{\eta}$ the \emph{total variation} and the \emph{chi-square distances} to stationarity, respectively.

As the number of individuals varies we obtain an infinite family of continuous-time finite Markov chains $\{(\mathcal{E}_{K,N}, \mathcal{L}_{N,p}, \nu_{N,p}), N \ge 2\}$. 
For each $N \ge 2$ let us denote by  $\chi^2_{N \mathbf{e}_k}(t)$ (resp.\ 
\(
	\operatorname{d}^{\mathrm{TV}}_{N \mathbf{e}_k}( t )
\))
the chi-square distance (resp.\ total variation distance) to stationarity of the process driven by $\mathcal{L}_{N,p}$ at time $t$, when the initial distribution is concentrated at $N \mathbf{e}_k \in \mathcal{E}_{K,N}$. 
Note that $\chi^2_{N \mathbf{e}_k}(0) \rightarrow \infty$ and $\operatorname{d}^{\mathrm{TV}}_{N \mathbf{e}_k}(0) \rightarrow 1$, when $N \rightarrow \infty$. 

\begin{definition}[Chi-square and total variation cutoff]
	We say that $\{\chi^2_{N \mathbf{e}_k}(t), N \ge 2\}$ exhibits a $(t_N, b_N)$ chi-square cutoff if $t_N \ge 0$, $b_N \ge 0$, $b_N = o(t_N)$ and
\[
	\lim\limits_{c \rightarrow \infty} \limsup\limits_{N \rightarrow \infty} \chi^2_{N \mathbf{e}_k}(t_N + c\, b_N) = 0, \;\; \lim\limits_{c \rightarrow -\infty}  \liminf\limits_{N \rightarrow \infty} \chi^2_{N \mathbf{e}_k} (t_N + c\, b_N) = \infty.
\]
Analogously, we say that $\{ \operatorname{d}^{\mathrm{TV}}_{N \mathbf{e}_k}(t) , N \ge 2\}$ exhibits a $(t_N, b_N)$ total variation cutoff if $t_N \ge 0$, $b_N \ge 0$, $b_N = o(t_N)$ and
\[
	\lim\limits_{c \rightarrow \infty} \limsup\limits_{N \rightarrow \infty} \operatorname{d}^{\mathrm{TV}}_{N \mathbf{e}_k}(t_N + c \, b_N) = 0, \;\; \lim\limits_{c \rightarrow -\infty}  \liminf\limits_{N \rightarrow \infty} \operatorname{d}^{\mathrm{TV}}_{N \mathbf{e}_k}(t_N + c \, b_N) = 1.
\]
The sequences $(t_N)_{N \ge 2}$ and $(b_N)_{N \ge 2}$ are called the \emph{cutoff} and the \emph{window sequences}, respectively. 
\end{definition}

See Definition 2.1 and Remark 2.1 in \cite{chen_cutoff_2008}.

The cutoff phenomenon describes an abrupt transition in the convergence to stationarity: over a negligible period given by the window sequence $(b_N)_{N > 2}$, the distance from equilibrium drops from near its initial value to near zero at a time given by the cutoff sequence $(t_N)_{N \ge 2}$.

A stronger condition for the existence of a $(t_N, b_N)$ chi-square cutoff (resp.\ total variation cutoff)  is the existence of the limit
\[
	G_k(c) := \lim\limits_{N \rightarrow \infty} \chi^2_{N \mathbf{e}_k}(t_N + c \, b_N) \;\;\left(\text{resp. } H_k(c) := \lim\limits_{N \rightarrow \infty} \operatorname{d}^{\mathrm{TV}}_{N \mathbf{e}_k}(t_N + c \, b_N) \right),
\]
for a function $G_k$ (resp. $H_k$), for $k \in \left[ K \right]$, satisfying:
\[
\lim\limits_{c \rightarrow -\infty} G_k(c) = \infty \text{ and } \lim\limits_{c \rightarrow \infty} G_k(c) = 0, \;\;
\left( \text{resp. } \lim\limits_{c \rightarrow -\infty} H_k(c) = 1 \text{ and } \lim\limits_{c \rightarrow \infty} H_k(c) = 0) \right).
\]
Actually, in this case the $(t_N, b_N)$ cutoff is said to be \emph{strongly optimal}, see e.g.\ Definition 2.2 and Proposition 2.2 in \cite{chen_cutoff_2008}. 
See Chapter 2 in \cite{chen_cutoff_2006} and Sections 2.1 and 2.2 of \cite{chen_cutoff_2008} for more details about the definition of $(t_N, b_N)$ cutoff and window optimality.

The next two results establish the existence of cutoff phenomena in the chi-square and the total variation distances for the multi-allelic Moran process driven by $\mathcal{L}_{N,p}$, for $p \ge 0$, when the initial distribution is concentrated at $N \mathbf{e}_k$, for $k \in \left[ K \right]$.
In the chi-square case, we are able to explicitly provide the limit profile of the distance. 
Moreover, we prove that the total variation distance to stationarity of the mutation process driven by $\mathcal{L}_{N}$, i.e.\ for $p=0$, has a Gaussian profile when all the individuals are initially of the same type.

\begin{theorem}[Strongly optimal chi-square cutoff when $N \rightarrow \infty$]\label{thm:cutoff}
	For $k \in \left[ K \right]$, with $K \ge 2$, $p \ge 0$ and every $c \in \mathbb{R}$, we have
	\begin{equation}\label{eq:chi^2cutoff}
		\lim\limits_{N \rightarrow \infty} \chi^2_{N \mathbf{e}_k} \left( t_{N,c} \right) = {\exp \{ K_{k,p} \mathrm{e}^{-c} \} - 1},
	\end{equation}
	where \(t_{N,c} = \displaystyle \frac{\ln N + c}{2 |\pmb{\mu}|}\) and
	\(
		K_{k,p} = \displaystyle \frac{|\pmb{\mu}|(|\pmb{\mu}| - \mu_k)}{\mu_k (|\pmb{\mu}| + p)}.
	\)
	Consequently, the Markov process driven by $\mathcal{L}_{N,p}$ has a strongly optimal $\left( \frac{\ln N}{2|\pmb{\mu}|}, 1 \right)$ chi-square cutoff when $N \rightarrow \infty$.
\end{theorem}

\begin{theorem}[Total variation cutoff when $N \rightarrow \infty$]\label{thm:TVcutoff}
	For every $k \in \left[ K \right]$, with $K \ge 2$, $p \ge 0$ and every $c  > 0$, we have
	\begin{align*}
		 \mathrm{d}^{\mathrm{TV}}_{N \mathbf{e}_k} \left( \frac{\ln N - c}{2 |\pmb{\mu}|} \right) &\ge 1 - 32  \kappa_k \mathrm{e}^{-c},\\
		\lim\limits_{N \rightarrow \infty} \mathrm{d}^{\mathrm{TV}}_{N \mathbf{e}_k} \left( \frac{\ln N + c}{2 |\pmb{\mu}|} \right) &\le \frac{1}{2} \sqrt{{\exp \{ K_{k,p} \mathrm{e}^{-c} \} - 1}},
	\end{align*}
	where $\kappa_k = \max\limits_{r: r \neq k}\displaystyle \frac{\mu_r}{\mu_k} \wedge 1$ and
	\(
		K_{k,p} = \displaystyle \frac{|\pmb{\mu}|(|\pmb{\mu}| - \mu_k)}{\mu_k (|\pmb{\mu}| + p)}.
	\)
	Consequently, the Markov process driven by $\mathcal{L}_{N,p}$ exhibits a $\left( \frac{\ln N}{2|\pmb{\mu}|}, 1 \right)$ total variation cutoff when $N \rightarrow \infty$.

	Moreover, when $p = 0$ and for every $c \in \mathbb{R}$, the limit profile of the total variation distance satisfies
	\begin{equation*}
		\lim\limits_{N \rightarrow \infty} \operatorname{d}^{\mathrm{TV}}_{N \mathbf{e}_k}(t_{N,c}) = 2 \Phi\left( \frac{1}{2} \sqrt { K_{k,0} \mathrm{e}^{-c} } \right) - 1,
	\end{equation*}
	where $\Phi$ is the cumulative distribution function of the standard normal distribution.
	Thus, there exists a strongly optimal $\left( \frac{\ln N}{2|\pmb{\mu}|}, 1 \right)$ total variation cutoff for the process driven by $\mathcal{L}_{N}$ when $N \rightarrow \infty$.
\end{theorem}

Proof of Theorem \ref{thm:cutoff} and \ref{thm:TVcutoff} will be given in Section \ref{subsec:reversibleMoranModel}. 


	

	

During the proof of Theorem \ref{thm:cutoff}, we prove the following result which is of independent interest.
\begin{corollary}[Law of the process driven $\mathcal{L}_{N}$ ($p = 0$)]\label{corol:explicitLaw_t}
	The law of the process driven by $\mathcal{L}_{N}$ at time $t$ when initially all the individuals are of type $k \in \left[ K \right]$ is multinomial $\mathcal{M}\left( \cdot \mid N, \frac{\pmb{\mu}}{|\pmb{\mu}|}(1 - \mathrm{e}^{- |\pmb{\mu}| t}) + \mathrm{e}^{-|\pmb{\mu}| t}\mathbf{e}_k\right)$.
\end{corollary}

Several authors have studied the existence of a cutoff in Moran type models. For instance,
Donnelly and Rodrigues \cite{donnelly_convergence_2000} proved the existence of a cutoff for the two-allelic neutral Moran model in the separation distance.
In order to do that, they used a duality property of the Moran process and found an asymptotic expression for the convergence in separation distance for a suitable scaled time, when the number of individuals tends to infinity. 
Khare and Zhou \cite{khare_rates_2009} proved bounds for the chi-square distance in a discrete-time multi-allelic Moran process that implies the existence of a cutoff. 
Diaconis and Griffiths \cite{MR3915331} studied the existence of a chi-square and total variation cutoffs for a discrete-time analogous of the mutation process generated by $\mathcal{L}_{N}$, i.e.\ when $p = 0$.
See also the study of similar models arried out by Diaconis et al.\ \cite{2008Gibbs} using spectral theory.
Theorems \ref{thm:cutoff} and \ref{thm:TVcutoff} sharpen the results in \cite{khare_rates_2009} and \cite{MR3915331}, since they provide the explicit limit profiles for the chi-square and the total variation distances, for $p \ge 0$ and $p=0$, respectively.

The main difficulty in proving the existence of a total variation cutoff is that even if the Markov chain is reversible, the spectral decomposition of the generator of the Markov chains may only provide a good-enough upper bound for the total variation distance.
In general, one needs to find a sufficiently sharp lower bound using another method. In our case, we could obtain such a bound by using the detailed description of the eigenvectors of $\mathcal{Q}_{N,p}$ of lower modulus, i.e.\ Theorem \ref{thm:lower_bound}. 

Theorem \ref{thm:TVcutoff} is, to the best of our knowledge, the first result ensuring the existence of a total variation cutoff phenomenon for the neutral Moran model with parent independent mutation and neutral reproduction, i.e.\ with $p > 0$.
Notice that the bound given by Theorem \ref{thm:lower_bound} applies to a general class of Moran-type process with neutral reproduction (not necessarily in the parent independent mutation setting).
Hence, this result should facilitate the proof of the existence of total variation cutoff phenomena for other (non-reversible) Markov chains, using other methods in mixing times theory, such as coupling or spectral theory.

\section*{Results for the discrete-time Moran model}

We next briefly discuss some results that can be obtained for an analogous discrete-time multi-allelic Moran model.
Let $(\mu_{i,j})_{i,j \in [K]}$ be a $K$-dimensional stochastic matrix.
We define the mutation and reproduction transition matrices $\pmb{Q}_N$ and $\pmb{A}_N$ as follows
\begin{align*}
	\pmb{Q}_N (\eta, \eta - \mathbf{e}_i + \mathbf{e}_j) = \frac{\eta(i)}{N} \mu_{i,j}, \;\; \text{ and } \;\;
	\pmb{A}_N (\eta, \eta - \mathbf{e}_i + \mathbf{e}_j) = \frac{\eta(i)}{N} \frac{\eta(j)}{N},  
\end{align*}
for every $\eta \in \mathcal{E}_{K,N}$ and $i \neq j \in [K]$.
In addition,
\begin{align*}
	\pmb{Q}_N (\eta, \eta) = 1 - \sum_{i \neq j} \pmb{Q}_N (\eta, \eta - \mathbf{e}_i + \mathbf{e}_j ) , \;\; \text{ and } \;\;
	\pmb{A}_N (\eta, \eta ) = 1 - \sum_{i \neq j} \pmb{A}_N (\eta, \eta - \mathbf{e}_i + \mathbf{e}_j ).
\end{align*}
Besides, every other entry of $\pmb{Q}_N$ and $\pmb{A}_N$ is null.
Then, $\pmb{Q}_N$ and $\pmb{A}_N$ are well-defined transition matrices defining two discrete-time Markov chains, similar to the ones generated by $\mathcal{ Q}_N$ and $\mathcal{A}_N$.
The discrete-time multi-allelic Moran model is the Markov chain with transition matrix
\( \pmb{Q}_{N,p} := p \, \pmb{Q}_N + (1-p) \, \pmb{A}_N\) for some $p \in (0,1]$.
Note that this is the Moran process defined by \cite[\S\ 4.1.1]{khare_rates_2009} when the mutation probabilities of their model take the form $m_{i,j} = (1-p) \mathbb{1}_{\{ i = j\} } + p \, \mu_{i,j}$.

The transition matrices of the discrete-time mutation and reproduction processes and the generator of their continuous-time analogous viewed as matrices, are connected through the identities 
\begin{equation}\label{eq:link_discrete-continuous}
	\pmb{Q}_N = \mathrm{I} + \frac{1}{N} \mathcal{ Q}_N \,\, \text{ and } \,\, \pmb{A}_N =  \mathrm{I} + \frac{1}{N} \mathcal{A}_N,
\end{equation}
where $\mathrm{I}$ denotes the identity matrix.
The factor $1/N$ is necessary in order to ensure that the diagonal elements in $\pmb{Q}_N$ and $\pmb{A}_N$ are positive.

Thus, given an irreducible stochastic matrix $(\mu_{i,j})_{i,j \in [K]}$, the results on the eigenvalues and eigenvectors of $\mathcal{Q}_N$ and $\mathcal{A}_N$ (namely, Theorems \ref{thm:eigenvalues_generalisation}, \ref{thm:specA} and \ref{thm:mainspec_intro}) are easily translated to those of $\pmb{Q}_N$ and $\pmb{A}_N$.
In particular, if $\lambda_0 = 1$ and $\lambda_k$, for $k \in \left[ K-1 \right]$, are the $K$ roots, counting algebraic multiplicities, of the characteristic polynomial of $(\mu_{i,j})_{i,j \in [K]}$. 
For every $\eta \in \bigcup\limits_{L=1}^{N} \mathcal{E}_{K-1,L}$, let us define
\begin{equation*} 
	\lambda_{\eta,p} := 1 -  p \sum_{k=1}^{K-1} \frac{\eta(k)}{N} \lambda_k -(1-p) \frac{|\eta|(|\eta|-1)}{N^2}.
\end{equation*}
Then, the eigenvalues of $\pmb{Q}_{N,p}$, counting algebraic multiplicities, are $1$ and $\lambda_{\eta,p}$, for
$\eta \in \bigcup\limits_{L=1}^{N} \mathcal{E}_{K-1,L}$.

In addition, we have some information on the eigenvectors of $\pmb{Q}_{N,p}$, in terms of the eigenvectors of $(\mu_{i,j})_{i,j \in [K]}$.
Furthermore, when $(\mu_{i,j})_{i,j \in [K]}$ allows a real eigenvalue $\lambda < 1$ we can find an estimate that in the same spirit of Theorem \ref{thm:lower_bound}, provides a lower bound for the total variation distance to stationarity at time $( N \ln (N) - N c )/2$, for every $c \ge 0$.

Adhering strictly to the demonstration of Corollary  \ref{thm_reversible_distrib}, we also get that the discrete-time multi-allelic Moran process is reversible if and only if its mutation matrix is parent independent, i.e.\ if $\mu_{i,j} = \mu_j$.
Besides, the stationary distribution $\nu_{N,p}$ of the Markov chain with transition rates $\pmb{Q}_{N,p}$ satisfies
\begin{equation*}\label{def:nuKNp-discrete}
	\nu_{N,p}(\eta) := \left\{
	\begin{array}{ccc}
		\mathcal{DM}(\eta \mid N, N \pmb{\mu} \, p / (1-p)) & \text{ if } & p \in (0,1)\\
		\mathcal{M }(\eta \mid \pmb{\mu}) & \text{ if } & p = 1,
	\end{array}
	\right.
\end{equation*}
We can also extrapolate the results on the ergodicity for those processes with parent independent mutation rates, obtaining analogous results to Theorems \ref{thm:cutoff} and \ref{thm:TVcutoff}.
In particular, we can ensure that discrete-time Moran model with parent independent mutation exhibits a $\left( {N \ln N}/{2}, N \right)$ total variation cutoff when $N \rightarrow \infty$, when initially all the individuals are of the same type.
The extra factor $N$ in the cutoff time and window is directly linked to the identity \eqref{eq:link_discrete-continuous} and the expression for the eigenvalues of $\pmb{Q}_{N,p}$.
One can check that in the discrete time setting when $p \in (0,1)$ and taking $l_{N,c} = \lceil (N \ln(N) + N c)/2 \rceil$, for $c \in \mathbb{R}$, we get
\begin{equation}\label{eq:profile_chi-square_discrete}
	\lim\limits_{N \rightarrow \infty} \chi^2_{N \mathbf{e}_k} \left( l_{N,c} \right) = {\exp \left\{ (1-p) \frac{1 - \mu_k}{\mu_k} \mathrm{e}^{-c} \right\} - 1},
\end{equation}
and when $p = 0$,
\begin{equation*}
	\lim\limits_{N \rightarrow \infty} \operatorname{d}^{\mathrm{TV}}_{N \mathbf{e}_k}(l_{N,c}) = 2 \Phi\left( \frac{1}{2} \sqrt { \frac{1 - \mu_k}{\mu_k}} \mathrm{e}^{-c/2}  \right) - 1.
\end{equation*}
The explicit expression \eqref{eq:profile_chi-square_discrete} for the chi-square limit profile strengthen Proposition 4.7 in \cite{khare_rates_2009}.

\section*{Discussion and open problems}\label{sec:discussion}



There are several future directions to explore in order to better understand Moran models. 
Despite the fact that it is non-reversible in general, the neutral multi-allelic Moran model with reversible mutation process seems an interesting model for both theoretical and practical reasons (cf.\ \cite{SCHREMPF201788}). 
One possible first step to study the eigenfunctions of $\mathcal{Q}_{N,p}$ when $Q$ is reversible, could be the study of the eigenfunctions of the generator of the reproduction process $\mathcal{A}_{N}$, for $K \ge 3$, extending the results in \cite[\S 4.2.2]{zhou_examples_2008}.

There are several ways to continue the study of the existence of cutoff phenomena for Moran processes. 
For example, using the results of Zhou and Lange \cite{zhou_composition_2009}, it could be possible to prove the existence of a (strongly optimal) chi-square cutoff for the composition chain, when the process driven by the mutation matrix is reversible. 
A possible generalisation of Theorems \ref{thm:cutoff} and \ref{thm:TVcutoff} would be to prove the existence of a cutoff phenomenon for the Moran process with parent independent mutation, when initially all the individuals are not of the same type.


Another interesting problem to address is the study of the spectrum of the multi-allelic Moran process with selection, i.e.\ when the parameter $p$ in  \eqref{def:generatorQKNp} may depend on $i$ and $j$.
Under parent independent settings and selection at birth (cf.\ \cite{durrett_probability_2008}, \cite{etheridge_mathematical_2011} and \cite{muirhead_modeling_2009}) the infinitesimal rate matrix of the process is reversible, but an explicit expression for its spectral gap is unknown. 
The multi-allelic Moran process with selection at death (cf.\ \cite{muirhead_modeling_2009}) seems more complicated from the spectral point of view because it is non-reversible. 
However, this process is very interesting in population genetics but also because of its interpretation as a Fleming\,--\,Viot particle system, which approximates the quasi-stationary distribution of a continuous-time Markov chain (see e.g.\ \cite{asselah_quasistationary_2011}, \cite{cloez_quantitative_2016} and
\cite{ferrari_quasi_2007}). 
We believe the exact results exhibit here for the neutral Moran process will offer clues on the study of the spectrum of the more complicated Moran processes with selection.

\subsection*{Structure of the article} 
The rest of the paper is organised as follows. 
In Section \ref{sec:two_model} we study the state spaces of the neutral multi-allelic Moran models, when the individuals are assumed distinguishable or indistinguishable, respectively. 
We particularly focus on the study of the vector spaces of real functions defined on the state spaces of these two models. The notations and results in Section \ref{sec:two_model} are used to prove our main theorems in Section \ref{sec:eigenvalues_moran}. 
Sections \ref{sec:spectrum_mutation}, \ref{subsec:reproduction_process} and \ref{subsec:walkfv} are devoted to the proofs of Theorems \ref{thm:eigenvalues_generalisation}, \ref{thm:specA} and \ref{thm:mainspec_intro}, respectively. 
In Section \ref{section:SLEM&ergodicity} we focus on the applications of our main results to the asymptotic exponential ergodicity in total variation distance of the process driven by $\mathcal{Q}_{N,p}$ to its stationary distribution, using the eigenstructure of $Q$.
In particular, we prove Corollary \ref{corol:ergodicity_gral_TV} and Theorem \ref{thm:lower_bound}.
In Section \ref{sec:reversible} we consider the neutral multi-allelic Moran process with parent independent mutation and provide a complete description of its eigenvalues and eigenfunctions. 
We also prove Theorems \ref{thm:cutoff} and \ref{thm:TVcutoff} about the existence of a cutoff phenomena in the chi-square and the total variation distances, when initially all the individuals are of the same type. 

\section{State spaces for distinguishable and indistinguishable particle processes}\label{sec:two_model}

The Moran model can be seen as a system of $N$ interacting particles on $K$ sites moving according to a continuous-time Markov chain. 
For the same model, we study two different situations.
Although the sites themselves are supposed to be distinguishable, the $N$ particles can be considered either 
 \emph{distinguishable} or \emph{indistinguishable}. 
According to both interpretations we describe two state spaces for the two Markov chains modelling the $N$ independent particle systems. 
We study how the vector spaces of the real functions defined on those state spaces are related.

For $N$ \emph{distinguishable} particles on $K$ sites, the state space of the model describes the
location of each particle, i.e.\ it is the set $\left[ K \right]^N$. 
This is the state space considered in
\cite{etheridge_mathematical_2011} and \cite{ferrari_quasi_2007}.
The set of real functions on $\left[ K \right]$, denoted $\mathbb{R}^{\left[ K \right]}$, may be endowed
with a vector space structure. Thus, the set of real functions on $\left[ K \right]^N$ may be considered
as  a tensor product of $N$ vectors in  $\mathbb{R}^K$ as we commented in the introduction. 

When the $N$ particles are considered \emph{indistinguishable}, what matters is the number of particles present at each of the $K$ sites.
The state space for this second model, as in \cite{cloez_quantitative_2016} and \cite{etheridge_coalescent_2009}, is the set $\mathcal{E}_{K,N}$ defined by (\ref{def:state_spaceEKN}) 
with cardinality equal to $\operatorname{Card} \left(\mathcal{E}_{K,N}\right) = \binom{ K - 1 + N }{ N }$.

For any $k$, $1 \leq k \leq K$, let us denote by $x_k$ the $k$-th coordinate function defined by
\[
	x_k : \eta=(\eta(1),\eta(2), \dots, \eta(K)) \in \mathcal{E}_{K,N} \mapsto \eta(k) \in \mathbb{R}.
\]
Let us also denote by $\mathbf{x}^{\alpha}$ the monomial on $\mathcal{E}_{K,N}$ defined by
\begin{equation}\label{def:monomialEKN}
	\mathbf{x}^{\alpha} := x_1^{\alpha_1} x_2^{\alpha_2} \dots x_{K}^{\alpha_K},
\end{equation}
where $\alpha \in \mathcal{E}_{K,L}$, for $L \in \left[ N \right]$. 

For $0 \leq L \leq N$,
let us denote by $H_{K,L}$ the vector space of homogeneous polynomial functions of
degree $L$ in variables $x_k, \; 1 \leq k \leq K$ on  $\mathcal{E}_{K,N}$ and the null function.
From the definition of $\mathcal{E}_{K,N}$, it follows that
the function $\sum_{k=1}^K x_k$ is equal to the constant function equal to $N$.
$H_{K,L}$ may be considered as a subspace of $H_{K,L'}$ when
$0 \leq L < L' \leq N$ by identifying
$P(x_1,x_2,\dots,x_K) \in H_{K,L}$ with 
$$
\frac{1}{N^{L'-L}} \left(\sum_{k=1}^K x_k \right)^{L'-L}  P(x_1,x_2,\dots,x_K)
\in H_{K,L'}.
$$

We will say that the \emph{degree of homogeneity} of a homogeneous polynomial
 $P$ is $L$, if $P$ is the sum of monomials 
$\mathbf{x}^\alpha = x_1^{\alpha_1}x_2^{\alpha_2}\dots x_K^{\alpha_K}$ with the same value of
$|\alpha| =L$, and the value $L$ is the smallest as possible. This
corresponds to the fact that there is no factor equal to $x_1+x_2+\dots+x_K$ in the factorisation of $P$.
The {\em total degree} of a polynomial $P$ is the minimum value of $L$
such that $P = P_L + R$ where $P_L$
is homogeneous of degree $L$ and all the monomials in $R$ have a maximum degree strictly less than $L$.
Such an expression for $P$, which is not unique, may be obtained by replacing
$x_K$ by $N-\sum_{k=1}^{K-1} x_k$ in $P$ and adding the monomials
in $P(x_1,\dots,x_{K-1}, N-\sum_{k=1}^{K-1} x_k)$ of maximum total degree to define
$P_L$.

The following result shows that each element in $\mathbb{R}^{\mathcal{E}_{K,N}}$ can be identify to a polynomial in $H_{K,N}$.

\begin{lemma} \label{thm:interp} 
Let $K \geq 2$ and $N \geq 1$. Then
\begin{itemize}
\item[(a)] For any real function $f$ on $\mathcal{E}_{K,N}$ there exists 
a unique homogeneous polynomial $P \in H_{K,N}$ of degree $N$ such that 
$f(\eta) = P(\eta)$, for all $\eta \in \mathcal{E}_{K,N}$.
\item[(b)] The set of monomials of degree $N$
\[
	\mathcal{B}_{H_{K,N}} := \{ \mathbf{x}^{\alpha}, \alpha \in \mathcal{E}_{K,N}  \}
\]
where $\mathbf{x}^{\alpha}$ is defined by (\ref{def:monomialEKN}), is a basis of $\mathbb{R}^{\mathcal{E}_{K,N}}$.
\end{itemize}
\end{lemma}

The proof of Lemma \ref{thm:interp} is mostly technical and is deferred to Appendix \ref{appendix:two_model}.

\begin{remark}[Dimension of $H_{K,N}$]
As a consequence of Lemma \ref{thm:interp}-(b) we have that the dimension of $H_{K,N}$ equals $\binom{K + N - 1}{N}$.
\end{remark}


A natural link between the two state spaces of distinguishable and indistinguishable particles is
\begin{equation} \label{def:phiKN}
\phi_{K,N}: \, (k_1,k_2,\dots, k_N) \in \left[ K \right]^N \mapsto (\eta(1),\eta(2),\dots,\eta(K)) \in \mathcal{E}_{K,N},
\end{equation}
where $ \eta(k) = \operatorname{Card}(\{n ,\; 1 \leq n \leq N,\; k_n = k \})$, for all $k \in \left[ K \right]$. 
The function $\phi_{K,N}$ forgets the identity of the $N$ particles.
Note that $\psi_{K,N}$, defined in (\ref{def:psiKN}), is a right inverse of $\phi_{K,N}$, i.e.\ $\phi_{K,N} \circ \psi_{K,N}  = \operatorname{Id}_{\mathcal{E}_{K,N}}$, where $\operatorname{Id}_{\mathcal{E}_{K,N}}$ stands for the identity function on $\mathcal{E}_{K,N}$.

Let us denote by $\operatorname{Sym}$ the {\em symmetrisation} endomorphism, acting on function $f \in \mathbb{R}^{\left[ K \right]^N}$ as defined by (\ref{def:sim_operator}).
In fact, $\operatorname{Sym}$ is the projection onto the subspace of symmetric functions, denoted $\mathrm{Sym}\big(\mathbb{R}^{\left[ K \right]^N}\big)$.  

Note that $\phi_{K,N}$ is a symmetric function on $\left[ K \right]^N$. 
Furthermore, the identity $\phi_{K,N}(\mathbf{x}) = \phi_{K,N}(\mathbf{y})$ holds if and only if $\mathbf{y}$ is obtained from $\mathbf{x}$ by a permutation of its components.
Hence, if $f$ is symmetric and $\mathbf{x}$ and $\mathbf{y}$ are elements in  $\left[ K \right]^N$ such that $\phi_{K,N}(\mathbf{x}) = \phi_{K,N}(\mathbf{y})$, then $f(\mathbf{x}) = f(\mathbf{y})$.

In general, for every function $f$ on $\left[ K \right]^N$ it is not always possible to define a function $\tilde{f}$ on $\mathcal{E}_{K,N}$ such that  $f = \tilde{f} \circ \phi_{K,N}$ holds. We claim that such a function $\tilde{f}$ exists if and only if $f$ is symmetric.

\begin{lemma}[Link between $\mathbb{R}^{\mathcal{E}_{K,N}}$ and $\operatorname{Sym}(\mathbb{R}^{\left[ K \right]^N})$]\label{lemma:simmetric}
	The linear operator 
	\begin{equation} \label{def:_PhiKN}
		\Phi_{K,N}: f \in \operatorname{Sym}\left(\mathbb{R}^{{\left[ K \right]}^N}\right) \mapsto f \circ \psi_{K,N} \in \mathbb{R}^{\mathcal{E}_{K,N}},
	\end{equation}
	where $\psi_{K,N}$ is defined by (\ref{def:psiKN}), is an isomorphism. In particular, the dimension of the space of symmetric functions on $\left[ K \right]^N$ is 
	\[
		\dim\left(\operatorname{Sym}\left(\mathbb{R}^{{\left[ K \right]}^N}\right)\right) = \binom{K + N - 1}{N}.
	\]

\end{lemma}

\begin{proof}
	Note that $\Phi_{K,N}$ is linear and well defined.
	Moreover, for any function $h$ on ${\mathcal{E}_{K,N}}$, the function $h \circ \phi_{K,N}$ is symmetric on $\left[ K \right]^N$ and satisfies $\Phi_{K,N} \, (h \circ \phi_{K,N}) = h$, proving that $\Phi_{K,N}$ is an isomorphism.
\end{proof}

Lemma \ref{lemma:simmetric} justifies the well definiteness of $\tilde{V}_{\eta}$, defined by (\ref{def:Veta_tilde}), for $\eta \in \bigcup_{L=1}^{N} \mathcal{E}_{K-1, L}$. The relationship between $f$ and $\tilde{f}$ is shown in the following diagram:
\begin{equation*} \label{gr:sym}
\xymatrix{
 \left[ K \right]^N \ar[d]_{{\textstyle \phi_{K,N}}} \ar[rd]^{{\textstyle f}} & \\
 \mathcal{E}_{K,N} \ar[r]_{{\textstyle \tilde{f}}} & \mathbb{R}.
}
\end{equation*} 

We denote by $U_0$ the $K$-dimensional all-one vector, which is always a right eigenvector associated to zero of every $K$-dimensional rate matrix of a continuous-time Markov chain.
Let $K \geq 2$, $N \geq 2$ and $1 \leq L \leq N$ and
let us consider $L$ vectors $V_1, V_2, \dots, V_L$ in $\mathbb{R}^K$, non-proportional to $U_0$,
and $f$ the function equal to the following symmetrised tensor product 
\begin{equation*} \label{def:ff}
f := \operatorname{Sym} ( V_1 \otimes V_2 \otimes \dots \otimes V_L 
\otimes \underbrace{U_0 \otimes \dots \otimes U_0}_{N-L} ) \in \operatorname{Sym}\left(\mathbb{R}^{\left[ K \right]^N}\right).
\end{equation*}
Note that,
\begin{equation} \label{caract:deff}
	f(k_1, k_2, \dots, k_N) = \frac{1}{N!} 
   \sum_{\sigma \in \mathcal{S}_N } V_1(k_{\sigma(1)})V_2(k_{\sigma(2)}) \times \dots \times V_L(k_{\sigma(L)}). 
\end{equation}
We denote by $\mathcal{I}_{L,N}$, for $1 \leq L \leq N$, the set of all injective
applications from $ \left[ L \right] $ to $ \left[ N \right] $. For every $\sigma \in \mathcal{S}_N$, the map $s_\sigma: n \in  \left[ L \right]  \mapsto
\sigma(n) \in \{\sigma(1),\dots,\sigma(L)\}$ is in $\mathcal{I}_{L,N}$ and
$\sigma$ is completely determined by this function $s_{\sigma}$ and a bijective application 
$\beta : (L+1,\dots,N) \rightarrow  \left[ N \right]  \setminus s_{\sigma}( \left[ L \right] )$.
For each $s_{\sigma}$, there are $(N-L)!$ such applications $\beta$. Thus, using (\ref{caract:deff}) we obtain
\begin{align*}
	f(k_1, k_2, \dots, k_N) &= \frac{(N-L)!}{N!} 
   \sum_{s \in \mathcal{I}_{L,N} } V_1(k_{s(1)})V_2(k_{s(2)}) \times \dots \times V_L(k_{s(L)}).
\end{align*}
In order to simplify the calculations we denote by $\xi(V_1,V_2,\dots,V_L)$ the function 
on $\left[ K \right]^N$ defined by
\begin{equation} \label{def:xi}
\xi(V_1,V_2,\dots,V_L): (k_1,k_2,\dots,k_N) \mapsto \sum_{s \in \mathcal{I}_{L,N}}
V_1(k_{s(1)}) V_2(k_{s(2)}) \dots V_L(k_{s(L)}). 
\end{equation}
Note that $\xi(V_1,V_2,\dots,V_L) = \frac{N!}{(N-L)!} f$.
Since $\xi(V_1,V_2,\dots,V_L)$ is symmetric, Lemma \ref{lemma:simmetric} ensures the existence of a unique function $\tilde{\xi}(V_1,V_2,\dots,V_L)$ on $\mathcal{E}_{K,N}$ given by 
\begin{equation}\label{def:xi_tilde}
	\tilde{\xi}(V_1,V_2,\dots,V_L) = \Phi_{K,N} \, {\xi}(V_1,V_2,\dots,V_L).
\end{equation}
The following two identities are thus satisfied:
\begin{equation} \label{carac:propchi}
\xi(V_1,V_2,\dots,V_L) = \tilde{\xi}(V_1,V_2,\dots,V_L) \circ \phi_{K,N}, \;\;\;
\tilde{\xi}(V_1,V_2,\dots,V_L) = \xi(V_1,V_2,\dots,V_L) \circ \psi_{K,N},
\end{equation}
where $\phi_{K,N}$ and $\psi_{K,N}$ are defined in (\ref{def:phiKN}) and (\ref{def:psiKN}), respectively.

The next result provides recursive expressions for the functions $\xi(V_1,\dots,V_L)$ and $\tilde{\xi}(V_1,\dots,V_L)$, for $L \in \left[ N \right]$. Furthermore, we prove that $\tilde{V}_{\eta}$, as defined by (\ref{def:Veta_tilde}), is a polynomial of total degree $|\eta|$, for $\eta \in \bigcup_{L = 1}^N \mathcal{E}_{K-1,L}$.

\begin{lemma} \label{thm:chiV}
The following properties are verified:
\begin{itemize}
\item[(a)] For $L=1$: if $V_1 = [a_1,a_2,\dots,a_K]^T$ is non-proportional to $U_0$,
then $\xi(V_1)$ and $\tilde{\xi}(V_1)$, defined by (\ref{def:xi}) and (\ref{def:xi_tilde}), satisfy
\begin{align} 
\xi(V_1) &: (k_1,k_2,\dots,k_N) \mapsto  \sum_{i=1}^N V_1(k_i), \nonumber \\ 
\tilde{\xi}(V_1) &: (\eta(1),\eta(2),\dots,\eta(K)) \mapsto  \sum_{j=1}^K a_j \eta(j). \label{eq:chi1}
\end{align}
\item[(b)] For any $L,\; 2 \leq L \leq N-1$: if the $L$ vectors 
$V_i = [a_{i,1},a_{i,2},\dots,a_{i,K}]^T$, $1 \leq i \leq L$, are non-proportional to $U_0$, then $\xi(V_1,\dots,V_L)$ and $\tilde{\xi}(V_1,\dots,V_L)$ satisfy
\begin{align} 
\xi(V_1,\dots,V_L) &= \xi(V_1,\dots,V_{L-1}) \xi(V_L) - \sum_{i=1}^{L-1} \xi(V_1,\dots, V_{i-1}, V_i \odot V_L,V_{i+1},\dots, V_{L-1}), \nonumber \\ 
\tilde{\xi}(V_1,\dots,V_L) &= \tilde{\xi}(V_1,\dots,V_{L-1}) \tilde{\xi}(V_L)- \sum_{i=1}^{L-1} \tilde{\xi}(V_1,\dots, V_{i-1},V_{i} \odot V_L,V_{i+1},\dots, V_{L-1}), \nonumber 
\end{align}
where $V_{i} \odot V_L$ stands for the Hadamard (componentwise) product of the vectors $V_i$ and $V_L$.

In particular, when $L=2$ and the two vectors $V_1 = [a_1,a_2,\dots,a_K]^T$ and
$V_2 = [b_1,b_2,\dots,b_K]^T$ are non-proportional to $U_0$,
then $\tilde{\xi}(V_1,V_2)$ is the quadratic polynomial given by
\begin{align}  \label{eq:2degreeVP}
\tilde{\xi}(V_1,V_2) &= \tilde{\xi}(V_1) \tilde{\xi}(V_2) - \tilde{\xi}(V_1 \odot V_2).
\end{align}
\item[(c)] For any $L, \; 1 \leq L \leq N$: if the $L$ vectors
$V_i = [a_{i,1},a_{i,2},\dots,a_{i,K}]^T$, $1 \leq i \leq L$, are non-proportional to $U_0$, then
$\tilde{\xi}(V_1,V_2,\dots,V_L)$ is a polynomial of total degree $L$ satisfying
\begin{equation} \label{eq:prodxi}
\tilde{\xi}(V_1,V_2,\dots,V_L) = \prod_{i=1}^L \tilde{\xi}(V_i) + q,
\end{equation}
where $q$ is a polynomial of total degree strictly less than $L$. 
In particular, $\tilde{V}_{\eta}$, as defined by (\ref{def:Veta_tilde}), is a polynomial of total degree $|\eta|$, for $\eta \in \bigcup_{L = 0}^N \mathcal{E}_{K-1,L}$.
\end{itemize}
\end{lemma}

The proof of Lemma \ref{thm:chiV} can be found in Appendix \ref{appendix:two_model}.

The following result helps us to construct from a given basis of $\mathbb{R}^K$, the three bases for the three vector spaces $\mathbb{R}^{\left[ K \right]^N}$, $\operatorname{Sym}(\mathbb{R}^{\left[ K \right]^N})$ and $\mathbb{R}^{\mathcal{E}_{K,N}}$, respectively.

\begin{proposition}\label{propo:basis}
	Let $U_0$ be the all-one vector in $\mathbb{R}^K$ and $U_1, U_2, \dots, U_{K-1} \in \mathbb{R}^K$ such that 
	\begin{equation*} 
		\mathcal{U} = \{U_0, U_1, \dots, U_{K-1}\}
	\end{equation*}
	is a basis of $\mathbb{R}^K$. The following statements are true:
	\begin{itemize}
		\item[a)] $\mathcal{U}^N$, defined as
		\(
			\mathcal{U}^N := \{ W_1 \otimes W_2 \otimes \dots \otimes W_N, \text{ where } W_i \in \mathcal{U}, \text{ for } i \in \left[ N \right] \}
		\)
		is a basis of $\mathbb{R}^{\left[ K \right]^N}$.
		\item[b)] $\mathcal{S}^N$, defined as
		\[
			\mathcal{S}^N := \{ \underbrace{U_0 \otimes \dots \otimes U_0}_{N \text{ times}} \} \cup \bigcup_{L = 1}^{N} \{ V_{\eta}, \eta \in \mathcal{E}_{K-1,L},\}
		\]
		where $V_{\eta}$ is defined by (\ref{def:Veta}), is a basis of $\operatorname{Sym}\left(\mathbb{R}^{\left[ K \right]^N}\right)$.
		\item[c)] $\tilde{\mathcal{S}}^N$, defined as
		\begin{equation}\label{def:basisEKN}
			\tilde{\mathcal{S}}^N := \{ \underbrace{U_0 \otimes \dots \otimes U_0}_{K \text{ times}} \} \cup \bigcup_{L = 1}^{N} \{ \tilde{V}_{\eta}, \eta \in \mathcal{E}_{K-1,L}\},
		\end{equation}
		where $\tilde{V}_{\eta}$ is defined by (\ref{def:Veta_tilde}), is a basis of $\mathbb{R}^{\mathcal{E}_{K,N}}$.
	\end{itemize}
\end{proposition}

The proof of Proposition \ref{propo:basis} is deferred to Appendix \ref{appendix:two_model}.

\section{Spectrum of the neutral multi-allelic Moran process} \label{sec:eigenvalues_moran}

The main goal of this section is to prove Theorem \ref{thm:mainspec_intro}. In Section \ref{sec:spectrum_mutation} we prove Theorem \ref{thm:eigenvalues_generalisation} describing the set of eigenvalues of the composition chain $\mathcal{Q}_{N}$ in terms of the eigenvalues of $Q$. Moreover, we construct right eigenvectors of $\mathcal{Q}_{N}$ using the symmetrised tensor product of right eigenvectors of $Q$. Later, in Section \ref{subsec:reproduction_process} we prove Theorem \ref{thm:specA}. Using the results in these two sections we prove Theorem \ref{thm:mainspec_intro} in Section \ref{subsec:walkfv}.

\subsection{Proof of Theorem \ref{thm:eigenvalues_generalisation}} \label{sec:spectrum_mutation}

As we commented in Section \ref{sec:two_model}, the $N$ particles in the neutral multi-allelic Moran type process can be considered distinguishable or indistinguishable. 
Throughout the paper we suppose that $Q$ is irreducible. Thus, $0$ is a simple eigenvalue of $Q$
with eigenvector $U_0$.
The generator for the distinguishable case, denoted by $\mathcal{D}_{N}$, acts on a real function $f$ on $\left[ K \right]^N$ as follows
\[
	\big(\mathcal{D}_{N} f\big)(k_1,k_2,\dots,k_N) := \sum_{i=1}^N \sum_{k=1}^{K} \mu_{k_i,k} 
	\big[f(k_1,\dots k_{i-1},k,k_{i+1},\dots,k_N) - f(k_1,\dots,k_N)\big],
\]
for all $(k_1,k_2,\dots,k_N) \in \left[ K \right]^N$.
If the function is given in a  tensor product form, we get
\begin{equation} \label{eq:Ltens}
\mathcal{D}_{N} \big( V_1 \otimes V_2 \otimes \dots \otimes V_N \big)
= \sum_{n = 1}^N V_1 \otimes V_2 \otimes \dots \otimes Q \, V_n \otimes \dots \otimes V_N,
\end{equation}
where $Q V_n(k) := \sum\limits_{r = 1}^K \mu_{k,r} V_n(r) =  \sum\limits_{r = 1}^K \mu_{k,r} (V_n(r) - V_n(k))$, for all $k \in \left[ K \right]$.

\begin{remark}[$\mathcal{D}_{N}$ as a Kronecker sum]
	In fact, the infinitesimal generator satisfies $\mathcal{D}_{N} = Q \oplus Q \oplus \dots \oplus Q$, where $\oplus$ denotes the Kronecker sum. The well-known relationship between the exponential of a Kronecker sum and the Kronecker product of exponential matrices, namely:
	\[
		\exp\{Q \oplus Q \oplus \dots \oplus Q\} = \exp\{Q\} \otimes \exp\{Q\} \otimes \dots \otimes \exp\{Q\},
	\]
	makes clearer the idea that $\mathcal{D}_{N}$ is the infinitesimal generator of the system of $N$ particles moving independently according to the infinitesimal generator $Q$. See \cite[Ch.\ XIV]{pease_methods_1965} and \cite[\S 2.2]{davis_circulant_1979} for further details on the Kronecker sum.
\end{remark}

The Markov chain generated by $\mathcal{D}_{N}$ is usually called \emph{product chain}.
The infinitesimal generator $\mathcal{D}_{N}$ inherits its spectral properties from those of $Q$.  Namely, if $\pi$ is the stationary distribution of $Q$, then $\pi \otimes \pi \otimes \dots \otimes \pi$ is the stationary distribution of $\mathcal{D}_{N}$. 
Moreover, if $V_1, V_2, \dots, V_N$ are $N$ (not necessarily distinct) eigenvectors of $Q$, then $V_1 \otimes V_2 \otimes \dots \otimes V_N$ is an eigenvector of $\mathcal{D}_{N}$. 
Consequently, if $Q$ is diagonalisable, then $\mathcal{D}_{N}$ is also diagonalisable and the tensors products of vectors in an eigen-basis of $Q$ form an eigen-basis of $\mathcal{D}_{N}$, as in Proposition \ref{propo:basis}-(a). 
In particular, if $\lambda_0 = 0, \lambda_1,  \dots, \lambda_{K-1}$ are the $K$ complex eigenvalues of $Q$, then the eigenvalues of $\mathcal{D}_{N}$ are given by the sums of eigenvalues of $Q$, i.e.\ the spectrum of $\mathcal{D}_{N}$ is
\[
	\big\{ 
	z_0 + z_1 + \dots + z_{K-1}: z_i \in \{\lambda_0, \lambda_1 , \dots, \lambda_{K-1}\} \big\}.
\]
See Sections 12.4 and 20.4 in \cite{levin_markov_2017} for the proofs of these results and more details on product chains.

When the $N$ particles are considered indistinguishable, the infinitesimal generator of the Markov chain, denoted by $\mathcal{Q}_{N}$, is that defined by (\ref{def:generatorQKN}), i.e.\
\begin{equation*}
	\big( \mathcal{Q}_{N} f \big)(\eta) = \sum_{i,j \in \left[ K \right]} \eta(i) \mu_{i,j} \left[ f(\eta - \mathbf{e}_i + \mathbf{e}_j) - f(\eta)\right],
\end{equation*}
for all $\eta \in \mathcal{E}_{K,N}$ and for every function $f$ on $\mathcal{E}_{K,N}$. 
Zhou and Lange \cite{zhou_composition_2009} noticed that $\mathcal{Q}_{N}$ is a lumped chain of $\mathcal{D}_{N}$ and used this fact to study the relationship between the spectral properties of both chains. 
They studied the eigenvalues and the left eigenfunctions of $\mathcal{Q}_{N}$. 
In particular, they proved that the stationary distribution of $\mathcal{Q}_{N}$ is multinomial with probability vector $\pi$, denoted $\mathcal{M}(\cdot \mid N, \pi)$, where $\pi$ is the unique stationary probability of $Q$. 
Our approach differs from that on \cite{zhou_composition_2009}: we study the right eigenfunctions of $\mathcal{Q}_{N}$ using the connections between the real functions on $\mathcal{E}_{K,N}$ and the symmetric real functions on $\left[ K \right]^N$ studied in Section \ref{sec:two_model}. 
In addition, our methods allow us to explicitly describe the spectrum of $\mathcal{Q}_{N}$, for every mutation matrix $Q$ generating an irreducible process, even when $Q$ is non-diagonalisable. 
We first study the relationship between the generators $\mathcal{Q}_{N}$ and $\mathcal{D}_{N}$ through the operator $\Phi_{K,N}$.
\begin{lemma}[Link between the generators $\mathcal{Q}_N$ and $\mathcal{D}_N$]\label{lemma:relation_generator}
	For any symmetric function $\xi$ on $\left[ K \right]^N$, the function $\mathcal{D}_{N} \, \xi$ is also symmetric. In addition,
	\begin{equation*} 
		\mathcal{Q}_{N} \, (\Phi_{K,N} \, \xi) = \Phi_{K,N} \, (\mathcal{D}_{N} \, \xi),
	\end{equation*}
	where $\Phi_{K,N}$ is defined by (\ref{def:_PhiKN}).
\end{lemma}

\begin{proof}
	The symmetry of $\mathcal{D}_{N} \, \xi$ is a consequence of the symmetry of $\xi$ and the linearity of $\mathcal{D}_{N}$.

	For $\eta \in \mathcal{E}_{K,N}$ let us define $(k_1, k_2, \dots, k_N) = \psi_{K,N}(\eta)$, i.e.\ $k_i$ is the position on $\left[ K \right]$ of the $i$-th particle according to the definition of $\psi_{K,N}$. We have
	\begin{align*}
		\big( \mathcal{D}_{N} \, {\xi} \circ \psi_{K,N} \big) (\eta)  &= \sum_{i = 1}^N \sum_{k = 1}^K \mu_{k_i, k}  \Big[ \xi(k_1, \dots, k_{i-1}, k, k_{i+1}, \dots, k_N) - \xi \big( \psi_{K,N} (\eta) \big) \Big]\\
			&= \sum_{k = 1}^K \sum_{r = 1}^K \sum\limits_{i: \, k_i = r} \mu_{k_i,k} \Big[ \xi\big(k_1, \dots, k_{i-1}, k, k_{i+1}, \dots, k_N) - \xi(\psi_{K,N} (\eta)\big) \Big].
	\end{align*}
	Using the symmetry of $\xi$, for all $\eta$ such that $\psi_{K,N}(\eta)(i) = r$ we obtain
	\[
		\xi(k_1, \dots, k_{i-1}, k, k_{i+1}, \dots, k_N) - \xi(\psi_{K,N} (\eta)) = \xi(\psi_{K,N} (\eta - \mathbf{e}_r + \mathbf{e}_k)) - \xi(\psi_{K,N} (\eta)).
	\]
	Thus,
	\begin{align*}
		\big(\mathcal{D}_{N} \, {\xi} \circ \psi_{K,N} \big) (\eta)  &= \sum_{k = 1}^K \sum_{r = 1}^K \sum\limits_{i: \, k_i = r} \mu_{k_i,k} \Big[ \xi\big (\psi_{K,N} (\eta - \mathbf{e}_r + \mathbf{e}_k) \big) - \xi\big(\psi_{K,N} (\eta)\big) \Big] \\
			&= \sum_{k = 1}^K \sum_{r = 1}^K \eta(r) \mu_{r,k} \Big[ \xi(\psi_{K,N} \big( \eta - \mathbf{e}_r + \mathbf{e}_k ) \big) - \xi\big( \psi_{K,N} (\eta) \big) \Big]\\
			&= \big( \mathcal{Q}_{N} \, \xi \circ \psi_{K,N} \big) (\eta),
	\end{align*}
	for every $\eta \in \mathcal{E}_{K,N}$.
\end{proof}

The following lemma describes all the eigenvalues of $\mathcal{Q}_{N}$, defined by (\ref{def:generatorQKN}), in the case where the mutation matrix is diagonalisable.

\begin{lemma}[Eigenvalues of $\mathcal{Q}_{N}$ for diagonalisable $Q$] \label{lemma:eigenvaluesQKN}
Assume $Q$ is diagonalisable and $$\mathcal{U} = \{U_0, U_1, \dots, U_{K-1}\}$$ is the basis of $\mathbb{R}^K$ formed by right eigenvectors of $Q$, such that $U_0$ is the all-one vector. Consider $\tilde{V}_{\eta}$ and $\lambda_\eta$ defined as in (\ref{def:Veta_tilde}) and (\ref{eq:lameta}), respectively. Then
\begin{itemize}
\item[(a)]
$\lambda_\eta$ is an eigenvalue of $\mathcal{Q}_{N}$ with right eigenvector $\tilde{V}_\eta$.
\item[(b)]
The spectrum of $\mathcal{Q}_{N}$ is formed by $0$ and all $\lambda_\eta$ for
$\eta \in \bigcup\limits_{L=1}^{N} \mathcal{E}_{K-1,L}$. 
\item[(c)] $\mathcal{Q}_{N}$ is diagonalisable.
\end{itemize}
\end{lemma}

\begin{proof}
(a) For $\eta \in \mathcal{E}_{K-1,L}$ let us denote $U_{\eta}$ as in (\ref{def:Ueta}). Because $Q \, U_0 = 0$ and $Q \, U_k = \lambda_k \, U_k, \; 1 \leq k \leq K-1$, from (\ref{eq:Ltens}), we get
\(
\mathcal{D}_{N} ( U_{\eta} ) = \lambda_{\eta} U_{\eta}.
\)
More generally, for every permutation $\sigma \in \mathcal{S}_N$,
\(
\mathcal{D}_{N} \left(\sigma U_{\eta} \right) = \lambda_{\eta} (\sigma U_{\eta}),
\)
and thus, using the linearity of $\mathcal{D}_{N}$ we get
\begin{equation*} \label{eq:LKN1}
\mathcal{D}_{N} \, V_{\eta}  = \lambda_{\eta}
 V_{\eta},
\end{equation*}
where $V_{\eta}$ is defined as in (\ref{def:Veta}).
Applying $\psi_{K,N}$ to both members of the previous equality we obtain
\(
	(\mathcal{D}_{N} V_{\eta}) \circ \psi_{K,N} = \lambda_{\eta} V_{\eta} \circ \psi_{K,N}.
\)
Now, using Lemma (\ref{lemma:relation_generator}), and the expressions (\ref{def:Veta}) and (\ref{def:Veta_tilde}), definitions of
$V_{\eta}$ and $\tilde{V}_{\eta}$, respectively, we obtain
\(
\mathcal{Q}_{N} \tilde{V}_{\eta} = \lambda_\eta \tilde{V}_{\eta},
\)
which proves (a).

(b)-(c) Because $\mathcal{U}$ is a basis of $\mathbb{R}^K$,
the set $\tilde{\mathcal{S}}^N$ as defined in (\ref{def:basisEKN}) is a basis
of $\mathbb{R}^{\mathcal{E}_{K,N}}$, due to Proposition \ref{propo:basis}-(c). Therefore, all the eigenvalues of $\mathcal{Q}_{N}$ are those described in part (b) and $\mathcal{Q}_{N}$ is diagonalisable.
\end{proof}

\begin{remark}\label{remark:QNonRateMatrix}
	Note that the results in Lemma \ref{lemma:eigenvaluesQKN} remains valid for all operator $\mathcal{Q}_N$ defined using a diagonalisable matrix $Q$, not necessarily a rate matrix, with complex entries and such that $Q U_0 = \mathbf{0}$ and $\lambda_0 = 0$ has algebraic multiplicity equal to one.
\end{remark}

Lemma \ref{lemma:eigenvaluesQKN} provides all the eigenvalues and right eigenvectors of $\mathcal{Q}_{N}$ when $Q$ is diagonalisable. 
However, an ergodic rate matrix is not necessarily diagonalisable. 
%
Thereby, we want to extend the results in Lemma \ref{lemma:eigenvaluesQKN} to the case where the matrix $Q$ is non-diagonalisable, as stated in Theorem \ref{thm:eigenvalues_generalisation}.
Let us first recall two known facts in the theory of real matrices. 
We denote by $M_n(\mathbb{R})$ and $M_n(\mathbb{C})$ the vector space of $n$-dimensional real and complex matrices, respectively. For a matrix $M \in M_n(\mathbb{C})$ we denote by $\operatorname{Spec}(M) \in \mathbb{C}^n$ its spectrum counting the algebraic multiplicities of the eigenvalues.
It is known that the set of diagonalisable complex matrices is dense in $M_n(\mathbb{C})$. Serre \cite[Cor.\ 5.1]{MR2744852}, for instance, proves this result as a consequence of Schur's Theorem \cite[Thm.\ 5.1]{MR2744852}. 
Using the same reasoning we can prove the following: 
\begin{description}
	\item[Fact 1] The set of diagonalisable complex matrices with each row summing to zero is dense in the set of the irreducible rate matrices: for every rate matrix $Q \in M_n(\mathbb{R})$ and $\epsilon > 0$ there exists a diagonalisable matrix $\bar{Q} \in M_n(\mathbb{C})$ such that $\|Q - \bar{Q}\| < \epsilon$. 
	Moreover, $\bar{Q}$ can be chosen such that $0 \in \operatorname{Spec}(\bar{Q})$, with $0$ having geometric multiplicity $1$ and $\bar{Q} \, U_0 = \mathbf{0}$, where $\mathbf{0}$ denotes the $K$ dimensional null column vector, i.e.\ each row of $\bar{Q}$ sums to zero.
\end{description}

The idea of the proof of Fact 1 is to modify diagonal elements in the upper-triangular matrix obtained by Schur's Theorem \cite[Thm.\ 5.1]{MR2744852} to get a matrix with $n$ different eigenvalues, and thus diagonalisable. 
Indeed, since $Q$ is an irreducible rate matrix, the eigenspace associated to the eigenvalue $\lambda_0 = 0$ has dimension one and it is generated by $U_0$. 
Moreover, the other $n-1$ complex eigenvalues have strictly negative real parts. 
Thus, it is possible to modify the diagonal of the upper triangular matrix obtained by the Schur's Theorem in such a way that the eigenvalues of the modified matrix, denoted $\bar{Q}$, are zero and $n-1$ complex numbers with different and strictly negative real parts.
Furthermore, because of the Schur's factorisation, $U_0$ is also an eigenvector of $\bar{Q}$ associated to the null eigenvalue, i.e.\ $\bar{Q} U_0 = \mathbf{0}$.

Note that, since $M_n(\mathbb{C})$ is a finite dimensional vector space, the result in Fact 1 holds for every norm defined on $M_n(\mathbb{C})$. In the sequel we will use the uniform norm, denoted $\| \cdot \|_{\text{Unif}}$, and defined as follows 
\[
	\| A \|_{\text{Unif}} := \max\limits_{i,j} |a_{i,j}|,
\]
for every matrix $A = (a_{i,j})_{i,j} \in M_n(\mathbb{C})$.

The second fact is related to the continuity of the eigenvalues of a matrix with respect to its entries.  Consider the following distance between two sets of $n$ elements in $\mathbb{C}$:
\begin{equation*} 
	\operatorname{D}\left( \{z_i\}_{i = 1}^n, \{\omega_i\}_{i = 1}^n \right) := \inf\limits_{\sigma \in \mathcal{S}_n} \max_{j} |z_j - \omega_{\sigma(j)}|,
\end{equation*}
where $\mathcal{S}_{n}$ denotes de symmetric group on $[n]$, for every $n \in \mathbb{N}$.

\begin{description}
	\item[Fact 2] The eigenvalues are continuous with respect to the entries of the matrix in the following sense: consider $M \in M_n(\mathbb{C})$, then for all $\epsilon > 0$ there exists a $\delta > 0$ such that for every matrix $N \in M_n(\mathbb{C})$ such that $\|M - N\| < \delta$, then $\operatorname{D} \left( \operatorname{Spec}(M), \operatorname{Spec}(N)  \right) < \epsilon$.
\end{description}
See e.g.\ \cite{MR884486} and \cite[Thm.\ 5.2]{MR2744852} for a proof of Fact 2.

\begin{proof}[Proof of Theorem \ref{thm:eigenvalues_generalisation}]
From Lemma \ref{lemma:eigenvaluesQKN} we know that the statement of Theorem \ref{thm:eigenvalues_generalisation} holds for a diagonalisable rate matrix $Q$. Let us prove it in the general case using the Facts 1 and 2 we previously discussed.

For a mutation rate matrix $Q \in M_K(\mathbb{R})$ with spectrum $\operatorname{Spec}(Q) = \{0, \lambda_1, \dots, \lambda_{K-1}\}$, let us define by $\sigma_{N}(Q)$ the set formed by $0$ and $\lambda_{\eta}$, for $\eta \in \bigcup_{L = 1}^{K-1} \mathcal{E}_{K-1,L}$, where the values $\lambda_k$ in the definition (\ref{eq:lameta}) of $\lambda_{\eta}$ are those in $\operatorname{Spec}(Q)$. Then, proving Theorem \ref{thm:eigenvalues_generalisation}-(a) is equivalent to prove that $\sigma_{N}(Q)$ is the spectrum of $\mathcal{Q}_{N}$, i.e.\ $\operatorname{D} \left( \operatorname{Spec}(\mathcal{Q}_{N}), \sigma_{N} (Q) \right) = 0$. 

For a matrix $\bar{Q} \in M_K(\mathbb{C})$ whose rows sum to zero (not necessarily a rate matrix), let us define $ \bar{ \mathcal{ Q} }_{N}$ similarly to the definition of $\mathcal{Q}_{N}$ (\ref{def:generatorQKN}), but with $\bar{Q}$ as mutation matrix instead of $Q$.
As we commented in Remark \ref{remark:QNonRateMatrix}, Lemma \ref{lemma:eigenvaluesQKN} remains valid and it ensures us that $\operatorname{Spec}(\bar{  \mathcal{Q}  }_{N}) = \sigma_{N}(\bar{Q})$. Thus, using the triangular inequality we get
\[
	\operatorname{D} \left( \operatorname{Spec}(\mathcal{Q}_{N}), \sigma_{N} (Q) \right) \le \operatorname{D} \left( \operatorname{Spec}(\mathcal{Q}_{N}), \operatorname{Spec}(\bar{\mathcal{Q}}_{N}) \right) + \operatorname{D} \left( \operatorname{Spec}(\bar{\mathcal{Q}}_{N}), \sigma_{N} (Q) \right).
\]
Moreover,
\begin{align*}
	\| \mathcal{Q}_{N} - \bar{\mathcal{Q}}_{N} \|_{\text{Unif}} &\le N \|Q - \bar{Q}\|_{\text{Unif}},\\
	\operatorname{D} \left( \operatorname{Spec}(\bar{\mathcal{Q}}_{N}), \sigma_{N} (Q) \right) &\le N \operatorname{D} \left( \operatorname{Spec}(\bar{Q}), \operatorname{Spec}( Q ) \right).
\end{align*}
Fix $\epsilon > 0$. Using Fact 2, we know there exist $\delta_1, \delta_2 > 0$ such that
\begin{eqnarray*}
	\operatorname{D} \left( \operatorname{Spec}(\mathcal{Q}_{N}), \operatorname{Spec}(\bar{\mathcal{Q}}_{N}) \right) \le \frac{\epsilon}{2} &\text{ if }& \|\mathcal{Q}_{N} - \bar{ \mathcal{Q} }_{N} \|_{\text{Unif}} < \delta_1, \\
	\operatorname{D} \left( \operatorname{Spec}(\bar{Q}), \operatorname{Spec}( Q ) \right) \le \frac{\epsilon}{2 N} &\text{ if }& \|Q - \bar{Q} \|_{\text{Unif}} < \delta_2.
\end{eqnarray*}
Thus,
\[
	\operatorname{D} \left( \operatorname{Spec}(\mathcal{Q}_{N}), \sigma_{N} (Q) \right) \le \frac{\epsilon}{2} + N \operatorname{D}(\operatorname{Spec}( \bar{Q} ), \operatorname{Spec}(Q)) < \epsilon,
\]
whenever $\|Q - \bar{Q} \|_{\text{Unif}} < \min\{ \delta_1 / N, \delta_2 \}$. Since $\epsilon$ can be taken arbitrary small, by Fact 1, the proof of (a) is finished.

The proof of $(b)$ is exactly the same as the proof of (a) in Lemma \ref{lemma:eigenvaluesQKN}. Note that, since $\eta(r) = \dots = \eta(K-1) = 0$, the definition of $\tilde{V}_{\eta}$ only depends on the $r$ linearly independent vectors forming $\mathcal{U}$. Finally, the result in (c) trivially comes from Lemma \ref{lemma:eigenvaluesQKN}.
\end{proof}

\begin{remark}[Alternative proof for Theorem \ref{thm:eigenvalues_generalisation}]
	The \emph{Jordan--Chevalley decomposition} is an elegant tool to find the eigenvalues of $\mathcal{Q}_{N}$ and prove Theorem \ref{thm:eigenvalues_generalisation}. 
	The Jordan-Chevalley decomposition ensures the existence of two matrices $Q_{\mathrm{Diag}}$ and $Q_{\mathrm{Nil}}$ such that 
		$Q = Q_{\mathrm{Diag}} + Q_{\mathrm{Nil}}$.
	Moreover, $Q_{\mathrm{Diag}}$ is diagonalisable, $Q_{\mathrm{Nil}}$ is nilpotent, they commute and such a decomposition is unique. 
	See \cite[Prop.\ 3.20]{MR2744852} and \cite{MR2850401} for more details about the Jordan-Chevalley decomposition. Then, it can be proved that the Jordan-Chevalley decomposition of $\mathcal{Q}_{N}$ is $\mathcal{Q}_{N} = (\mathcal{Q}_{\mathrm{Diag}})_{N} + (\mathcal{Q}_{\mathrm{Nil}})_{N}$, where $(\mathcal{Q}_{\mathrm{Diag}})_{N}$ and $(\mathcal{Q}_{\mathrm{Nil}})_{N}$ are defined similarly to $\mathcal{Q}_{N}$ in (\ref{def:generatorQKN}), substituting $Q$ by $Q_{\mathrm{Diag}}$ and $Q_{\mathrm{Nil}}$, respectively. 
	Now, since the spectrum of $\mathcal{Q}_{N}$ is that of $(\mathcal{Q}_{\mathrm{Diag}})_{N}$, the proof of Theorem \ref{thm:eigenvalues_generalisation} follows from Lemma \ref{lemma:eigenvaluesQKN}.
\end{remark}

\subsection{Proof of Theorem \ref{thm:specA}} \label{subsec:reproduction_process}

In this section, given $K \geq 2$ and $N\geq 2$,
we consider the continuous-time Markov chain of $N$ indistinguishable particles on $K$ sites, with state space $\mathcal{E}_{K,N}$, where, with rate $1$, any particle jumps to one of the positions of another particle chosen at random.
We denote by $\mathcal{A}_{N}$ the infinitesimal generator of this reproduction process, which is defined in (\ref{def:reproduction}) as
\begin{equation*}
	\big( \mathcal{A}_{N} f \big) (\eta) = \sum_{i,j \in \left[ K \right]} \eta(i) \eta(j) \big[ f(\eta - \mathbf{e}_i + \mathbf{e}_j) - f(\eta) \big]
\end{equation*}
for every real function $f$ and all $\eta \in\mathcal{E}_{K,N}$.

\begin{remark}[First degree eigenfunctions of $\mathcal{A}_{N}$]\label{remark:1and2degree_eigenAKN}
	Note that the states $ \{N \, \mathbf{e}_k\}_{k=1}^K \subset \mathcal{E}_{K,N}$ are the only absorbing states for the interaction process generated by $\mathcal{A}_{N}$. 
	Thus, the distribution concentrated at $N \, \mathbf{e}_k$, denoted $\delta_{\{N \mathbf{e}_k\}}$, is stationary for $\mathcal{A}_{N}$, for $k \in \left[ K \right]$. 
	It is not difficult to check that the real functions on $\mathcal{E}_{K,N}$, $x_0 \equiv 1$ and $x_k: \eta \mapsto \eta_k$, for $k \in \left[ K-1 \right]$, are linearly independent vectors of $\mathbb{R}^{\mathcal{E}_{K,N}}$ and  they satisfy $\mathcal{A}_{N} x_k = 0$, for all $k = 0,1, \dots, K-1$.
	Thus, the right eigenspace associated to $0$ is the space of homogeneous polynomials of degree $1$, which has dimension $K$.  
\end{remark}

Actually, it can be proved that the generator $\mathcal{A}_{N}$ preserves the total degree of a polynomial, in the sense that the image of a polynomial is another polynomial of the same total degree.
To prove Theorem \ref{thm:specA} we first formally describe the preserving degree polynomial property satisfied by $\mathcal{A}_{N}$.

\begin{lemma}[$\mathcal{A}_{N}$ preserves polynomial total degree] \label{lemma:specA}
Assume $K \geq 2$ and $N \geq 2$. 
Let $P$ be a polynomial on $\mathcal{E}_{K,N}$
of total degree $L$ with $1 \leq L \leq N$.
Then,
\begin{equation*}
\mathcal{A}_{N} V_P = - L(L-1) V_P + V_R,
\end{equation*}
where $R$ is a polynomial with a total degree strictly less than $L$.
\end{lemma}

The proof of Lemma \ref{lemma:specA} is technical and it is deferred to Appendix \ref{section:proof_lemma_specA}.  We next prove Theorem \ref{thm:specA}.

\begin{proof}[Proof of Theorem \ref{thm:specA}]
(a) For $K \geq 2$ and $N \geq 2$, let us define the sets $\mathcal{B}_L$ of monomials 
in $\mathcal{E}_{K,N}$ as follows
\[
\mathcal{B}_0 := \{1\},  \;\;
\mathcal{B}_1 := \{x_1,x_2,\dots,x_{K-1} \}, \; \;
\mathcal{B}_L := \{\mathbf{x}^{\alpha}, \; \alpha \in \mathcal{E}_{K-1,L}\},\;\; 
\]
for $2 \leq L \leq N$, where $\mathbf{x}^{\alpha} = x_1^{\alpha_1} x_2^{\alpha_2}\dots x_{K-1}^{\alpha_{K-1}}$ for
$\alpha := (\alpha_1,\alpha_2,\dots,\alpha_{K-1})$.
Then, consider the ordered set
\begin{equation*}
\mathcal{B} = \mathcal{B}_0 \cup \mathcal{B}_1 \cup \dots \cup \mathcal{B}_N.
\end{equation*}
The set $\mathcal{B}$ is a basis of the space of real functions on $\mathcal{E}_{K,N}$, due to Lemma \ref{thm:interp}-(b).
The matrix similar to $\mathcal{A}_{N}$ with respect to this basis
is $\bar{\mathcal{A}}_{N} = W^{-1} \mathcal{A}_{N} W$, where $W$
is the matrix with $P$, with $P \in \mathcal{B}$, as column vectors.
Thanks to the result in Lemma \ref{lemma:specA}-(a), $\bar{\mathcal{A}}_{N}$ is a block upper triangular matrix, where 
the first diagonal block has size $K$ and is a null matrix. 
The other diagonal blocks have size $\mbox{Card}(\mathcal{E}_{K-1,L}) =
\binom{  K  - 2 + L }{ L }$ and are diagonal matrices
with constant diagonal elements equal to $-L(L-1)$, with $2 \leq L \leq N$.
This analysis gives us the eigenvalues of $\mathcal{A}_{N}$
are $0$ with algebraic multiplicity $K$ and $-L(L-1)$ with algebraic multiplicity $\binom{  K  - 2 + L }{ L }$ for $2 \leq L \leq N$.

Now, using the block multiplication of matrices, it is not difficult to see that $(\bar{\mathcal{A}}_{N})^n$ is also a block diagonal matrix, where the $L$-th block is a diagonal matrix of dimension $\binom{K  - 2 + L}{L}$ with all the entries on the diagonal equal to $(-L(L-1))^n$, for $2 \le L \le N$.
Thus, for every real polynomial $\Upsilon$ the matrix $\Upsilon(\bar{\mathcal{A}}_{N}) = W^{-1} \Upsilon(\mathcal{A}_{N}) W$ is a block diagonal matrix with diagonal elements $\Upsilon(-L(L-1))$. Taking
\[
 	\Upsilon: s \mapsto s\prod_{L=2}^N [s + L (L-1)],
\]
we get $\Upsilon(\bar{\mathcal{A}}_{N}) = \mathbf{0}_{K,N}$, where $\mathbf{0}_{K,N}$ is the $\binom{K - 1 + N }{N}$ dimensional null matrix. Thus, $\Upsilon(\mathcal{A}_{N}) = \mathbf{0}_{K,N}$ and $\Upsilon$ is necessarily the \emph{minimal polynomial} of $\mathcal{A}_{N}$, which factors into distinct linear factors. We thus conclude that $\mathcal{A}_{N}$ is diagonalisable.
\end{proof}

\begin{remark}[On the right eigenfunctions of $\mathcal{A}_{N}$]
	Theorem \ref{thm:specA} does not provide a characterisation of the eigenspace associated to the eigenvalue $-L(L-1)$, for $L \in \left[ N \right]$. For the special case $K=2$, Watterson \cite{watterson_markov_1961} does provide such a decomposition for the discrete analogue of $\mathcal{A}_{N}$ in terms of cumulative sums of discrete Chebyshev polynomials. In addition, Zhou  \cite[\S 4.2.2]{zhou_examples_2008} provides an equivalent but simpler expression for the eigenvectors of the equivalent analogous of $\mathcal{A}_{N}$, for $K = 2$, in terms of univariate Hahn polynomials. 
	In the general case ($K \ge 3$), it is possible to describe the eigenspaces associated to the first three eigenvalues of $\mathcal{A}_{N}$. As we commented in Remark \ref{remark:1and2degree_eigenAKN}, the right eigenspace associated to $0$ is the space of homogeneous polynomials of first degree. 
	Moreover, the right eigenspace associated to $-2$ has dimension $K(K-1)/2$ and it is generated by the set of monomials $\{ x_k x_r, \; 1 \leq k < r \leq K\}$.
	Additionally, for $L=3$, it is possible to prove that a simple basis of the right eigenspace associated to $-6$ has dimension $K(K+1)(K-1)/6$ and is given by eigenvectors
	$\{ x_k^2 x_r - x_k x_r^2, \; 1 \leq k < r \leq K\} \cup
	\{x_k x_r x_s,\; 1 \leq k < r < s \leq K\}$.
	The complete characterisation of the eigenvectors of $\mathcal{A}_{N}$, for $K \ge 3$, is a topic of further research.
\end{remark}

\subsection{Proof of Theorem \ref{thm:mainspec_intro}} \label{subsec:walkfv}

This section is devoted to the proof of Theorem \ref{thm:mainspec_intro} providing a description of the spectrum of the neutral multi-allelic Moran process with generator $\mathcal{Q}_{N,p}$, defined by (\ref{def:generatorQKNp}) as
\begin{equation*}
	\big( \mathcal{Q}_{N,p} f \big) (\eta) = \sum_{i,j \in \left[ K \right]} \eta(i) \left( \mu_{i,j} + \frac{p}{N} \, \eta(j) \right) \big[ f(\eta - \mathbf{e}_i + \mathbf{e}_j) - f(\eta)\big],
\end{equation*}
for every real function $f$ in $\mathcal{E}_{K,N}$ and every $\eta \in \mathcal{E}_{K,N}$.

Assume $K \ge 2$, $N \ge 2$ and $p \in [0,\infty)$ and suppose that $Q$ is diagonalisable with eigenvalues $0$ and $\lambda_k$, for $k \in \left[ K-1 \right]$. For any $\eta \in \mathcal{E}_{K-1,L}$, with $L \in \left[ N \right]$, let us recall the definition of $\lambda_{\eta,p}$:
	\begin{equation} \label{eq:lametap_section}
		\lambda_{\eta,p} = - L(L-1) \frac{p}{N} + \sum_{k=1}^{K-1} \eta(k) \lambda_k.
	\end{equation}
Then, we will prove that the eigenvalues of $\mathcal{Q}_{N,p}$ are $0$ and all $\lambda_{\eta,p}$ for $\eta \in \bigcup\limits_{L=1}^{N} \mathcal{E}_{K-1,L}$.

\begin{proof}[Proof of Theorem \ref{thm:mainspec_intro}]

Recall that 
\(
\mathcal{Q}_{N,p} = \mathcal{Q}_{N} + \frac{p}{N} \mathcal{A}_{N},
\)
where $\mathcal{Q}_{N}$ and $\mathcal{A}_{N}$ are the generators of the mutation and the reproduction processes defined by (\ref{def:generatorQKN}) and (\ref{def:reproduction}), respectively. 

Let us first prove the result when the mutation rate matrix $Q$ is diagonalisable. 
As proved in Lemma \ref{lemma:eigenvaluesQKN}, the vector $\tilde{V}_{\eta}$ is an eigenvector of $\mathcal{Q}_{N}$
with eigenvalue $\lambda_{\eta}$, for
$\eta \in \bigcup_{L=1}^{N} \mathcal{E}_{K-1,L}$.
Let us denote by $\tilde{V}_0$ the  all-one vector in $\mathbb{R}^{\mathcal{E}_{K,N}}$. Then, the set $\mathcal{B} = \{V_0\} \cup \{\tilde{V}_{\eta}, \; \eta \in \bigcup_{L=1}^{N} \mathcal{E}_{K-1,L} \}$
is a basis of $\mathbb{R}^{\mathcal{E}_{K,N}}$, thanks to Proposition \ref{propo:basis}-(c). 
Let us denote by $W$ the matrix with the elements of $\mathcal{B}$ as columns such that $W^{-1} \mathcal{Q}_{N} W$ is a diagonal matrix with diagonal entries equal
to
$0$ and $\lambda_\eta$, for $\eta \in \bigcup_{L=1}^{N} \mathcal{E}_{K-1,L}$.

For $1 \leq L \leq N$ and $\eta \in \mathcal{E}_{K-1,L}$, the expression (\ref{def:Veta_tilde}) and Lemma \ref{thm:chiV}-(c) ensure that $\tilde{V}_{\eta}$  is a polynomial of total degree equal to $L$.
Using now Theorem \ref{thm:specA}-(b), we get
$$
\mathcal{A}_{N} \tilde{V}_{\eta} = - L(L-1) \tilde{V}_{\eta} + R,
$$
where $R$ is a polynomial of total degree strictly less than $L$.
This fact means that, like in Theorem \ref{thm:specA}-(c), $W^{-1} \mathcal{A}_{N} W $ 
is a block upper triangular matrix, where 
the diagonal blocks of size $\operatorname{Card}(\mathcal{E}_{K-1,L}) =
\binom{  K + L - 2 }{ L }$ are diagonal matrices
with constant diagonal elements
equal to $-L(L-1)$, for $2 \leq L \leq N$.
The first diagonal block of size $K$ is a null matrix.
It follows that
\[
	W^{-1} \mathcal{Q}_{N,p} W  = W^{-1} \mathcal{Q}_{N} W + \frac{p}{N} W^{-1} \mathcal{A}_{N} W
\]
is a block upper triangular matrix, where the first diagonal block has dimension one and is null, i.e.\ the first column is null. Moreover, the $L$-th diagonal block has dimension $\binom{  K + L - 2 }{ L }$ and its diagonal elements are the eigenvalues 
\(
	\lambda_{\eta,p} \text{ with } \eta \in \mathcal{E}_{K-1, L},
\)
for $L \in \left[ N \right]$. Thus, these are the eigenvalues of $\mathcal{Q}_{N,p}$.

Now, consider a general mutation matrix $Q \in M_K(\mathbb{R})$, not necessarily diagonalisable, with spectrum $\operatorname{Spec}(Q) = \{0, \lambda_1, \dots, \lambda_{K-1}\}$. 
Let us define by $\sigma_{N,p}(Q)$ the set formed by $0$ and $\lambda_{\eta,p}$, for $\eta \in \bigcup_{L = 1}^{K-1} \mathcal{E}_{K-1,L}$, where the values $\lambda_k$ in the definition (\ref{eq:lametap_section}) of $\lambda_{\eta, p}$, are those in $\operatorname{Spec}(Q)$.
Define $\bar{\mathcal{Q}}_{N,p}$ similarly to (\ref{def:generatorQKNp}) but with a diagonalisable matrix $\bar{Q} \in M_K(\mathbb{C})$, whose rows have null sum (not necessarily a rate matrix), instead of $Q$.
Then,
\begin{align*}
	\| \mathcal{Q}_{N,p} - \bar{\mathcal{Q}}_{N,p} \|_{\text{Unif}} &= \| \mathcal{Q}_{N} - \bar{\mathcal{Q}}_{N} \|_{\text{Unif}}, \\
	\operatorname{D} \left( \operatorname{Spec}( \bar{ \mathcal{Q} }_{N,p}), \sigma_{N,p}(Q) \right) &= \operatorname{D} \left( \operatorname{Spec}( \bar{ \mathcal{Q} }_{N}), \sigma_{N}(Q) \right).
\end{align*}
Hence, $\sigma_{N,p}(Q)$ is proved to be the spectrum of $\mathcal{Q}_{N,p}$, analogously to the proof of Theorem \ref{thm:eigenvalues_generalisation}-(a). 
\end{proof}

\begin{remark}[Alternative proof of Theorem \ref{thm:mainspec_intro}]
	Another proof of Theorem \ref{thm:mainspec_intro} can be carried out using the Jordan form of the mutation rate matrix $Q$. Indeed, the vectors $\tilde{V}_{\eta} \in \mathbb{R}^{\mathcal{E}_{K,N}}$ can be defined using the basis of $\mathbb{R}^K$ that transforms $Q$ in its normal Jordan form. Then, defining a suitable order among the vectors $\tilde{V}_{\eta}$, for $\eta \in \bigcup_{L = 1}^{N} \mathcal{E}_{K-1,L}$, it is possible to show that $\mathcal{Q}_{N,p}$ is similar to an upper triangular matrix with the values $\lambda_{\eta, p}$ on the diagonal.
\end{remark}

\section{Applications to the convergences to stationarity}\label{section:SLEM&ergodicity}

This section devoted to  some applications of the results in Section \ref{sec:eigenvalues_moran} to the study of the ergodicity of the process driven by $\mathcal{Q}_{N,p}$ in total variation, using spectral properties of $Q$.
In this section we prove Corollary \ref{corol:ergodicity_gral_TV} and Theorem \ref{thm:lower_bound}.
First, let us establish that the Jordan form of $Q$ is a diagonal block in the Jordan form of $\mathcal{Q}_{N,p}$. 

\begin{corollary}[Jordan forms of $Q$ and $\mathcal{Q}_{N,p}$] \label{cor:spectrum_QKNp1stdegreePoly}
	Consider $K \ge 2$, $N \ge 2$ and $p \ge 0$.
	If $J$ is the Jordan form of $Q$, then the Jordan normal form of $\mathcal{Q}_{N,p}$ is $J \oplus J'$, where $J'$ is a Jordan matrix of dimension $\binom{K - 1 + N}{N} - K$.
	In particular, $\mathrm{e}^{t Q}$ and $\mathrm{e}^{ t \mathcal{Q}_{N,p} }$ have that same second largest eigenvalue in modulus $(\mathrm{SLEM})$, for every $t \ge 0$.
\end{corollary}

\begin{proof}
	The image by $\mathcal{Q}_{N,p}$ of a first degree polynomial is also a first degree polynomial, i.e.\ the space of first degree polynomials is invariant by $\mathcal{Q}_{N,p}$. 
	Moreover, as a consequence of Lemma \ref{lemma:relation_generator} we obtain
	\[
		\mathcal{Q}_{N,p} \, \tilde{\xi} ({V}) = \mathcal{Q}_{N} \, \tilde{\xi} ({V}) = \Phi_{K,N} \, \mathcal{D}_{N} \xi ({V}) = \Phi_{K,N} \, \xi({Q V}) = \tilde{\xi}(Q {V}).
	\]
	Let $\mathcal{U} = \{U_0, \dots, U_{K-1}\}$ by a Jordan basis of $Q$ formed by generalised eigenvectors of $Q$. 
	Since $\mathcal{Q}_{N, p} \, \tilde{\xi}(U_k) = \tilde{\xi}(Q \, U_k)$, for every $k \in \left[ K-1 \right]_0$, we have that $\{\tilde{\xi}(U_0), \dots, \tilde{\xi}(U_{K-1})\}$ is a system of linearly independent generalised eigenvectors of $Q_{N,p}$. 
	They are precisely the generalised eigenvectors of $\mathcal{Q}_{N,p}$ associated to the eigenvalues in $\operatorname{Spec}(Q) \subset \operatorname{Spec}(\mathcal{Q}_{N,p})$. 
	We can complete this system to a Jordan basis of $\mathbb{R}^{\mathcal{E}_{K,N}}$, adding the generalised eigenvectors of the other eigenvalues on $\mathcal{Q}_{N,p}$. 
	With respect to this Jordan basis $\mathcal{Q}_{N,p}$ becomes similar to $J \oplus J'$, where $J$ is the Jordan matrix of $Q$ and $J'$ is a Jordan matrix of dimension $\binom{K - 1 + N}{N} - K$.
	Note that the eigenvalues $\{ \lambda_0, \lambda_1, \dots, \lambda_{K-1} \}$ are those eigenvalues of $\mathcal{Q}_{N,p}$ of smallest modulus. We thus get that $\mathrm{e}^{t Q}$ and $\mathrm{e}^{ t \mathcal{Q}_{N,p} }$ have the same $\mathrm{SLEM}$, for every $t \ge 0$.
\end{proof}

Every irreducible finite  Markov chain convergences exponentially to stationarity, see e.g.\ \cite[Thm.\ 4.9]{levin_markov_2017}. In addition, the sharpest asymptotic speed of convergence is associated to the $\mathrm{SLEM}$ and the size of the largest Jordan block corresponding to any eigenvalue with this modulus. We recall that the size of the largest Jordan block associated to an eigenvalue $\lambda$ is equal to the multiplicity of $\lambda$ in the minimal polynomial of the rate matrix of the Markov chain.

\begin{proof}[Proof of Corollary \ref{corol:ergodicity_gral_TV}]

Let $\mathrm{e}^{- \rho t}$ be the $\mathrm{SLEM}$ of $\mathrm{e}^{t Q}$ and $s$ the largest multiplicity in the minimal polynomial of $\mathrm{e}^{t Q}$ of all the eigenvalues with modulus $\mathrm{e}^{- \rho t}$, or equivalently, the size of the largest Jordan block associated to eigenvalues with modulus $\mathrm{e}^{- \rho t}$. 
Then, 
\begin{equation}\label{eq:convergence_MC_infinitynorm}
	\operatorname{D}^{\mathrm{TV}}_Q( t ) = \Theta (t^{s-1} \mathrm{e}^{- \rho t}),
\end{equation}
see e.g.\ \cite[Thm.\ 3.2]{MR3296158}.
The result follows as a consequence of Corollary \ref{cor:spectrum_QKNp1stdegreePoly} and \eqref{eq:convergence_MC_infinitynorm}.

\end{proof}

The following example uses Corollary \ref{corol:ergodicity_gral_TV} to provide the rates for the exponential convergence to stationarity of the neutral multi-allelic Moran process considered in \cite{corujo_dynamics_2020}. 

\begin{example}[Circulant mutation rate matrix]
Consider the following mutation rate matrix
 \begin{equation*} 
 	Q_{\theta} =
 	\left(
 		\begin{array}{cccccc}
 			-(1+ \theta) &  1 &  0 &    \dots &  0 & \theta\\
 			 \theta & -(1 + \theta) &    1 & \dots & 0 & 0\\
 			 0 &  \theta &  -(1 + \theta) & \dots & 0 & 0\\
 			  \vdots &  \vdots &    \vdots & \ddots & \vdots & \vdots \\
 			 1 &   0 &  0 &  \dots & \theta & -(1 + \theta)\\			
 		\end{array}
 	\right),
 \end{equation*}
 where $\theta \ge 0$. $Q_{\theta}$ is the infinitesimal generator of a simple asymmetric random walk on the $K$-cycle graph. The neutral multi-allelic Moran type process with mutation rate $Q_{\theta}$ was considered in \cite{corujo_dynamics_2020}. 
 Since $Q_{\theta}$ is circulant,  it is possible to explicitly diagonalise it using the Fourier matrix. The eigenvalues of $Q_{\theta}$ are
	\[
		\lambda_k = -2(1 + \theta)\sin^2\left( \frac{\pi k}{K} \right) + \mathrm{i}(1-\theta) \sin\left( \frac{ 2 \pi k}{K} \right),
	\]
	for $0 \leq k \leq K-1$. 
	Thus, the $\mathrm{SLEM}$ of $\mathrm{e}^{t Q_{\theta} }$ is $\mathrm{e}^{ -2(1 + \theta)\sin^2\left( \frac{\pi}{K} \right) t} $, which is attained for two eigenvalues, each one of them with algebraic multiplicity equals to $1$, for $\theta \neq 1$. 
	When $\theta = 1$, the $\mathrm{SLEM}$ of $\mathrm{e}^{t Q_{\theta} }$ is $\mathrm{e}^{ - 4 \sin^2\left( \frac{\pi}{K} \right)}$ and it is attained for a unique eigenvalue with algebraic and geometric multiplicities equal to $2$. 
	Let $\mathcal{Q}_{\theta}$ be the infinitesimal generator of the neutral multi-allelic Moran process with mutation rate $Q_{\theta}$. Then,
	\[
		\operatorname{D}^{\mathrm{TV} }_{\mathcal{Q}_\theta} \left( t \right) = \Theta \left( \mathrm{e}^{- 2(1 + \theta)\sin^2\left( \frac{\pi}{K} \right) t} \right).
	\]
\end{example}

\begin{example}[Convergence rate for a process with non-diagonalisable mutation rate matrix]
	Consider the rate matrix $Q$ defined as
	$$
	Q = \left(
		\begin{array}{rrr}
			-9 & 7 & 2 \\
			1 & -7 & 6 \\
			5 & 7 & -12
		\end{array}
	\right)
	$$
	and $\mathcal{Q}_{N,p}$ the infinitesimal generator of the associated neutral multi-allelic Moran process with mutation rate matrix $Q$.  Then, $\lambda_0 = 0$ and $\lambda_1 = \lambda_2 = -14$, because $-14$ has algebraic multiplicity $2$. Then, for $N$ fixed, the eigenvalues of $\mathcal{Q}_{N,p}$ are
	\[
		\lambda_{L,p} := \eta(1) \lambda_1 + \eta(2) \lambda_2 - L (L - 1) \frac{p}{N} = -14 L - L(L-1) \frac{p}{N},
	\]
	for $L \in \left[ N \right]_0$. In addition, $\lambda_{L,p}$ has algebraic multiplicity $\operatorname{Card}(\mathcal{E}_{2,L}) = L + 1$.

	Note that the minimal polynomial of $Q$ is $m_Q: s \mapsto s \, (s+14)^2$ and according to the notation in Corollary \ref{corol:ergodicity_gral_TV} we get $\rho = 14$ and $s = 2$. Then,
	\[
		\operatorname{D}^{\mathrm{TV}}_{ \mathcal{Q}_{N,p} }( {t} ) = \Theta\left( t \, \mathrm{e}^{- 14 t }\right).
	\]
\end{example}

\subsection{Proof of Theorem \ref{thm:lower_bound}}\label{sec:proof:Thmlowerbound}

Let us denote by $\Gamma_{\mathcal{L}}$ the \emph{carr\'e-du-champ} operator associated to the Markov generator $\mathcal{L}$ on a state space $\mathcal{E}$, i.e.
\[
	\Gamma_{\mathcal{L}} f : \eta \mapsto \big(\mathcal{L} f^2\big)(\eta) - 2 f(\eta) (\mathcal{L} f)(\eta),
\]
for all $\eta \in \mathcal{E}$.

The {carr\'e-du-champ} operator is associated to the evolution in time of the variance of the test function. Indeed,
\[
	\operatorname{Var}_{\eta} (f (\eta_t)) = \int_0^t \mathrm{e}^{s \mathcal{L}} \Big(\Gamma_{\mathcal{L}} \big(\mathrm{e}^{(t-s) \mathcal{L}} f\big) \Big)(\eta) \mathrm{d}s,
\]
where $(\mathrm{e}^{t \mathcal{L}})_{t \ge 0}$ denotes the semigroup generated by $\mathcal{L}$.
See, for example, \cite[p.\ 695]{cloez_quantitative_2016}.

\begin{proof}[Proof of Theorem \ref{thm:lower_bound}]

Our method of proof is based on Wilson's method (cf.\ \cite[Thm.\ 13.28]{levin_markov_2017}).
Let us denote $V = [v_1, v_2, \dots, v_K]$ a real right-eigenvector satisfying $Q V = - \lambda V$.
Then, using Theorem \ref{thm:mainspec_intro} and Lemma \ref{thm:chiV} (specifically equations \eqref{eq:chi1} and \eqref{eq:2degreeVP}) we get that $\tilde{\xi}(V)$ and $\tilde{\xi}(V,V)$ are right-eigenfunctions of $\mathcal{Q}_{N,p}$ satisfying 
\begin{align*}
	\big(\mathrm{e}^{ t \mathcal{Q}_{N,p}} \tilde{\xi}(V)\big) (\eta) &= \mathrm{e}^{ -t  \lambda} \tilde{\xi}(V)(\eta)\\
	\big(\mathrm{e}^{ t \mathcal{Q}_{N,p}} \tilde{\xi}(V,V)\big) (\eta) &= \mathrm{e}^{ - 2 (\lambda + p/N) t } \tilde{\xi}(V,V)(\eta),
\end{align*}
for every $\eta \in \mathcal{E}_{K,N}$. 
We recall that from \eqref{eq:2degreeVP} we have
\[
	\tilde{\xi}(V,V) = \tilde{\xi}(V)^2 - \tilde{\xi}(V \odot V).
\]
where $V \odot V = [v_1^2, \dots, v_K^2]$ is the componentwise square vector of $V$.

Thereafter, using $\tilde{\xi}(V)$ as a test function we get
\begin{equation}\label{eq:lower_bound_TV_general}
	\operatorname{d}^{\mathrm{TV}}(\delta_{N \mathbf{e}_k} \mathrm{e}^{ t_{N,c} \mathcal{Q}_{N,p} }, \nu_{N,p}) \ge \mathbb{P}_{N \mathbf{e}_k}\left[\tilde{\xi}(V)(\eta_t) \ge \mu_t/2 \right] - \mathbb{P}_{\nu_{N,p}}\left[\tilde{\xi}(V)(\eta_\infty) \ge \mu_t/2 \right],
\end{equation}
where $\mu_t =  \mathbb{E}_{N \mathbf{e}_k}\left[ \tilde{\xi}(V)(\eta_t) \right]  = \mathrm{e}^{ -t  \lambda} N v_k$.
Using Markov and Chebyshev inequalities we obtain
\begin{align*}
	\mathbb{P}_{\nu_{N,p}} \left[ \tilde{\xi}(V)(\eta_\infty) \ge \frac{\mu_t}{2} \right] &\le 4  \mathrm{e}^{ 2 t  \lambda} \frac{ \operatorname{Var}_{\nu_{N,p}} \left[ \tilde{\xi}(V)(\eta_\infty) \right] }{
	|v_k|^2 N^2} , \\
	\mathbb{P}_{N \mathbf{e}_k}\left[\tilde{\xi}(V)(\eta_t) \ge \frac{\mu_t}{2} \right] &\ge 1 - 4 \mathrm{e}^{ 2 t  \lambda} \frac{ \operatorname{Var}_{N \mathbf{e}_k} \left[ \tilde{\xi}(V)(\eta_t) \right]}{
	|v_k|^2 N^2}.
\end{align*}
Thus, plugging these last expressions into \eqref{eq:lower_bound_TV_general} we get
\[
	\operatorname{d}^{\mathrm{TV}}(\delta_{N \mathbf{e}_k} \mathrm{e}^{ t \mathcal{Q}_{N,p} }, \nu_{N,p}) \ge 1 - 8 \frac{\mathrm{e}^{ 2 \lambda t }}{|v_k|^2 N} \sup_{t \ge 0} \frac{  \operatorname{Var}_{N \mathbf{e}_k} \left[ \tilde{\xi}(V)(\eta_t) \right]}{ N}.
\]

We are interested in finding a lower bound for $\operatorname{D}^{\mathrm{TV}}_{\mathcal{Q}_N}$ at time $t_{N,c} = (\ln(N) - c)/2\lambda$.
It remains to prove a bound for the last factor in the previous expression.

Note that
\begin{align*}
	\Gamma_{\mathcal{Q}_{N,p}} \tilde{\xi}(V)  &= \mathcal{Q}_{N,p} (\tilde{\xi}(V)^2) - 2 \tilde{\xi}(V) \mathcal{Q}_{N,p} (\tilde{\xi}(V)) \\
	&= \mathcal{Q}_{N,p} (\tilde{\xi}(V,V)) +  \mathcal{Q}_{N,p} \tilde{\xi}(V \odot V) - 2 \tilde{\xi}(V) \mathcal{Q}_{N,p} (\tilde{\xi}(V)) \\
	&= - 2 \left(\lambda + \frac{p}{N} \right) \tilde{\xi}(V,V) + 2 \lambda \tilde{\xi}(V)^2 + \tilde{\xi}\big(Q (V \odot V) \big)\\
	&= - 2 \frac{p}{N} \tilde{\xi}(V,V) + 2 \lambda \tilde{\xi}(V \odot V) +  \tilde{\xi}\big(Q (V \odot V) \big).
\end{align*}
Thus,
\begin{align*}
	\frac{\operatorname{Var}_{N \mathbf{e}_k} \left( \tilde{\xi}(V) (\eta_t) \right)}{N} &= \frac{1}{N} \int_0^t \mathrm{e}^{s \mathcal{Q}_{N,p}} \Big( \Gamma_{\mathcal{Q}_{N,p}} \big( \mathrm{e}^{(t-s) \mathcal{Q}_{N,p}} \tilde{\xi}(V) \big) \Big) (N \mathbf{e}_k) \mathrm{d}s \\
	& \hspace*{-45 pt} = \frac{1}{N} \int_0^t \mathrm{e}^{- \lambda (t-s)} \mathrm{e}^{s \mathcal{Q}_{N,p}} \Big( \Gamma_{\mathcal{Q}_{N,p}}  \tilde{\xi}(V) \Big) (N \mathbf{e}_k) \mathrm{d}s\\
	& \hspace*{-45 pt} = \frac{1}{N} \int_0^t \mathrm{e}^{- \lambda (t-s)} \mathrm{e}^{s \mathcal{Q}_{N,p}} \Big( - 2 \frac{p}{N} \tilde{\xi}(V,V) + 2 \lambda \tilde{\xi}(V \odot V ) +  \tilde{\xi}\big(Q (V \odot V)\big) \Big) (N \mathbf{e}_k) \mathrm{d}s.
\end{align*}
Note that
\[
\frac{1}{N} \int_0^t \mathrm{e}^{- \lambda (t-s)} \mathrm{e}^{s \mathcal{Q}_{N,p}} \Big( - 2 \frac{p}{N} \tilde{\xi}(V,V) \Big) (N \mathbf{e}_k) \mathrm{d}s = -2 \left( 1 - \frac{1}{N} \right) p \, v_k^2\,  \mathrm{e}^{- \lambda t} \int_0^t \mathrm{e}^{ - (\lambda + 2 p/N)s} \mathrm{d}s  \le 0.
\]
Hence, 
\begin{align*}
	\frac{\operatorname{Var}_{ N \mathbf{e}_k} \left( \tilde{\xi}(V) (\eta_t) \right)}{N } 
	&\le \frac{1}{N } \int_0^t \mathrm{e}^{- \lambda (t-s)} \mathrm{e}^{s \mathcal{Q}_{N,p}} \Big( 2 \lambda \tilde{\xi}(V \odot V) +  \tilde{\xi}\big(Q (V\odot V)\big) \Big) ({N \mathbf{e}_k}) \mathrm{d}s \\
	&\le\left(  2 \lambda \left\| \frac{\tilde{\xi}(V \odot V )}{N} \right\|_{\infty} + \left\| \frac{\tilde{\xi}\big(Q (V \odot V)\big)}{N} \right\|_{\infty} \right) \int_{0}^{t} \mathrm{e}^{- \lambda (t-s)} \mathrm{d}s \\
	&\le (2 + \|Q\|_{\infty}/\lambda) \|V\|_{\infty}.
\end{align*}
The lower bound for $\operatorname{D}_{\mathcal{Q}_{N,p}}^{\mathrm{TV}}$ is obtained considering the initial distribution concentrated at $N \mathbf{e}_{k^\star}$, where $k^\star$ satisfies $|v_{k^\star}| = \|V\|_{\infty}$.

\end{proof}

\section{Neutral multi-allelic Moran type process with parent independent mutation} \label{sec:reversible}

In this section we discuss some applications of the Theorem \ref{thm:mainspec_intro} and its consequences to the neutral multi-allelic Moran model with parent independent mutation scheme. 
We will use some well-known results on finite state reversible Markov chains and their convergence to stationarity.
We refer the interested reader to \cite{bremaud_markov_2020}, \cite{levin_markov_2017} and \cite{MR1490046}, for further details.
We will focus on the case where the Moran process has \emph{parent independent mutation} \cite{etheridge_mathematical_2011}, and thus it is reversible.
In fact, as we claimed in Lemma \ref{thm_reversible_distrib}, the neutral Moran process with $p > 0$ is reversible if and only if its mutation matrix satisfies the \emph{parent independent} condition. 
We explicitly diagonalise the infinitesimal generator of the neutral multi-allelic Moran process with parent independent mutation rate using the multivariate Hahn and Krawtchouk polynomials, which allows us to provide an explicit expression for the transition function of this process. 
Using these results we prove Theorems \ref{thm:cutoff} and \ref{thm:TVcutoff}.

\subsection{Proof of Theorems \ref{thm:cutoff} and \ref{thm:TVcutoff} }\label{subsec:reversibleMoranModel}
Let us recall that the generator of the neutral multi-allelic Moran process with parent independent mutation defined by (\ref{def:generatorQKNp_reversible}), which acts on a real function $f$ on $\mathcal{E}_{K,N}$ as follows
\begin{equation*}\label{def:generatorQKNp_reversible}
	\big( \mathcal{L}_{N,p} f \big) (\eta) := \sum_{i,j=1}^{K} \eta(i) \left[ f(\eta - \mathbf{e}_i + \mathbf{e}_j) - f(\eta)\right]
\left( \mu_j + p \frac{\eta(j)}{N} \right),
\end{equation*}
for all $\eta \in \mathcal{E}_{K,N}$.

\subsection*{Multivariate orthogonal Hahn and Krawtchouk polynomials} The rest of the section is devoted to the characterisation of the eigenfunctions of $\mathcal{L}_{N,p}$ and the proof of Theorem \ref{thm:cutoff}. Let us establish some notation that will be useful in the sequel to study the eigenfunctions of $\mathcal{L}_{N,p}$. For a $K$-dimensional real vector $\mathbf{x}$ we define the following quantities:
\begin{equation*}
	|\mathbf{x}_i| := \sum_{j=1}^i x_j, \;\;\;
	|\mathbf{x}^i| := \sum_{j=i}^K x_j.
\end{equation*}
We set by convention $|\mathbf{x}^{i}| := 0$, for all $i > K$.

The orthogonal polynomials we define below are indexed by the set $\bigcup_{L = 0}^N  \mathcal{E}_{K-1,L}$, where $\mathcal{E}_{K-1,0} = \{\mathbf{0}\}$ is the set formed by the $K-1$ dimensional null vector.
We define the \emph{multivariate Hahn polynomials} on $\mathcal{E}_{K,N}$, indexed by $\eta \in \mathcal{E}_{K-1, L}$, for $L \in \left[ N \right]_0$, and denoted $H_{\eta}(\mathbf{x}; N, \pmb{\alpha})$, as follows
\begin{equation}\label{def:Hahn_poly_multiv}
	H_{\eta}(\mathbf{x}; N, \pmb{\alpha}) := \frac{1}{(N)_{[|\eta|]}} \prod_{k=1}^{K-1} (- N + |\mathbf{x}_{k-1}| + |\eta^{k+1}|)_{(\eta(k))} H_{\eta(k)}(x_k; M_k,\alpha_k, \gamma_k)
\end{equation}
where
\(
	M_k = N - |\mathbf{x}_{k-1}| - |\eta^{k+1}|
\),
\(
	\gamma_k =  |\pmb{\alpha}^{k+1}| + 2 |\eta^{k+1}|
\)
and $H_n(x; M, \beta, \gamma)$ is the \emph{univariate Hahn polynomial} defined by
\begin{align}
	H_n(x; M, \beta, \gamma)  &:= {}_3 F_2\left(
		\begin{array}{rc|}
			-n,&  n + \beta + \gamma - 1, -x \\
			\beta, & -M 
		\end{array}
		\; 1
	\right) \label{eq:Hahn_hypergeometric_notation}\\
		&= \sum_{j=0}^n \frac{(-n)_{(j)} (n + \beta + \gamma - 1)_{(j)} (-x)_{(j)} }{ \beta_{(j)} (-M)_{(j)}} \frac{1}{ j!}. \nonumber
\end{align}
Note that for $\mathbf{0} \in \mathcal{E}_{K-1,0}$ we obtain $H_{\mathbf{0}}(\cdot \, ; N, \pmb{\alpha}) \equiv 1$. In addition, it is not difficult to check that $H_{\eta}(N \mathbf{e}_{K}; N, \pmb{\alpha}) \equiv 1$, for all $\eta \in \bigcup\limits_{L=0}^N \mathcal{E}_{K-1,L}$.

We also define the \emph{multivariate Krawtchouk polynomials} on $\mathcal{E}_{K,N}$ denoted $K_{\eta}(\mathbf{x}; N, \mathbf{q})$, indexed by $\eta \in \bigcup\limits_{L = 0}^N \mathcal{E}_{K-1, L}$, with $\mathbf{q} \in (0,1)^K$ such that $|\mathbf{q}| = 1$, as the multivariate polynomials satisfying:
\begin{equation}\label{def:Krawtchouk_poly_multiv}
	K_{\eta}(\mathbf{x}; N, \mathbf{q}) := \frac{1}{(N)_{[|\eta|]}} \prod_{k=1}^{K-1} (- N + |\mathbf{x}_{k-1}| + |\eta^{k+1}|)_{(\eta(k))} K_{\eta(k)}\left(x_k; M_k, \frac{q_k}{|\pmb{q}^k|} \right)
\end{equation}
where
\(
	M_k = N - |\mathbf{x}_{k-1}| - |\eta^{k+1}|,
\)
and $K_n(x; N, q)$ is the \emph{univariate Krawtchouk polynomial} defined by
\begin{align}
	K_n(x; N, q)  &:= {}_2 {F}_1\left(
		\begin{array}{rc|}
			-n,&  -x \\
			 & \hspace*{-25pt} -N 
		\end{array}
		\; \frac{1}{q}
	\right) \label{eq:Krawtchouk_hypergeometric_notation}\\
		&= \sum_{j=0}^n \frac{(-n)_{(j)} (-x)_{(j)} }{ (-N)_{(j)}} \frac{1}{ j! q^j}. \nonumber
\end{align}
In addition, $K_{\mathbf{0}}(\cdot \, ; N, \mathbf{q}) \equiv 1$, for $\mathbf{0} \in \mathcal{E}_{K-1,0}$, and $K_{\eta}(N \mathbf{e}_{K}; N, \mathbf{q}) \equiv 1$, for all $\eta \in \bigcup\limits_{L=0}^N \mathcal{E}_{K-1,L}$.

See \cite[Ch.\ 6]{ismail_classical_2005} and \cite[Ch.\ 9]{koekoek_hypergeometric_2010} for more details about the univariate Hahn and Krawtchouk polynomials. 
We define the univariate Hahn and Krawtchouk polynomials in (\ref{eq:Hahn_hypergeometric_notation}) and (\ref{eq:Krawtchouk_hypergeometric_notation}), respectively, using the hypergeometric functions notation which could be very useful for algebraic manipulations (cf.\ \cite[Ch.\ 10]{koekoek_hypergeometric_2010}). 
For instance, consider $\pmb{\alpha} = N \pmb{\mu}/p$ in the definition of Hahn polynomials, then
\begin{align*}
	 \lim\limits_{p \rightarrow 0^+}H_{\eta(k)}(x_k; M_k, \alpha_k,  &|\pmb{\alpha}^{k+1}| + 2  |\eta^{k+1}|) = \lim\limits_{p \rightarrow 0^+}H_{\eta(k)}\left(x_k; M_k, \frac{N \mu_k}{p},  \frac{N |\pmb{\mu}^{k+1}|}{p} + 2  |\eta^{k+1}|\right) \\
	 	&= \lim\limits_{p \rightarrow 0^+}{}_3 F_2\left(
		\begin{array}{rc|}
			-\eta(k), &  \eta(k) + N \mu_k/p + N|\pmb{\mu}^{k+1}|/p + 2  |\eta^{k+1}| - 1, -x_k \\
			N \mu_k/p, & -M_k 
		\end{array}
		\; 1
	\right)\\
	&= {}_2 F_1\left(
		\begin{array}{c|}
			-\eta(k),  -x_k \\
			 -M_k 
		\end{array}
		\; \frac{\mu_k + |\pmb{\mu}^{k+1}|}{\mu_k}
	\right) \\
	&= K_{\eta(k)} \left(x_k; N , \frac{\mu_k}{|\pmb{\mu}^{k}|} \right),
\end{align*}
for every $k \in \left[ K \right]$, where the calculation of the limit in the third equation follows from \cite[Eq.\ (1.4.5)]{koekoek_hypergeometric_2010} and the last inequality follows from the definition of univariate Krawtchouk polynomials in (\ref{eq:Krawtchouk_hypergeometric_notation}). Now, using the previous limit and the definitions (\ref{def:Hahn_poly_multiv}) and (\ref{def:Krawtchouk_poly_multiv}) of the multivariate Hahn and Krawtchouk polynomials we get
\[
	\lim\limits_{p \rightarrow 0^+} H_{\eta}\left(\mathbf{x}; N, N \frac{\pmb{\mu}}{p}\right) = K_{\eta}\left(\mathbf{x}; N, \frac{\pmb{\mu}}{|\pmb{\mu}|}\right).
\]
Thus, similarly to how we define $\nu_{N,p}$ in (\ref{def:nuKNp}), we define the multivariate polynomial $Q_{\eta}(\cdot \, ; N, \pmb{\mu}, p)$ by
\begin{equation}\label{def:right_eigenvectors}
	Q_{\eta}(\mathbf{x}; N, \pmb{\mu}, p) := \left\{
		\begin{array}{ccc}
		H_{\eta}\left(\mathbf{x}; N , \frac{N \pmb{\mu}}{p}\right) & \text{ if } & p > 0\\
		K_{\eta}\left(\mathbf{x}; N , \frac{\pmb{\mu}}{|\pmb{\mu}|}\right) & \text{ if } & p = 0,
		\end{array}
	\right.
\end{equation}
for every $\eta \in \bigcup_{L = 0}^N \mathcal{E}_{K-1,L}$, and for all $\mathbf{x} \in \mathcal{E}_{K,N}$. Note that the functions $Q_{\eta}(\mathbf{x}; N, \pmb{\mu}, p)$ are continuous when $p$ tends towards zero, in the sense that:
\[
	\lim\limits_{p \rightarrow 0^+} Q_{\eta}\left(\mathbf{x}; N, \pmb{\mu}, p\right) = Q_{\eta}\left(\mathbf{x}; N, \pmb{\mu}, 0\right),
\]
for every $\mathbf{x} \in \mathcal{E}_{K,N}$.
The following result sets some important properties of the multivariate Hahn and Krawtchouk polynomials.

\begin{proposition}[Orthogonality of the Hahn and Krawtchouk polynomials]\label{proposition:Hahn_polynomials}
	The multivariate polynomials $Q_{\eta}$ defined by (\ref{def:right_eigenvectors}) satisfy the following properties:
	\begin{itemize}
		\item[a)] $Q_{\eta}(\cdot \, ; N, \pmb{\mu}, p)$ is a polynomial on $\mathcal{E}_{K,N}$ of total degree $|\eta|$, for every $\eta \in \bigcup\limits_{L=0}^N \mathcal{E}_{K-1,L}$.
		\item[b)] The polynomials $Q_{\eta}(\cdot \, ; N, \pmb{\mu}, p)$ are orthogonal on $\mathcal{E}_{K,N}$ with respect to the probability distribution $\nu_{N,p}$, defined by (\ref{def:nuKNp}), i.e.
	\begin{align*}
		\mathbb{E}_{\nu_{N,p}} \left[ Q_{\eta}( \cdot \, ; N, \pmb{\mu},p ) Q_{\eta'}( \cdot \, ; N, \pmb{\mu}, p) \right] &= \sum_{\xi \in \mathcal{E}_{K,N} } Q_{\eta}( \xi \, ; N, \pmb{\mu}, p ) Q_{\eta'}( \xi \, ; N, \pmb{\mu}, p) \nu_{N,p}(\xi) \\
			&=  d_{\eta,p}^2 \; \delta_{\eta, \eta'},
	\end{align*}
	for every $\eta, \eta' \in \bigcup\limits_{L=0}^N \mathcal{E}_{K-1,L}$, where $\delta_{\eta, \eta'}$ stands for the Kronecker delta function and 
	\[
		d_{\eta,p}^2 = \left\{
		\begin{array}{ccc}
		\displaystyle \frac{(|\pmb{\alpha}| + N)_{(|\eta|)}}{ (N)_{[|\eta|]} |\pmb{\alpha}|_{(2|\eta|)} } \prod_{j = 1}^{K - 1} \frac{(|\pmb{\alpha}^j| + |\eta^j| + |\eta^{j + 1}| - 1)_{(\eta(j))} (|\pmb{\alpha}^{j+1}| + 2 |\eta^{j+1}|)_{(\eta(j))} \eta(j)! }{ (\alpha_j)_{(\eta(j))} }, & & p > 0 \\
		\displaystyle \frac{1}{ (N)_{[|\eta|]} } \prod_{j = 1}^{K - 1} \frac{ (|\pmb{\pi}^j|)^{\eta(j)} (|\pmb{\pi}^{j+1}|)^{\eta(j)} }{\pi_j^{\eta(j)}} \eta(j)!,&  & p = 0,
		\end{array}
		\right.
	\]
	where $\pmb{\alpha} = N \pmb{\mu}/p$ and $\pmb{\pi} = \pmb{\mu}/|\pmb{\mu}|$.
	\end{itemize}
	 
\end{proposition}
See Theorem 5.4 in \cite{iliev_discrete_2007} and Proposition 2.1, also Remark 2.2, in \cite{khare_rates_2009} for the proofs of these results on multivariate Hahn polynomials. See Theorem 6.2 in \cite{iliev_discrete_2007} and Proposition 2.4 in \cite{khare_rates_2009} for the proofs for the multivariate Krawtchouk polynomials. 
The system of orthogonal polynomials for a fixed multinomial distribution is not unique. A general construction of the multivariate Krawtchouk polynomials can be found in \cite{MR3258404}.

\subsection*{Kernel polynomials for Dirichlet multinomial and multinomial distributions} 

Consider $\nu$ a multivariate distribution on $\mathcal{E}_{K,N}$ and $\{Q^0_\eta\}$ an orthonormal system of polynomials in $l^2(\mathbb{R}^{\mathcal{E}_{K,N}}, \nu)$. Then, the \emph{kernel polynomial} associated to $\nu$ is defined by
\[
	h_n(\mathbf{x}, \mathbf{y}) := \sum\limits_{|\eta| = n} Q_\eta^0(\mathbf{x}) Q_\eta^0(\mathbf{y}),
\]
for all $\mathbf{x}, \mathbf{y} \in \mathcal{E}_{K,N}$ and for every $n \in \left[ N \right]_0$. 
The kernel polynomials are invariant under the choice of the orthonormal systems, i.e.\ they only depend on the distribution $\nu$. 
Kernel polynomials are used for manipulating sums of products of orthogonal polynomials. 
They are especially useful to obtain explicit expressions for the transition function of a reversible Markov chain with polynomial eigenfunctions, as we show in Remark \ref{prop:eigenvectors_of_reversible} below.

We next review the expressions for the kernel polynomials of the Dirichlet multinomial and the multinomial distributions. 
Let us denote by $h_n(\mathbf{x}, \mathbf{y} ; p)$ the $n$-th kernel polynomial of $\nu_{N,p}$, for all $n \in \left[ N \right]_0$. Then, it can be proven that
\begin{equation}\label{eq:kernel_Nek&Nek}
	h_n(N \mathbf{e}_k, N \mathbf{e}_k ; p) = \binom{N}{ n} \frac{(|\pmb{\alpha}| + 2 n - 1) (|\pmb{\alpha}|)_{(n-1)} (|\pmb{\alpha}| - \alpha_k)_{(n)} }{ (|\pmb{\alpha}| + N)_{(n)} (\alpha_k)_{(n)} },
\end{equation}
for all $p > 0$, see \cite[Eq.\ (2.18)]{khare_rates_2009}. 

For $p = 0$, $\nu_{N,0}$ follows a $\mathcal{M}( \cdot \mid N, \pmb{\mu}/|\pmb{\mu}|)$ distribution and its $n$-th kernel polynomial  satisfies
\begin{equation}\label{eq:kernel_poly_Ne_k}
	h_n (\mathbf{x}, N \mathbf{e}_k; 0) = \sum_{m = 0}^n \binom{N}{m} \binom{N - m}{n - m} (-1)^{n- m} \frac{(x_k)_{[m]}}{N_{[m]}} \left( \frac{\mu_k}{|\pmb{\mu}|} \right)^{-m},
\end{equation}
and
\begin{equation}\label{eq:kernel_poly_Ne_k_Ne_k}
	h_n (N \mathbf{e}_k, N \mathbf{e}_k; 0) = \binom{N}{n} \left(\frac{|\pmb{\mu}|}{\mu_k} - 1 \right)^n.
\end{equation}

For more details on the kernel polynomials for the multinomial distribution see e.g.\ \cite[Prop.\ 2.8]{khare_rates_2009} and \cite{MR3915331}.
Also, for more details on the kernel polynomials for the Dirichlet multinomial distribution see e.g.\ \cite[Prop.\ 2.6]{khare_rates_2009} and \cite{MR3037164}.

We next show that the right eigenfunctions of $\mathcal{L}_{N,p}$ are given by multivariate orthogonal polynomials defined by (\ref{def:right_eigenvectors}). 

\begin{remark}[Eigenfunctions of $\mathcal{L}_{N,p}$]\label{prop:eigenvectors_of_reversible}
	The right eigenfunctions of $\mathcal{L}_{N,p}$ are the multivariate polynomials $Q_{\eta}(\cdot\, ; N, \pmb{\mu}, p)$ with associated eigenvalue $\lambda_{L,p}$, for $\eta \in \mathcal{E}_{K-1,L}$, for $L \in \left[ N \right]_0$. Moreover, the set of right eigenfunctions
	\begin{equation*}
		\left\{ Q_{\eta}(\cdot\, ; N, \pmb{\mu}, p), \eta \in \bigcup\limits_{L=0}^N \mathcal{E}_{K-1,L} \right\}
	\end{equation*} 
	is orthogonal in $l^2(\nu_{N,p})$, for all $p \ge 0$. In addition, the functions $\phi_\eta(\cdot\, ; N, \pmb{\mu},p)$ defined by
	\[
		\phi_\eta(\eta' ; N, \pmb{\mu}, p) := \nu_{N,p}(\eta') Q_{\eta}(\eta' ; N, \pmb{\mu}, p)
	\]
	are left eigenfunctions of $\mathcal{L}_{N,p}$ and the set of left eigenfunctions
	is orthogonal in $l^2(1/\nu_{N,p})$.

	Furthermore, the transition kernel of the Markov chain driven by $\mathcal{L}_{N,p}$ can be decomposed as follows:
	\begin{equation}\label{eq:explicit_transition_eta_t}
		 (\mathrm{e}^{t \mathcal{L}_{N,p}} \delta_{\xi})(\eta) = \nu_{N,p}(\xi) \left( 1 + \sum_{L = 1}^N \mathrm{e}^{\lambda_{L,p} t} h_L(\eta, \xi; p) \right),
	\end{equation}
	where $h_L(\eta, \xi; p)$ is the kernel polynomial associated to $\nu_{N,p}$.
\end{remark}

Griffiths and Span\`{o} \cite{MR3037164} give the expression (\ref{eq:explicit_transition_eta_t}) for the transition kernel of the process driven by $\mathcal{L}_{N,p}$, for $p > 0$, as an example of the usefulness of the kernel polynomials for the Dirichlet multinomial distribution. 

The following result provides an explicit expression for the chi-square distance between the distribution of the Markov process driven by $\mathcal{L}_{N,p}$ starting at $N \mathbf{e}_k$ and its stationary distribution at a given time $t$.

\begin{corollary}[Explicit expression for the chi-square distance]\label{thm:convergence_reversible_Chi2}
	For $K \geq 2$, $N \geq 2$ and $p \ge 0$, we obtain the following explicit expression for the chi-square distance between the distribution of the reversible process driven by $\mathcal{L}_{N,p}$ at time $t$
	when the initial distribution is concentrated at $N \mathbf{e}_k$, for $k \in [K]$:
	\begin{equation}\label{eq:kernel_polynomial_Hahn_eN}
		\chi^2_{N \mathbf{e}_k}(t) 
		= 
		\left\{
		\begin{array}{cc}
		\left[ 1 + \mathrm{e}^{- 2 |\pmb{\mu}|  t} \left( \frac{|\pmb{\mu}|}{\mu_k} - 1 \right) \right]^N - 1 & \text{ if } p = 0 \\ 
		\displaystyle \sum_{L = 1}^N \mathrm{e}^{ 2 \lambda_{L,p} t} \binom{N}{L} \frac{(|\pmb{\alpha}| + 2 L - 1) (|\pmb{\alpha}|)_{(L-1)} (|\pmb{\alpha}| - \alpha_k)_{(L)} }{ (|\pmb{\alpha}| + N)_{(L)} (\alpha_k)_{(L)} } & \text{ if } p > 0
		\end{array}
		\right.
	\end{equation}
	
\end{corollary}

\begin{proof}
	Using classical results on reversible Markov chains, see e.g.\ \cite[Eq.\ (2.1)]{khare_rates_2009}, we obtain the following equality for the chi-square distance:
	\[
		\chi^2_{N \mathbf{e}_k}( t ) = \sum_{L = 1}^N \mathrm{e}^{2 \lambda_{L,p} t} h_L(N \mathbf{e}_k, N \mathbf{e}_k; p),
	\]
	where $h(N \mathbf{e}_k, N \mathbf{e}_k; p)$ stands for the kernel polynomials associated to $\nu_{N,p}$, as defined in \eqref{eq:kernel_Nek&Nek} and \eqref{eq:kernel_poly_Ne_k_Ne_k}.
	Thus, the expression for $\chi^2_{N \mathbf{e}_k}( t )$ in \eqref{eq:kernel_polynomial_Hahn_eN} simply comes from \eqref{eq:kernel_Nek&Nek}, when $p > 0$.

	To prove the case when $p = 0$, note that \eqref{eq:kernel_poly_Ne_k_Ne_k} implies
	\begin{align*}
		\chi^2_{N \mathbf{e}_k}(t) = \sum_{L=1}^N \mathrm{e}^{- 2 L |\pmb{\mu}|  t} \binom{N}{L}\left( \frac{|\pmb{\mu}|}{\mu_k} - 1 \right)^L = \left[ 1 + \mathrm{e}^{- 2 |\pmb{\mu}|  t} \left( \frac{|\pmb{\mu}|}{\mu_k} - 1 \right) \right]^N - 1.
	\end{align*}
\end{proof}

We now take advantage of the explicit expression in (\ref{eq:kernel_polynomial_Hahn_eN}) to prove the existence of a strongly optimal cutoff in the chi-square distance for the multi-allelic Moran process with parent independent mutation when $N \rightarrow \infty$.

\begin{proof}[Proof of Theorem \ref{thm:cutoff}]
	Let us first prove the existence of the chi-square cutoff. 
	When $p = 0$, for $\displaystyle t_{N,c} = \frac{\ln N + c}{2 |\pmb{\mu}|}$ we obtain
	\begin{align*}
		\lim\limits_{N \rightarrow \infty} \chi^2_{N \mathbf{e}_k}(t_{N,c}) &= \lim\limits_{N \rightarrow \infty} \left[ 1 + \frac{\mathrm{e}^{-c}}{N} \left( \frac{|\pmb{\mu}|}{\mu_k} - 1 \right) \right]^N - 1\\
			&= \exp\left\{ - \left( \frac{|\pmb{\mu}|}{\mu_k} - 1 \right) \mathrm{e}^{-c} \right\} - 1.
	\end{align*}
	Now, since $K_{k,0} = {|\pmb{\mu}|}/{\mu_k} - 1$, we have proved the existence of the limit (\ref{eq:chi^2cutoff}) for $p = 0$.

	Now, for $p > 0$ let us focus on expression (\ref{eq:kernel_polynomial_Hahn_eN}). For every $L \in \mathbb{N}$ and $k \in \left[ K \right]$, let us denote
	\[
		\phi_{L,k}(N) := \frac{(|\pmb{\alpha}| + 2 L - 1) (|\pmb{\alpha}|)_{(L-1)} (|\pmb{\alpha}| - \alpha_k)_{(L)} }{ (|\pmb{\alpha}| + N)_{(L)} (\alpha_k)_{(L)} }.
	\]
	We thus have
	\begin{align*}
		\phi_{L,k}(N) &:= \frac{|\pmb{\alpha}| + 2 L - 1}{|\pmb{\alpha}| + L - 1} \frac{ \prod\limits_{r = 0}^{L-1} (|\pmb{\alpha}| + r) (|\pmb{\alpha}| - \alpha_k + r)}{ \prod\limits_{r = 0}^{L-1} (|\pmb{\alpha}| + N + r) (\alpha_k + r) }\\
			&= \frac{N|\pmb{\mu}|/p + 2 L - 1}{N|\pmb{\mu}|/p + L - 1} \left[\frac{|\pmb{\mu}|(|\pmb{\mu}|-\mu_k)}{ \mu_k (|\pmb{\mu}| + p)} \right]^L \frac{ \prod\limits_{r = 0}^{L-1} \left(1 +\frac{p}{N|\pmb{\mu}|} r \right) \left(1 + \frac{p}{N(|\pmb{\mu}|-\mu_k)} r \right)}{ \prod\limits_{r = 0}^{L-1} \left(1 + \frac{p}{N (|\pmb{\mu}| + p)}r \right) \left(1 + \frac{p}{N \mu_k} r\right) }.
	\end{align*}
	Hence, for all $L \in \mathbb{N}$ we get
	\[
		\lim\limits_{N \rightarrow \infty} \phi_{L,k}(N) = \left[\frac{|\pmb{\mu}|(|\pmb{\mu}|-\mu_k)}{ \mu_k (|\pmb{\mu}| + p)} \right]^L = (K_{k,p})^L.
	\]
	Moreover,
	\[
		\binom{N}{L} \underrel{N }{\sim} \frac{N^L}{L!} \;\;\text{ and }\;\;
		\mathrm{e}^{2 \lambda_L t_N} \underrel{N }{\sim} \frac{(\mathrm{e}^{-c})^L}{N^L},
	\]
	where for two sequences $(f_N)$ and $(g_N)$ the notation $f_N \underrel{N }{\sim} g_N$ means ${f_N} - {g_N} = o \left(g_N\right)$.
	According to (\ref{eq:kernel_Nek&Nek}) we have
	\[
		h_L(N \mathbf{e}_k, N \mathbf{e}_k; p) = \binom{N}{L} \frac{(|\pmb{\alpha}| + 2 L - 1) (|\pmb{\alpha}|)_{(L-1)} (|\pmb{\alpha}| - \alpha_k)_{(L)} }{ (|\pmb{\alpha}| + N)_{(L)} (\alpha_k)_{(L)} }.
	\]
	Plugging these asymptotic expressions in the $L$-th summand of (\ref{eq:kernel_polynomial_Hahn_eN}) yields
	\begin{equation*}
		\lim\limits_{N \rightarrow \infty} \mathrm{e}^{2 \lambda_L t_N} h_L(N \mathbf{e}_k, N \mathbf{e}_k; p)  = \frac{(K_{k,p} \, \mathrm{e}^c)^L}{L!}.
	\end{equation*}
	Moreover,
	\begin{align*}
		\mathrm{e}^{2 \lambda_L t_N} h_L(N \mathbf{e}_k, N \mathbf{e}_k; p) &\le \mathrm{e}^{-L (c + \ln(N))} h_L(N \mathbf{e}_k, N \mathbf{e}_k; p) \\
			&= \frac{\mathrm{e}^{-c L}}{N^L} \binom{N}{L} \frac{|\pmb{\alpha}| + 2L - 1}{\pmb{\alpha}| + L - 1} \frac{|\pmb{\alpha}|_{(L)} (|\pmb{\alpha}| - \alpha_k)_{(L)}}{(|\alpha| + N)_{(L)} (\alpha_k)_{(L)}}\\
			&= \frac{\mathrm{e}^{-c L}}{L !} \frac{|\pmb{\alpha}| + 2L - 1}{\pmb{\alpha}| + L - 1} \prod_{r = 0}^{L-1} \left[ \frac{N - r}{N} \frac{|\pmb{\alpha}| + r}{|\pmb{\alpha}| + N + r} \frac{|\pmb{\alpha}| - \alpha_k + r}{\alpha_k + r} \right]\\
			&\le  3 \frac{(\gamma \mathrm{e}^{-c})^L}{L !},
	\end{align*}
	where $\gamma = \max\{1, K_{k,0}\}$.

	For an arbitrary small $\epsilon > 0$ let us consider $M \in \mathbb{N}$ such that
	\( \displaystyle
		3 \sum_{L = M+1}^{\infty} \frac{(\gamma \mathrm{e}^{-c})^L}{L!} \le \frac{\epsilon}{3},
	\)
	and let $N_{\epsilon}$ be a positive integer such that
	\[
		\left| \sum_{L = 1}^M \mathrm{e}^{2 \lambda_L t_N} h_L(N \mathbf{e}_k, N \mathbf{e}_k; p) -  \sum_{L = 1}^M \frac{(K_{k,p})^L}{L!} \right| \le \frac{\epsilon}{3},
	\]
	for all $N \ge N_{\epsilon}$.
	Note that
	\[
		\sum_{L = M+1}^{\infty} \frac{(K_{k,p} \mathrm{e}^{-c})^L}{L!} \le \frac{\epsilon}{3}.
	\]
	Then, for all $N \ge N_{\epsilon}$, using the triangular inequality we have
	\[
		\left| \sum_{L = 1}^N \mathrm{e}^{2 \lambda_L t_N} h_{L}(N \mathbf{e}_k) - \left(\exp\{ K_{k,p} \mathrm{e}^{-c} - 1 \} \right)  \right|
			\le \epsilon,
	\]
	which concludes the proof for the chi-square cutoff for the process driven by $\mathcal{L}_{N,p}$, for $p \ge 0$.
	\end{proof}

	Let us establish a result that will be very useful during the proof of Theorem \ref{thm:TVcutoff}. 
\begin{lemma}[Lemma A.2 in \cite{2020arXiv200513437N}]\label{lemma:LemmaA2}
	Let $\psi_N \in (0,1)$, for all $N \in \mathbb{N}$, such that $N \psi_N \rightarrow \infty$, when $N \rightarrow \infty$. Then, for all $y \in \mathbb{R}$ we have
	\[
		\lim\limits_{N \rightarrow \infty} \operatorname{d}^{\mathrm{TV}} \left( \operatorname{Bin}(N,\psi_N), \operatorname{Bin}\left(N, \psi_N + \sqrt{\frac{\psi_N(1 - \psi_N)}{N} } y  \right) \right) = 2 \Phi\left( \frac{1}{2} |y| \right) - 1,
	\]
	where where $\operatorname{Bin}(N,\psi)$ stands for the binomial distribution with $N$ trials and probability of success $\psi$, and $\Phi$ is the cumulative distribution function of the standard normal distribution, i.e.
	\[
		\Phi: t \mapsto \int_{- \infty}^{t} \frac{1}{\sqrt{2 \pi}} \mathrm{e}^{-s^2 / 2} \mathrm{d}s.
	\]
\end{lemma}
This lemma characterises the limit profile of the total variation distance between two random variables $B_1$ and $B_2$, following binomial distributions, when the difference between their means is of the same order of the standard deviation of $B_1/N$. 
The proof can be found in the Appendix A.2 \cite{2020arXiv200513437N}, see also  \cite{2021Nestoridi}, the published version of this paper.

	\begin{proof}[Proof of Theorem \ref{thm:TVcutoff}]

	First note that the lower and upper bounds for $\mathrm{d}^{\mathrm{TV}}_{N \mathbf{e}_k} \left( t_{N,c} \right)$ are simply consequences of Theorems \ref{thm:lower_bound} and \ref{thm:cutoff}, respectively.
	Indeed, for $c < 0$ and using Theorems \ref{thm:lower_bound} we have 
	\[
		\lim\limits_{N \rightarrow \infty} \mathrm{d}^{\mathrm{TV}}_{N \mathbf{e}_k} \left( t_{N,c} \right) \ge 1 - \kappa \frac{\|V\|_{\infty}}{v_k} \mathrm{e}^{-c},
	\]
where $\kappa = 8(2 + \|Q_{\pmb{\mu}}\|_{\infty}  / |\pmb{\mu}|) = 32$, and $V$ is any right eigenvector of $Q_{\pmb{\mu}}$ with eigenvalue $|\pmb{\mu}|$.
Finally, the desired inequality is obtained considering the eigenvector $V = 1/\mu_k \mathbf{e}_k - 1/\mu_s \mathbf{e}_s$, where $s \in [K]$ satisfies $\mu_s \wedge \mu_k = \min\limits_{r: r\neq k} \mu_r \wedge \mu_k$.

In addition, using the classical inequality between the chi-square and the total variation distances and Theorem \ref{thm:cutoff} we get 
\[
	\lim\limits_{N \rightarrow \infty} \mathrm{d}^{\mathrm{TV}}_{N \mathbf{e}_k} \left( t_{N,c} \right) \le \lim\limits_{N \rightarrow \infty} \frac{1}{2}\sqrt{\chi^2_{N \mathbf{e}_k} \left( t_{N,c} \right)} = \frac{1}{2} \sqrt{ {\exp \{ K_{k,p} \mathrm{e}^{-c} \} - 1} }.
\]
This concludes to proof of the existence of the $\displaystyle \left(\frac{\ln N}{2 |\pmb{\mu}|}, 1\right)$ total variation cutoff.

Let us now prove the limit profile for the total variation distance when $p = 0$.
Using (\ref{eq:kernel_poly_Ne_k}) and (\ref{eq:explicit_transition_eta_t}) we get
\begin{align*}
	(\mathrm{e}^{t \mathcal{L}_{N}} \delta_{\xi}) ({N \mathbf{e}_k}) &=  \nu_{N}(\xi) \sum_{L = 0}^N \mathrm{e}^{- |\pmb{\mu}| L t}  \sum_{m=0}^L \binom{N}{m}\binom{N-m}{L - m} (-1)^{L-m} \frac{(\xi_k)_{[m]}}{N_{[m]}} \left( \frac{\mu_k}{|\pmb{\mu}|} \right)^{-m} \\
	&= \nu_{N}(\xi) \sum_{m = 0}^N  \binom{N}{m} \frac{(\xi_k)_{[m]}}{N_{[m]}} \left( \frac{\mu_k}{|\pmb{\mu}|} \right)^{-m} \sum_{L = m}^N \mathrm{e}^{- |\pmb{\mu}| L t} \binom{N-m}{L - m} (-1)^{L-m} \\
	&= \nu_{N}(\xi) \sum_{m = 0}^N  \binom{N}{m} \frac{(\xi_k)_{[m]}}{N_{[m]}} \left( \frac{\mu_k}{|\pmb{\mu}|} \right)^{-m} \mathrm{e}^{- |\pmb{\mu}| m t} (1 - \mathrm{e}^{- |\pmb{\mu}| t})^{N-m}\\
	&= \nu_{N}(\xi) (1 - \mathrm{e}^{- |\pmb{\mu}| t})^N \sum_{m = 0}^{\xi_k} \binom{\xi_k}{m}  \left[ \frac{\mu_k}{|\pmb{\mu}|} \mathrm{e}^{|\pmb{\mu}| t} (1 - \mathrm{e}^{- |\pmb{\mu}| t})   \right]^{-m}\\
	&= \nu_{N}(\xi) (1 - \mathrm{e}^{- |\pmb{\mu}| t})^{N- \xi_k} \left[ (1 - \mathrm{e}^{- |\pmb{\mu}| t}) + \frac{ |\pmb{\mu}| \mathrm{e}^{- |\pmb{\mu}| t}}{\mu_k }   \right]^{\xi_k}.
\end{align*}
Thus, the process driven by $\mathcal{L}_{N}$ starting at $N \mathbf{e}_k$ at time $t$ follows a $\mathcal{M}\left( \cdot \mid N, (1 - \mathrm{e}^{- |\pmb{\mu}| t}) \frac{\pmb{\mu}}{|\pmb{\mu}|} + \mathrm{e}^{- |\pmb{\mu}| t} \mathbf{e}_k \right)$ distribution, which proves Corollary \ref{corol:explicitLaw_t}.
Moreover,
\begin{align*}
	\operatorname{d}^{\mathrm{TV}}_{N \mathbf{e}_k} \left( t \right) &= \frac{1}{2} \sum_{\xi \in \mathcal{E}_{K,N}} \left| (\mathrm{e}^{t \mathcal{L}_{N}} \delta_{\xi})(N \mathbf{e}_k) - \nu_{N}(\xi)\right| \\
		&= \frac{1}{2} \sum_{L = 0}^N \sum_{ \substack{ {\xi \in \mathcal{E}_{K,N}:} \\ {\xi_k = L}}  } \nu_{N}(\xi) \left|(1 - \mathrm{e}^{- |\pmb{\mu}| t})^{N - \xi_k} \left[ 1 - \mathrm{e}^{- |\pmb{\mu}| t} + \frac{ |\pmb{\mu}| \mathrm{e}^{ - |\pmb{\mu}| t} }{\mu_k}   \right]^{\xi_k} - 1\right|\\
		&= \frac{1}{2} \sum_{L = 0}^N \binom{N}{L} \left( \frac{\mu_k}{|\pmb{\mu}|} \right)^L \left( 1 - \frac{\mu_k}{|\pmb{\mu}|} \right)^{N-L} \left|(1 - \mathrm{e}^{- |\pmb{\mu}| t})^{N - L} \left[ 1 - \mathrm{e}^{- |\pmb{\mu}| t} + \frac{ |\pmb{\mu}| \mathrm{e}^{ - |\pmb{\mu}| t} }{\mu_k}    \right]^{L} - 1\right| \\
		&= \operatorname{d}^{\mathrm{TV}} \left( \operatorname{Bin}\left(N , \frac{\mu_k}{|\pmb{\mu}|} \right) , \operatorname{Bin}\left(N, \frac{\mu_k}{|\pmb{\mu}|} (1 - \mathrm{e}^{- |\pmb{\mu}| t}) + \mathrm{e}^{ -|\pmb{\mu}| t} \right) \right).
\end{align*}
Then, we have proved that we can write $\operatorname{d}^{\mathrm{TV}} \left( t, N \mathbf{e}_k \right)$ as the total variation distance between two binomial distributions with parameters $N$ both and probabilities of success $\pi_k = \mu_k/|\pmb{\mu}|$ and $\tilde{\pi}_k = \pi_k(1 - \mathrm{e}^{- |\pmb{\mu}| t}) + \mathrm{e}^{ -|\pmb{\mu}| t}$, respectively. For $t_{N,c} = \frac{\ln N + c}{2 |\mu|}$ we get
\[
	\tilde{\pi}_k = \pi_k + \frac{\sqrt{\pi_k(1- \pi_k)} }{\sqrt{ N}} \sqrt{ \frac{1-\pi_k}{\pi_k} } \mathrm{e}^{-c/2} .
\]
Therefore, using Lemma \ref{lemma:LemmaA2} we obtain
\[
	\lim\limits_{N \rightarrow \infty} \operatorname{d}_{N \mathbf{e}_k}^{\mathrm{TV}} \left( t \right) = 2 \Phi\left( \frac{1}{2} \sqrt{  K_{k, 0} \mathrm{e}^{-c} } \right) - 1,
\]
where $K_{k, 0} = \frac{1-\pi_k}{\pi_k} = \frac{|\pmb{\mu}|}{\mu_k} - 1$.

\end{proof}

\appendix

\section{Proofs of Lemmas \ref{thm:interp} and \ref{thm:chiV}, and Proposition \ref{propo:basis}}\label{appendix:two_model}

This section is devoted to the proofs of Lemmas \ref{thm:interp} and \ref{thm:chiV}, and Proposition \ref{propo:basis}.

\begin{proof}[Proof of Lemma \ref{thm:interp}]
(a) Let us first prove that for any $\alpha \in \mathcal{E}_{K,N}$, there exists a unique polynomial
$P_{\alpha} \in H_{K,N}$, product of $N$ linear functions on $H_{K,1}$, such that
$P_{\alpha}(\eta) = 1$ if $\eta=\alpha$ and $0$ otherwise.
Indeed, let us define the polynomial $P_{\alpha}$ by
\begin{equation*} 
	P_{\alpha}: \mathbf{x} \in \mathcal{E}_{K,N} \mapsto \displaystyle \prod_{k=1}^K \prod_{a=0}^{\alpha_k-1} \frac{x_k - a}{\alpha_k - a},
\end{equation*}
where $\prod_{a=0}^{\alpha_k-1} (x_k - a) = 1$ when
$\alpha_k=0$.
Note that $P_{\alpha} = \mathbf{1}_{\alpha}$, for every $\alpha \in \mathcal{E}_{K,N}$. 
There are $\sum_{k=1}^K \alpha_k = N$ linear factors in the numerator. Also,
each term $x_k - a$ may be replaced by $x_k - \frac{a}{N} \sum_{k=1}^{K} x_k$ when $a \neq 0$,
so $P_{\alpha}(x)$ may be considered as a product of $N$ linear functions on $H_{K,1}$,
and because the uniqueness of such a function $P_{\alpha}$ is straightforward, (a) is proved.

Now, for every real function $f$ on $\mathcal{E}_{K,N}$, the result is immediately obtained  from (a) by setting
\[
	P := \sum_{\alpha \in \mathcal{E}_{K,N}} f(\alpha) P_{\alpha}.
\]

(b) From part (b) we have that $\mathcal{B}_{H_{K,N}}$ is a generator system of $\mathbb{R}^{\mathcal{E}_{K,N}}$. Moreover, 
\[
	\operatorname{Card} (\mathcal{B}_{H_{K,N}}) = \operatorname{Card} (\mathcal{E}_{K,N}) = \dim (\mathbb{R}^{\mathcal{E}_{K,N}}) = \binom{K - 1 + N}{N},
\] 
thus $\mathcal{B}_{H_{K,N}}$ is necessarily a basis of $\mathbb{R}^{\mathcal{E}_{K,N}}$.
\end{proof}

\begin{proof}[Proof of Lemma \ref{thm:chiV}]
(a) For $L=1$: An injection $s : \{1\} \rightarrow \{1,2,\dots,N\}$ is characterised by 
$s(1)=i$. It follows from (\ref{def:xi}) that
$$
\xi(V_1)(k_1,k_2,\dots,k_N) = \sum_{i = 1}^N V_1(k_i),
$$
which is a symmetric function.
For every $\eta = (\eta(1),\eta(2),\dots,\eta(K)) \in \mathcal{E}_{K,N}$,
we have
\[
	\tilde{\xi}(V_1)(\eta) = \big( \xi(V_1) \circ \psi_{K,N} \big) (\eta) = \sum_{j=1}^K V_1(j) \eta(j),
\]
which finishes the proof of part (a).

(b) From (\ref{def:xi}), we get
\begin{align}
\xi(V_1,V_2,\dots,V_L) (k_1,\dots,k_N) &= \sum_{s \in \mathcal{I}_{L-1,N}}
V_1(k_{s(1)}) \dots V_{L-1}(k_{s(L-1)})
\sum_{i \in \left[ N \right] \setminus s(\left[ L-1 \right]) } V_L(k_i) 
  \nonumber \\
&= \sum_{s \in \mathcal{I}_{L-1,N}}
V_1(k_{s(1)}) \dots V_{L-1}(k_{s(L-1)})
\left( \sum_{i=1}^N V_L(k_i) - \sum_{i=1}^{L-1} V_L(k_{s(i)}) \right)
 \nonumber \\
&= \xi(V_1,V_2,\dots,V_{L-1})(k_1,\dots,k_N) \xi(V_L)(k_1,\dots,k_N) \nonumber \\
& \;\;\;\; - \sum_{i=1}^{L-1} \xi(V_1,\dots,V_{i} \odot V_{L},\dots, V_{L-1})(k_1,\dots,k_N). \nonumber
\end{align}
Using (\ref{carac:propchi}) we obtain the result for $\tilde{\xi}(V_1,V_2,\dots,V_L)$.
The particular case $L=2$ comes from part (a).

(c) We can prove equation (\ref{eq:prodxi}) by induction on $L$.
For $L=1$ the result easily comes by (a). If we suppose that (\ref{eq:prodxi}) is satisfied for $L$, for $2 \leq L < N-1$, then, using (b) and (a), we can check that
(\ref{eq:prodxi}) holds for $L+1$.
\end{proof}

\begin{proof}[Proof of Proposition \ref{propo:basis}]
	
	Since $\mathcal{U}$ is a basis of $\mathbb{R}^K$ we trivially have that $\mathcal{U}^N$ is a basis of $\mathbb{R}^{\left[ K \right]^N}$, proving (a) (cf.\ Lemma 12.12 in \cite{levin_markov_2017}). 
	To prove (b) we prove that each element of $\mathcal{U}^N$ has image in $\mathcal{S}^N$ by $\operatorname{Sym}$, defined as in (\ref{def:sim_operator}). 
	First, $\operatorname{Sym}(U_0 \otimes \dots \otimes U_0) = U_0 \otimes \dots \otimes U_0$, since the constant function equal to one is symmetric. 
	Furthermore, for every $W = W_1 \otimes W_2 \otimes \dots \otimes W_N \in \mathcal{U}^N$ there is a permutation $\sigma \in \mathcal{S}_N$ such that $\sigma W = U_{\eta}$, with $\eta \in \mathcal{E}_{K-1,L}$, where $L \in \left[ N \right]$ is the number of components in the expression of $W$ different from $U_0$. 
	Thus, $\operatorname{Sym}(W) = \operatorname{Sym}(\sigma W) = V_{\eta}$, for $\eta \in \mathcal{E}_{K-1,L}$. 
	We have not proved that $V_{\eta} \neq V_{\alpha}$, for $\eta \neq \alpha$. 
	However, $\mathcal{S}^N$ is a generator system of $\operatorname{Sym}(\mathbb{R}^{\left[ K \right]^N})$ satisfying
		\begin{align*}
			\operatorname{Card}\left(\mathcal{S}^N\right) &\le 1 + \sum_{L = 1}^{N} \operatorname{Card}(\mathcal{E}_{K-1,L})\\
				&= \sum_{L = 0}^{N} \binom{K -2 + L }{L} = \binom{K - 1 + N }{N},
		\end{align*}
	where the last equality is the well-known \emph{Hockey\,--\,Stick identity} in combinatorics, see e.g.\ \cite{MR1952453}.
	Now, since 
	\[
	\dim\left(\operatorname{Sym}\left(\mathbb{R}^{{\left[ K \right]}^N}\right)\right) = \binom{K - 1 + N }{N},
	\] 
	we have that $\mathcal{S}^N$ is a generator system with a minimal number of vectors, therefore it is a basis of $\operatorname{Sym}(\mathbb{R}^{\left[ K \right]^N})$. To prove $(c)$ simply note that each element in $\tilde{\mathcal{S}}^N$ is the image by the isomorphism $\Phi_{K,N}$ of an element in $\mathcal{S}^N$. 
\end{proof}

\section{Proof of Lemma \ref{lemma:specA}}\label{section:proof_lemma_specA}

\begin{proof}[Proof of Lemma \ref{lemma:specA}]
Without lost of generality we can only prove the result for the monomials on  $\mathcal{E}_{K,N}$.
Consider $m$ a monomial on $\mathcal{E}_{K,N}$ of total degree
$|\alpha| = L$ with $0\leq L \leq N$.
Then, we want to prove that
\begin{equation*}
\mathcal{A}_{N} V_m = - L(L-1) V_m + V_q,
\end{equation*}
where $q$ is a polynomial with a total degree strictly less than $L$.

As we commented in Remark \ref{remark:1and2degree_eigenAKN}, the result is true for $L = 1$. Let us assume $L \ge 2$ and consider the monomial $m: \eta \mapsto \prod_{r=1}^K \eta(r)^{\alpha_r}$. 
Evaluating $V_m$ in $\mathcal{A}_{N}$, defined by (\ref{def:reproduction}), we obtain
\begin{equation} \label{eq:AVp}
(\mathcal{A}_{N} V_m)(\eta) = \sum_{k,r: k\neq r } \left( \prod_{s \notin \{k, r\}  } \eta(s)^{\alpha_s}\right)
\left[ (\eta(k) -1)^{\alpha_k} (\eta(r)+1)^{\alpha_r} - \eta(k)^{\alpha_k} \eta(r)^{\alpha_r}
\right] \eta(k) \eta(r),
\end{equation}
for all $\eta \in \mathcal{E}_{K,N}$.
Then, from the Newton's binomial formula, we get
\begin{equation*} 
\eta(k) (\eta(k) -1)^{\alpha_k} = \eta(k)^{\alpha_k+1} - \alpha_k \eta(k)^{\alpha_k}
+ \frac{\alpha_k (\alpha_k-1)}{2} \eta(k)^{\alpha_k-1} + a(\eta(k)), 
\end{equation*}
where $a(\eta(k))$ is a polynomial in $\eta(k)$ with degree strictly less than $\alpha_k -1$
if $\alpha_k \geq 2$ and null otherwise.
In the same way, we get
\begin{equation*} 
\eta(r) (\eta(r) +1)^{\alpha_r} = \eta(r)^{\alpha_r+1} + \alpha_r \eta(r)^{\alpha_r}
+ \frac{\alpha_r (\alpha_r-1)}{2} \eta(r)^{\alpha_r-1} + b(\eta(r)), 
\end{equation*}
where $b(\eta(r))$ is a polynomial in $\eta(r)$ with degree strictly less than $\alpha_r -1$
if $\alpha_r \geq 2$ and null otherwise.

Using this expansion in (\ref{eq:AVp}) and regrouping
terms with total degree in $\eta(k)$
and $\eta(r)$ strictly less than $\alpha_k + \alpha_r$ give
\begin{align}
(\mathcal{A}_{N} V_m)(\eta) &= \sum_{k,r: k \neq r} \left( \prod_{s \notin \{k, r\}  } \eta(s)^{\alpha_s}\right)
 ( \alpha_r \eta(k)^{\alpha_k+1} \eta(r)^{\alpha_r}
 - \alpha_k \eta(k)^{\alpha_k} \eta(r)^{\alpha_r+1} )  \nonumber \\
 & \;\;\;\; + \sum_{k,r: k \neq r} \left( \prod_{s \notin \{k, r\}  } \eta(s)^{\alpha_s} \right)
   \frac{\alpha_r (\alpha_r-1)}{2}  \eta(k)^{\alpha_k+1}  \eta(r)^{\alpha_r-1} \nonumber  \\
 & \;\;\;\; - \sum_{k,r: k\neq r} \left( \prod_{s \notin  \{k,  r\}} \eta(s)^{\alpha_s} \right) \alpha_k \alpha_r 
   \eta(k)^{\alpha_k} \eta(r)^{\alpha_r} \label{eq:AVp2} \\
 & \;\;\;\; + \sum_{k,r: k \neq r} \left( \prod_{s \notin \{k, r\}}\eta(k)^{ \alpha_s} \right)
   \frac{\alpha_k (\alpha_k-1)}{2}  \eta(k)^{\alpha_k-1}  \eta(r)^{\alpha_r+1}  + w(\eta),\nonumber 
\end{align}
where $w$ is a polynomial in $\eta$ of total
degree strictly less than $\sum_k \alpha_k = L$.
The first sum in the right member of (\ref{eq:AVp2}) is null because
the antisymmetry in $k,r$ of its summands.
The third term is
\begin{equation*}
- \sum_{k,r: k \neq r} \left( \prod_{s \notin \{k,r\}} \eta(s)^{\alpha_s} \right) \alpha_k \alpha_r 
   \eta(k)^{\alpha_k} \eta(r)^{\alpha_r} = - c_1 p(\eta),
\end{equation*}
with 
\begin{equation*}
c_1 = \sum_{k,r: k \neq r} \alpha_k \alpha_r = \left(\sum_{k = 1}^K \alpha_k\right)^2 - \sum_{k = 1}^K \alpha_k^2 = L^2 - \sum_{k = 1}^K \alpha_k^2.
\end{equation*}

By symmetry in $k$ and $r$, it is obvious that the second and the fourth sums in the right member
of (\ref{eq:AVp2}) are equal.
Using
\begin{equation*}
\alpha_r (\alpha_r -1) \eta(r)^{\alpha_r-1} = \eta(r) \frac{\partial^2}{\partial \eta(r)^2} \eta(r)^{\alpha_r},
\end{equation*}
it follows that
\begin{align*}
\sum_{ k \neq r} \left( \prod_{ s \notin \{k,r\} } \eta(s)^{\alpha_s} \right)
   \alpha_r (\alpha_r-1)  &\eta(k)^{\alpha_k+1}  \eta(r)^{\alpha_r-1} =
   \sum_{k,r: k \neq r} \eta(k) \eta(r) \frac{\partial^2}{\partial \eta(r)^2} m(\eta) \\
    &= \sum_{k,r} \eta(k) \eta(r) \frac{\partial^2}{\partial \eta(r)^2} m(\eta)
    - \sum_{ r =1}^K \eta(r)^2 \frac{\partial^2}{\partial \eta(r)^2} m(\eta) \\
    &= N  \sum_{r=1}^K \eta(r) \frac{\partial^2}{\partial \eta(r)^2} m(\eta) 
  - \sum_{r=1}^K \eta(r)^2 \frac{\partial^2}{\partial \eta(r)^2} m(\eta).
\end{align*}
The first summand in the last equality is an homogeneous polynomial of degree $L-1$ and the second one satisfies
\begin{equation*}
- \sum_{r=1}^K \eta(r)^2 \frac{\partial^2}{\partial \eta(r)^2} m(\eta) = - c_2 \; m(\eta),
\end{equation*}
with
\begin{equation*}
c_2 = \sum_{r=1}^K \alpha_r (\alpha_r -1) = \sum_{r=1}^K \alpha_r^2 - L.
\end{equation*}
As a conclusion, it comes from (\ref{eq:AVp2}) that
\begin{equation*}
\mathcal{A}_{N} V_m = -(c_1+c_2) V_m + V_q = -L(L-1) V_m + V_q,
\end{equation*}
where $q$ is a polynomial of total degree strictly less than $L$, which proves (a).


\end{proof}

\section{Proof of Lemma \ref{thm_reversible_distrib}}\label{appendix:poof_reversibility}

First we prove Lemma \ref{lemma:rever_complete} showing that the neutral multi-allelic Moran process driven by $\mathcal{Q}_{N,p}$ is reversible if and only if its mutation rate matrix can be written in the form of $Q_{\pmb{\mu}}$, given by (\ref{eq:def:Qrev_matrix}). 
We start by proving that when the neutral multi-allelic Moran process is reversible, then all the entries of the mutation matrix are positive and it can be written in the form of $Q_{\pmb{\mu}}$, i.e.\ the \emph{``only if part''}. 
Later, in Lemma \ref{thm_reversible_distrib_gnral} we prove that the process driven by $\mathcal{L}_{N,p}$ is reversible and we provide the explicit expression for its stationary distribution, i.e.\ we prove the \emph{``if part''}. Actually, the results in Lemma \ref{thm_reversible_distrib_gnral} are proved for a more general Moral model with \emph{selection at birth}.

\begin{lemma}\label{lemma:rever_complete}
If the process driven by the generator (\ref{def:generatorQKNp}) is reversible, then $\mu_{i,j} = \mu_j > 0$, for all $i \in \left[ K \right]$, and every $j \in \left[ K \right]$, $j \neq i$.
\end{lemma}

\begin{proof}
We first prove that if the process is reversible, then all the entries of the mutation matrix are positive. Let us denote by $\nu_{N,p}$ the stationary probability measure of the  process driven by $\mathcal{Q}_{N,p}$, which is assumed to be reversible. 
We denote $\mathcal{Q}_{N,p}[\eta, \xi] := (\mathcal{Q}_{N,p} \delta_{\xi}) (\eta)$, for all $\eta, \xi \in \mathcal{E}_{K,N}$.
Consider the states $\eta^{(1)}$ and $\eta^{(2)}$ defined as $\eta^{(1)} := N \mathbf{e}_i$ and $\eta^{(2)} := \eta^{(1)} - \mathbf{e}_i + \mathbf{e}_j $, 
for $i,j \in \left[ K \right]$ such that $i \neq j$. Since the process is reversible, the measure $\nu_{N}$ satisfies the balance equation 
\[
\nu_{N,p}(\eta^{(1)}) \mathcal{Q}_{N,p}[\eta^{(1)}, \eta^{(2)}] = \nu_{N,p}(\eta^{(2)}) \mathcal{Q}_{N,p}[\eta^{(2)}, \eta^{(1)}],
\]
see e.g.\ \cite[Thm\. 1.3]{kelly_reversibility_1979}.
We have 
\begin{align*}
	\mathcal{Q}_{N,p}[\eta^{(1)}, \eta^{(2)}] &= N \mu_{i,j}, \text{ and }\\
	\mathcal{Q}_{N,p}[\eta^{(2)}, \eta^{(1)}] &= \mu_{j,i} + p (N-1)/N > 0.
\end{align*}
Furthermore, since the process is irreducible we have that $\nu_{N}(\eta) > 0$, for all $\eta \in \mathcal{E}_{K,N}$. 
Finally, the balance equation implies that $\mu_{i,j} > 0$, for all $i \neq j$.

Now, we prove that for every $j \in \left[ K \right]$ we have $\mu_{i,j} = \mu_j > 0$, for all $i \in \left[ K \right]$. For $K = 2$, there is nothing to prove.
For $K \geq 3, \; N \geq 2$, let us consider a general model with a reversible stationary probability.
Let $i,j,k$ be three different indices on $\left[ K \right]$ and consider the four states $\eta^{(1)}$, $\eta^{(2)}$, $\eta^{(3)}$ and $\eta^{(4)}$ in $\mathcal{E}_{K,N}$ defined by
\[
\eta^{(1)} := N \mathbf{e}_i, \;\;\; 
\eta^{(2)} := \eta^{(1)} - \mathbf{e}_i + \mathbf{e}_j, \;\;\; 
\eta^{(3)} := \eta^{(1)}  - 2\, \mathbf{e}_i + \mathbf{e}_j + \mathbf{e}_k, \;\;\;
\eta^{(4)} :=  \eta^{(1)}  - \mathbf{e}_i + \mathbf{e}_k.
\]
Note that
\begin{equation*}
\begin{array}{ll}
\mathcal{Q}_{N,p}[\eta^{(1)},\eta^{(2)}] = N \mu_{i,j},         & \mathcal{Q}_{N,p}[\eta^{(2)},\eta^{(1)}] = \mu_{j,i} + (N-1)p/N,  \\
\mathcal{Q}_{N,p}[\eta^{(2)},\eta^{(3)}] = (N-1) \mu_{i,k},     & \mathcal{Q}_{N,p}[\eta^{(3)},\eta^{(2)}] = \mu_{k,i} + (N-2)p/N, \\
\mathcal{Q}_{N,p}[\eta^{(3)},\eta^{(4)}] = \mu_{j,i} + (N-2)p/N,  & \mathcal{Q}_{N,p}[\eta^{(4)},\eta^{(3)}] = (N-1) \mu_{i,j},   \\
\mathcal{Q}_{N,p}[\eta^{(4)},\eta^{(1)}] = \mu_{k,i} + (N-1)p/N,  & \mathcal{Q}_{N,p}[\eta^{(1)},\eta^{(4)}] = N \mu_{i,k}.
\end{array}
\end{equation*}
Then,
\begin{align*}
\frac{  \mathcal{Q}_{N,p}[\eta^{(4)},\eta^{(1)}]}{N (N-1)} \prod_{r = 1}^3  \mathcal{Q}_{N,p}[\eta^{(r)},\eta^{(r+1)}]  &=
\mu_{i,j} \mu_{i,k} \left(\mu_{j,i} + p \frac{N-2}{N} \right)\left(\mu_{k,i} + p \frac{N-1}{N} \right), \\
\frac{\mathcal{Q}_{N,p}[\eta^{(1)},\eta^{(4)}] }{N (N-1)} \prod_{r = 1}^3  \mathcal{Q}_{N,p}[\eta^{(r+1)},\eta^{(r)}] &= \mu_{i,k} \mu_{i,j}\left(\mu_{j,i} + p \frac{N-1}{N} \right)\left(\mu_{k,i} + p \frac{N-2}{N} \right).
\end{align*}
Therefore, since the stationary probability is reversible, the \emph{Kolmogorov cycle reversibility criterion} \cite[Thm.\ 1.8]{kelly_reversibility_1979} holds: 
\[
	 \mathcal{Q}_{N,p}[\eta^{(4)},\eta^{(1)}] \prod_{r = 1}^3  \mathcal{Q}_{N,p}[\eta^{(r)},\eta^{(r+1)}]  = \mathcal{Q}_{N,p}[\eta^{(1)},\eta^{(4)}] \prod_{r = 1}^3  \mathcal{Q}_{N,p}[\eta^{(r+1)},\eta^{(r)}],
\]
and we get $p (N-1) \mu_{i,j} \mu_{i,k} (\mu_{j,i} - \mu_{k,i}) = 0$. We know that $\mu_{i,j} > 0$ for all $i,j \in \left[ K \right]$, thus $\mu_{j,i} = \mu_{k,i}$, for all $j,k \in \left[ K \right]$, with $j \neq k$, and every $i \in \left[ K \right]$, with $i \notin \{j,k\}$. Denoting $\mu_j := \mu_{i,j}$ for any $i \in \left[ K \right]$, with $i \neq j$, we prove that the mutation matrix is of the form of $Q_{\pmb{\mu}}$ for a suitable vector $\pmb{\mu}$.

\end{proof}

It remains to prove that the stationary distribution of $\mathcal{L}_{N,p}$ is compound Dirichlet multinomial with suitable parameters. 
Actually, a more general version of Lemma \ref{thm_reversible_distrib} can be proved, where the values of the parameter $p$ in (\ref{def:generatorQKNp_reversible}) also depend on $j$, i.e.\ a model with selection at birth or fecundity selection \cite{muirhead_modeling_2009}. Abusing notation, for two vectors $\pmb{p} = (p_1, p_2, \dots, p_K)$ and $\pmb{\mu} = (\mu_1, \mu_2, \dots, \mu_K)$ such that $p_j, \mu_j > 0$, for all $j \in [K]$, let us denote by $\mathcal{L}_{N,\pmb{p}}$ the infinitesimal generator satisfying
\begin{equation}\label{def:generatorQKNp_reversible_gnral}
	(\mathcal{L}_{N,\pmb{p}} f) (\eta) := \sum_{i,j=1}^{K} \eta(i)
\left( \mu_j + p_j \frac{\eta(j)}{N} \right)
\left[ f(  \eta - \mathbf{e}_i + \mathbf{e}_j ) - f(\eta)\right],
\end{equation}
for every function $f$ on $\mathcal{E}_{K,N}$ and all $\eta \in \mathcal{E}_{K,N}$.
We define the \emph{weighted Dirichlet-compound multinomial distribution} with parameters $N$, $\pmb{\mu}$ and $\pmb{p}$, denoted $\mathcal{WDM}(\cdot \mid N,\pmb{\mu},\pmb{p})$, as follows
\begin{equation}\label{def:weighted_compound_Dirichlet_multinomial}
	\mathcal{WDM}(\eta \mid N,\pmb{\mu},\pmb{p}) := Z^{-1} \binom{N}{\eta} \prod_{k=1}^K p_k^{\eta(k)} (\alpha_k)_{(\eta(k))},
\end{equation}
for all $\eta \in \mathcal{E}_{K,N}$, where $\alpha_k = \mu_k/p_k$, for all $k \in \left[ K \right]$ and $Z$ is a normalisation constant satisfying
\begin{equation}\label{eq:normalisation_constant_weigthed_CMD}
	Z = \mathbb{E}\left[ \left( \sum_{j=1}^K p_j X_j \right)^N \right],
\end{equation}
where $(X_1,X_2, \dots, X_K)$ follows a $ \mathcal{DM}( \cdot \mid N, N \pmb{\mu})$. Note that the measure defined by (\ref{def:weighted_compound_Dirichlet_multinomial}) with the normalisation constant (\ref{eq:normalisation_constant_weigthed_CMD}) is a probability distribution. 
See \cite{johnson_discrete_1997} and \cite{navarro_multivariate_2006} for more details about the weighted multinomial distributions.

\begin{lemma}[Reversible probability of $\mathcal{L}_{N,\pmb{p}}$]\label{thm_reversible_distrib_gnral}
The process driven by (\ref{def:generatorQKNp_reversible_gnral}) is reversible and its stationary distribution is $\mathcal{WDM}(\cdot \mid N, \pmb{\alpha}, \pmb{p})$, where $\alpha_k = N \mu_k$, for all $k \in \left[ K \right]$.
\end{lemma}

\begin{remark}
	This result is known for multi-allelic Moran models with parent independent mutation. See e.g.\ \cite[Section 3]{etheridge_coalescent_2009}. However, we have not found a proof in the literature. So, for the sake of completeness we provide a proof. When the vector $\pmb{p}$ is constant we obtain the stationary distribution of the neutral case and we thus conclude the proof of Lemma \ref{thm_reversible_distrib}.
\end{remark}

\begin{proof}[Proof of Lemma \ref{thm_reversible_distrib_gnral}]
Let us define $q_k := p_k / N$, for $k \in \left[ K \right]$ and, abusing notation, $\mathcal{L}_{N,\pmb{p}}[\eta, \xi] := \mathcal{L}_{N,\pmb{p}} \delta_{\xi}(\eta)$, for all $\eta, \xi \in \mathcal{E}_{K,N}$.
Note that for $\eta, \xi \in \mathcal{E}_{K,N}$ with $\eta \neq \xi$, we have $\mathcal{L}_{N,\pmb{p}}[\eta,\xi] \neq 0$ if and only if there exist $i,j \in \left[ K \right]$, such that
$i \neq j$, $\eta(i) >0$ and $\xi = \eta  - \mathbf{e}_i + \mathbf{e}_j$. In this case
\[
	\mathcal{L}_{N,\pmb{p}}[\eta,\xi] = \eta(i) [\mu_j + \eta(j) q_j].
\] 
This implies that $\xi(j)=\eta(j)+1 >0$ and $\eta = \xi  - \mathbf{e}_j + \mathbf{e}_i$. As a consequence 
\[
	\mathcal{L}_{N,\pmb{p}}[\xi,\eta] = \xi(j) [\mu_i + \xi(i) q_i] = (\eta(j) + 1) [\mu_i + (\eta(i)-1) q_i].
\] 
Also $\eta(k)=\xi(k)$, for all $k\neq i, \; k\neq j$.

Therefore we get,
\begin{align}
Z \cdot \mathcal{WDM}(\eta \mid N,\pmb{\mu},\pmb{p}) & \mathcal{L}_{N,\pmb{p}}[\eta,\xi]  =
\binom{  N }{ \eta } \left[ \prod_{k=1}^K p_k^{\eta(k)} \left( \frac{\mu_k}{q_k} \right)_{(\eta(k))} \right]
\eta(i) [\mu_j + \eta(j) q_j] \nonumber  \\
&=
\frac{N!}{\prod\limits_{k \notin \{i, j\}} \eta(k)!} \; \frac{1}{\eta(i)! \eta(j) !} \; \left[
\prod_{k =1 }^K \prod_{l=0}^{\eta(k) -1} (\mu_k + l \, q_k) \right] \eta(i) [\mu_j + \eta(j) \, q_j], \label{eq:demnu}
\end{align}
where $Z$ is the normalisation constant given by (\ref{eq:normalisation_constant_weigthed_CMD}). Note that
\begin{equation}
\frac{N!}{\prod\limits_{k \notin \{i,j\} } \eta(k)!} \prod_{k \notin \{i,j\}}^K \prod_{l=0}^{\eta(k) -1} (\mu_k + l \, q_k) 
=
\frac{N!}{\prod\limits_{k \notin \{i,j\}} \xi(k)!} \prod_{k \notin \{i,j\}}^K \prod_{l=0}^{\xi(k) -1} (\mu_k + l \, q_k), \label{eq:nu1}
\end{equation}
because $\eta(k)=\xi(k)$, for  $k \notin \{ i, j\}$. 
Moreover,
\begin{equation}
 \frac{1}{\eta(i)! \, \eta(j) !} \eta(i) = \frac{1}{(\eta(i)-1)! \, \eta(j) !}
 = \frac{1}{\xi(i)! \, (\xi(j) - 1) !}
 = \frac{1}{\xi(i)! \, \xi(j) !} \; \xi(j), \label{eq:nu2}
\end{equation}
because $\xi(i) = \eta(i)-1$ and $\xi(j) = \eta(j) +1$.
In addition,
\begin{equation}
\prod_{l=0}^{\eta(i) -1} (\mu_i + l \, q_i) = \prod_{l=0}^{\xi(i)} (\mu_i + l \, q_i) = (\mu_i + \xi(i) \, q_i) \prod_{l=0}^{\xi(i)-1} (\mu_i + l \, q_i),
\label{eq:nu3}
\end{equation}
and 
\begin{equation}
\left[ \prod_{l=0}^{\eta(j) -1} (\mu_j + l \, q_j) \right] [\mu_j + \eta(j) \, q_j] = \prod_{l=0}^{\eta(j)} (\mu_j + l \, q_j)
=  \prod_{l=0}^{\xi(j)-1} (\mu_j + l \, q_j).\label{eq:nu4}
\end{equation}

Using (\ref{eq:nu1}), (\ref{eq:nu2}), (\ref{eq:nu3}) and (\ref{eq:nu4}) in (\ref{eq:demnu}) gives
\begin{align*}
	\mathcal{WDM}(\eta \mid N,\pmb{\mu},\pmb{p})  \mathcal{L}_{N,\pmb{p}}[\eta,\xi] &= Z^{-1} \binom{N}{\xi} \left[ \prod_{k = 1}^K p_k^{\xi(k)} \left( \frac{\mu_k}{p_k} \right)_{(\xi(k))} \right] \xi(j) [\mu_i + \xi(i) q_i ] \\
		&=  \mathcal{WDM}(\xi \mid N,\pmb{\mu},\pmb{p})  \mathcal{L}_{N,\pmb{p}}[\xi,\eta],
\end{align*}
for all $\eta, \xi \in \mathcal{E}_{K,N}$.
The distribution $\nu_{N}$ satisfies the detailed balance property, thus it is reversible for $\mathcal{L}_{N,\pmb{p}}$,  and it is the unique stationary measure, because the process generated by $\mathcal{L}_{N,\pmb{p}}$ is irreducible.

\end{proof}

\section*{Acknowledgement}

The research presented in this paper was mostly conducted while the author was a PhD student at Université Paris-Dauphine and the Institut de Mathématiques de Toulouse.
The author was partially supported by the ITI IRMIA++.
The author would like to thank Djalil Chafa\"{i}, Simona Grusea and Didier Pinchon for their encouragement and many fruitful discussions on this research. 
The author would also like to extend his gratitude to Vlada Limic, Denis Villemonais, Hua Zhou and the two anonymous reviewers for their careful reading and helpful comments, that greatly improved the quality of this manuscript.


\end{document}